\documentclass[oneside,reqno]{amsart}
\usepackage{amsfonts,amsmath,latexsym,verbatim,amscd,mathrsfs,color,array}
\usepackage[colorlinks=true,citecolor=blue]{hyperref}

\usepackage{amsmath,amssymb,amsthm,amsfonts,graphicx,color}
\usepackage{amssymb}
\usepackage{amsbsy}
\usepackage{epstopdf}

\newcommand{\ass}{\quad\mbox{as}\quad}

\newcommand{\bpsi}{ \tilde{\psi}  }
\newcommand{\bphi}{ \tilde{\phi}  }
\newcommand{\EE}{{\mathcal E}  }

\newcommand{\inn}{{\quad\hbox{in } }}
\newcommand{\onn}{{\quad\hbox{on } }}
\newcommand{\ttt}{\tilde }

\newcommand{\TT}{{\mathcal T}  }

\newcommand{\QQ}{{\mathcal Q}  }
\newcommand{\nn}{ {\nabla}  }

\newcommand{\pp}{ {\partial} }

\newcommand{\vp}{\varphi}

\newcommand{\OO}{{\mathcal O}}

\newcommand{\RR}{{{\mathcal R}}}

\newcommand{\R} {\mathbb R}
\newcommand{\Z} {\mathbb Z}
\newcommand{\cuad}{{\sqcap\kern-.68em\sqcup}}

\newcommand{\foral}{\quad\mbox{for all}\quad}
\newcommand{\ve}{\varepsilon}

\newcommand{\be}{\begin{equation}}
\newcommand{\ee}{\end{equation}}

\newcommand{\la}{\lambda}

\newcommand{\equ}[1]{(\ref{#1})}

\newtheorem{lemma}{Lemma}[section]
\newtheorem{prop}{Proposition}[section]
\newtheorem{theorem}{Theorem}
\newtheorem{corollary}{Corollary}[section]
\newtheorem{remark}{Remark}[section]
\newcommand{\bremark}{\begin{remark} \em}
\newcommand{\eremark}{\end{remark} }

\numberwithin{equation}{section}

\title[Vortex dynamics in Euler flows]
{Gluing methods for vortex dynamics in Euler flows}
\author[J. Davila]{Juan Davila}
\address{\noindent  Departamento de
Ingenier\'{\i}a  Matem\'atica-CMM   Universidad de Chile,
Santiago 837-0456, Chile}
\email{jdavila@dim.uchile.cl}

\author[M. del Pino]{Manuel del Pino}
\address{\noindent   Department of Mathematical Sciences University of Bath,
Bath BA2 7AY, United Kingdom \\
and  Departamento de
Ingenier\'{\i}a  Matem\'atica-CMM   Universidad de Chile,
Santiago 837-0456, Chile}
\email{m.delpino@bath.ac.uk}

\author[M. del Pino]{Monica Musso}
\address{\noindent   Department of Mathematical Sciences University of Bath,
Bath BA2 7AY, United Kingdom \\
and Departamento de Matem\'aticas, Universidad Cat\'olica de Chile, Macul 782-0436, Chile}
\email{m.musso@bath.ac.uk}

\author[J. Wei]{Juncheng Wei}
\address{\noindent  Department of Mathematics University of British Columbia, Vancouver, BC V6T 1Z2, Canada
}  \email{jcwei@math.ubc.ca}

\begin{document}

\begin{abstract}
A classical problem for the two-dimensional Euler flow for an incompressible fluid confined to a
smooth domain. is that of finding regular solutions with highly concentrated vorticities around $N$ moving
{\em vortices}. The  formal dynamic law for such objects was first derived in the 19th century by  Kirkhoff and Routh.  In this paper we devise a {\em gluing approach} for the construction of smooth $N$-vortex solutions. We capture in high precision the core of each vortex as a scaled finite mass solution
of Liouville's equation plus small, more regular terms. Gluing methods have been a powerful tool in geometric constructions by {\em desingularization}. We succeed in applying those ideas in this highly challenging setting.
\end{abstract}

\maketitle

\section{Introduction}

We consider the Euler equation
for an incompressible fluid confined to a smooth domain $\Omega\subset \R^2$  given by

\be\begin{cases} 	
{u}_t   +  ({ u }\cdot \nn ) {u } =    \nn p \ \,   \quad {\rm in }\  \Omega\times(0,T) \\
  \qquad\quad  {u }(\cdot,0) =  { u }_0  \, \ \quad \text{ in } \Omega\\
\qquad\quad\ \  { u }\cdot \nu  =   0  \quad\quad  \text{ on } \pp\Omega\times(0,T)\\
 \qquad\quad\     \nn \cdot {u }  =   0  \quad\quad  \text{ in } \Omega\times(0,T)
 \end{cases}
\label{10} \ee
Here $  { u}: \Omega\times [0,T)\to \R^2 $ designates the velocity field and
	$   p: \Omega\times [0,T)\to \R$ the pressure.  We assume in what follows that $\Omega$ is bounded and simply connected.
For a solution of \equ{10} its vorticity is defined as
$$\omega =  \nabla \times u =    \pp_{x_2} u_1 -  \pp_{x_1} u_2    .$$

\medskip
We have that
\equ{10} is equivalent to its vorticity-stream formulation
\be\begin{cases} 	
{\omega}_t   +   \nn^\perp \Psi \cdot \nn  \omega  =    0\ \,   \quad {\rm in }\  \Omega\times(0,T) \\
 \qquad\quad  {\omega }(\cdot,0) =  { \omega }_0  \, \ \quad \text{ in } \Omega\\
 \qquad\quad\     -\Delta  \Psi   =   \omega \quad\quad \text{ in } \Omega\times(0,T)\\
\qquad\quad\quad \quad     \Psi   =   0  \quad \quad \text{ on } \pp\Omega\times(0,T)
 %  {u }(\cdot,0) =&   { u }_0  \, \ \quad \text{ in } \Omega\nonumber
 \end{cases}
\label{11} \ee
For a solution $(\omega, \Psi)$ of \equ{11}, a solution of \equ{10} is recovered by means of the Biot-Savart law
$ u =  \nn^\perp \Psi$, where we denote  $(a,b)^\perp = (b,-a)$. The Poisson equation satisfied by the {\em stream function} $\Psi$ reads precisely as  $\omega =  \nabla \times u $.

\medskip
This paper deals with {\em vortex solutions} of Problem \equ{10}. Loosely speaking, a vortex solution of \equ{10} is one
for which its velocity field exhibits very fast rotation around one or more (time-dependent) points of the domain, in other words such that
its vorticity $\omega(x,t)$ appears highly concentrated around those points.
More precisely the issue is to consider a family of solutions of \equ{11} $(\omega_\ve,\Psi_\ve)$ dependent on a small concentration parameter
$\ve$ such that in distributional sense we have
\be\label{w*}
\omega_\ve (x,t)  \rightharpoonup %\omega^*(x,t) =
 \sum_{j=1}^N 8\pi \kappa_j \delta (x-\xi_j(t)) \ass \ve\to  0
\ee
where $\delta(x)$ is the Dirac mass at the origin, $\xi_j(t)\in \Omega$ for all $t\in (0,T)$ and $\kappa_j$ are constants.
Let $G(x,\xi)$ be the Green function of the problem
$$
-\Delta G(x,\xi)  = 8\pi\delta (x-\xi)      \quad\hbox{in } \Omega, \quad G(x,\xi) =0 \quad\hbox{on }\pp\Omega
$$
so that we should also have the formal limit
\be\label{psi*}\Psi_\ve(x,t) \rightharpoonup %\(x,t) =
\sum_{j=1}^N 8\pi \kappa_jG(x,\xi_j(t))  \ass \ve\to  0 .\ee
Since
\[
-\Delta \Gamma  =  8\pi \delta , \quad \Gamma (x) = 4\log \frac 1{|x|}
\]
we see that
\be \label{G}
G(x,\xi)  = \Gamma (x-\xi)  - H(x,\xi)
\ee
where $H(x,\xi)$ solves
$$
 -\Delta_x H(x,\xi)  =0   \quad\hbox{in } \Omega, \quad H(x,\xi) = \Gamma (x-\xi) \quad\hbox{on }\pp\Omega.
$$
A singular vortex solution of system \equ{11} is precisely one of the form $(\omega^s, \Psi^s)$ given by \equ{w*}-\equ{psi*}.
The time evolution of the vortices for such an object obeys a formal law that
can be derived as follows.
Letting
\be \label{op}\omega^s (x,t;\xi) =  \sum_{j=1}^N 8\pi \kappa_j \delta(x-\xi_j(t) ), \quad  \Psi^s(x,t,\xi) = \sum_{j=1}^N  \kappa_jG(x,\xi_j(t))
\ee
we find
%Singular (formal) vortex solutions:  Helmholtz (1858),  Kirchoff (1876), Routh (1881),  Lagally (1921) C.C. Lin (1941).
%so that $   -\Delta \Psi = \omega$.
$$
{\omega^s_t}  = - \sum_{j=1}^N 8\pi \kappa_j \nn \delta (x-\xi_j)\cdot \dot \xi_j,\quad
$$
$$
\nn^\perp \Psi^s \cdot \nn \omega^s =  \sum _{i,j=1}^N  8 \pi \, \kappa_i\kappa_j  \nn^\perp G(x,\xi_i) \cdot \nn \delta  (x-\xi_j)  .
$$
Since $   \Gamma(x)$ and $   \delta(x)$ are ``radially symmetric'', then $   \nn^\perp \Gamma (x-\xi_j)\cdot \nn \delta  (x-\xi_j)=0.$
Since $    \nn \delta  (x-\xi_j)$ is only supported at $ x=  \xi_j$, we get from \equ{G}
$$
 \nn^\perp G(x,\xi_j) \cdot \nn \delta  (x-\xi_j)  =
 - \nn^\perp  H(\xi_j,\xi_j)\cdot \nn \delta  (x-\xi_j) .
$$
Thus
$$
{\omega^s_t} +\nn^\perp \Psi^s \cdot \nn \omega^s    =
$$
$$
8\pi \, \sum_{j=1}^N [- \kappa_j \dot \xi_j  + \nn_x^\perp \big ( - \kappa_j^2 H(x,\xi_j) +  \sum_{i\ne j}\kappa_i\kappa_j  G(x,\xi_i )\big)  ] \cdot \nn \delta (x-\xi_j).
$$
So that $(\omega^s. \Psi^s)$
is a ``solution'' of Euler if and only if  $\xi= (\xi_1,\ldots , \xi_j)$ solves the ODE system
$$
\dot \xi_j (t)=    \nn^\perp_x ( - \kappa_j H(x,\xi_j(t)) +  \sum_{i\ne j}\kappa_i G(x,\xi_i(t) )  ) \big |_{x=\xi_j(t)}
$$
Equivalently, if and only  if $   \xi(t)= (\xi_1(t),\ldots , \xi_N(t))$ solves the Hamiltonian system
\be\label{sist}
\kappa_j\, \dot \xi_j(t)  =  \nn_{\xi_j}^\perp K(\xi(t)) , \quad j=1,\dots, N ,\quad t\in [0,T]
\ee
where
\be\label{K}
 K(\xi):=    -
 \frac 12 \sum_{i=1}^N \kappa_i^2 H(\xi_i,\xi_i) + \frac 12 \sum_{i \ne  j}\kappa_i \kappa_j G(\xi_i,\xi_j )
\ee
This is the {\em  Kirchoff-Routh function} \cite{cclin}. This dynamic law for the evolution of vortices traces back to the 19th century being first derived by Kirchoff and Routh
\cite{kirchoff,routh}.  System \equ{sist}, called {\em the $N$-vortex problem},  encodes interesting phenomena that have been the object of intensive recent investigation, see \cite{bartsch1,bartsch2,bartsch3} and their references.

\medskip
It is classical that the initial value problem for \equ{10} or \equ{11} is  well-posed for smooth initial data, we refer the reader for instance to the book \cite{majda} and references therein. In fact solutions are globally defined and smooth at all times.

\medskip
A classical problem  associated to the above derivation is the {\em  desingularized  $N$-vortex problem}, namely the existence of true smooth solutions of  Euler's equation \equ{11} with  highly concentrated vorticities around $N$ points $\xi_j(t)$  which obey a motion law  similar to \equ{sist}.

\medskip
Marchioro and Pulvirenti \cite{marchioropulvirenti} provided a first, important result for the  desingularized  $N$-vortex problem: Let  $\xi^0(t) = (\xi_{1}^0(t), \ldots, \xi_{N}^0(t))\in \Omega^N$ be a solution of system \equ{sist} with no collisions in $[0,T]$, namely
$$ \inf_{t\in [0,T]}  |\xi^0_i(t)-\xi^0_j(t)| >0 \foral i\ne j.$$    Then
there exists a smooth initial condition $\omega_{0\ve}(x)$, suitably $\ve$-concentrated around the points $\xi^0_j(0)$, such that the unique smooth solution
$(\omega_\ve,\Psi_\ve)$ of \equ{11} satisfies the convergence assertion \equ{w*}-\equ{psi*} for $\xi=\xi^0$ in the distributional sense.

\medskip
The proof in \cite{marchioropulvirenti} (done for simplicity in $\Omega=\R^2$) gives no clue on the behavior of the solution near the concentration cores, namely where the delicate behavior is happening.  Getting precise information is not easy because of the nearly singular character of the solution. Euler's equation is extremely sensitive to the regularity of the initial datum. Problem \equ{11} (in a suitable weak sense) is well posed for
initial conditions in $L^\infty(\Omega)$ \cite{yudovich}. Existence holds in some adequate measure spaces \cite{diperna}. See also
\cite{miot1,miot2} for related results.
%Lacave-Miot 2009  Lopes-Miot-Nussenzveig 2011.
The $H^s$-setting  is mysterious, see \cite{bourgain} and references therein. Euler's equation hides oscillatory behaviors that could completely spoil regularity and/or well-posedness.
For instance, it is known that nonuniqueness may arise in \equ{10} for continuous and even H\"older continuous settings.
Smooth solutions of \equ{10} are easily seen to preserve energy in the sense that
$$
E(t)= \int_\Omega |u( \cdot ,t)|^2 \ = \ constant.
$$
This is no longer the case, and continuous solutions with any prescribed energy $E(t)$ can be found, see \cite{delellis,schnirelmann}.
It is physically sound to obtain the asymptotic behavior of the energy density $$ e_\ve (t)= |u_\ve( \cdot ,t)|^2$$ for desingularized vortex solutions. That information does not follow from the rough distributional convergence \equ{w*}-\equ{psi*} or the method in \cite{marchioropulvirenti}.

\medskip
In this paper we revisit the $N$-vortex desingularization problem providing precise asymptotics of the desingularized solution $(\omega_\ve,\Psi_\ve)$ which
in particular yield the following information on the energy density:
\be \label{energy}
	\frac 1{|\log\ve|}| { u}_\ve(x,t) |^2  \rightharpoonup  \sum_{j=1}^N 8\pi \kappa_j^2\delta(x -\xi^0_j(t)),\quad u_\ve = \nn^\perp \Psi_\ve.
\ee
%In all what follows
% $  \xi(t) =  (\xi_1(t),\ldots, \xi_N(t))$ will denote a solution of system \equ{sist} with no collisions in $[0,T]$
        	%we will find a family of smooth solution $  (\omega_\ve,\Psi_\ve)$ to system \equ{11} that satisfies  \equ{w*}-\equ{psi*} as $\ve \to 0$
To state our result, we %let  $  \xi(t) =  (\xi_1(t),\ldots, \xi_N(t))$, $t\in [0,T]$, be an array of $N$  of distinct points in $\Omega$ and
consider the  $\ve$-regularization of $(\omega^s,\Psi^s)$  in \equ{op} given by

\begin{align}
\Psi_0  (x,t;\xi,\ve)  &=   \sum_{j=1}^N \kappa_j\, \big [  \, \log \frac 1{( \ve^2 +|x-\xi_j(t)|^2 )^2}  -  H(x,\xi_j(t))\, \big ] ,\label{piko1} \\
\omega_0  (x,t;\xi,\ve)  &=   \sum_{j=1}^N {\kappa_j}
\frac {8\ve^{2}}{ (\ve^2 +  |x-\xi_j(t)|^2 )^2 }
%\sum_{j=1}^N \frac {\kappa_j} {\ve^2} U_0 \left(\frac{x-\xi_j(t)} \ve \right ), \quad U_0(y) = \frac {8 }{( 1 +|y|^2 )^2}\ .
\label{piko2}\end{align}
Consistently, we directly check that  $$  -\Delta \Psi_0  = \omega_0 , \qquad \int_{\R^2} \frac {8\ve^{2}dx}{ (\ve^2 +  |x-\xi_j|^2 )^2 }    = 8\pi, $$ hence $(\omega_0, \Psi_0 )$ regularize $(\omega^s, \Psi^s)$ in the sense that, as $\ve \to 0$,
\begin{align*}
\omega_0 (x,t;\xi,\ve) &\rightharpoonup \omega^s(x,t;\xi)= \sum_{j=1}^N 8\pi\kappa_j \delta(x -\xi_i(t)) ,\\
 \Psi_0 (x,t;\xi,\ve ) &\rightharpoonup   \Psi^s(x,t;\xi) = \sum_{j=1}^N \kappa_j G(x,\xi(t)) .
\end{align*}
In addition, we see that in agreement with \equ{energy} we have
$$
\frac 1{|\log\ve|}|\nn \Psi _0 |^2  \rightharpoonup  \sum_{j=1}^N 8\pi \kappa_j^2\delta(x -\xi_j(t)) \ass \ve\to  0 .
$$

\medskip In what follows we fix
$  \xi^0(t) = (\xi_1^0(t),\ldots, \xi_N^0(t))\in \Omega^N$, a smooth, collisionless solution of the $N$-vortex system $\equ{sist}$ in $[0,T]$, and
consider the functions $(\Psi_0 , \omega_0)$ defined by \equ{piko1}-\equ{piko2} relative to $\xi = \xi^0$.

\medskip
Our main result states the existence of a solution  $(\omega_\ve, \Psi_\ve)$ of \equ{11}  of the form
 \be
 \begin{cases}\ \ \quad  \Psi_\ve (x,t) =  \Psi_0(x,t ;\xi_0,\ve)  + \psi_\ve(x,t),  \\\  \ \quad \omega_\ve(x,t) =  \omega_0(x,t  ;\xi_0,\ve)  + \phi_{\ve}(x,t) , \\ -\Delta \psi_{\ve}(x,t)  =  \phi_{\ve}(x,t) \end{cases}\quad \inn \Omega\times [0,T].
 \label{formas}\ee
where a precise control on the $  \ve$-smallness of    $\psi_\ve$ and $  \phi_\ve$ can be obtained.

\begin{theorem} \label{teo1}
	  There exists a solution $  (\omega_\ve,\Psi_\ve)$ of system $\equ{11}$ of the form $\equ{formas}$
	  such that for  any arbitrarily small $\sigma>0$  we have the uniform estimates
$$
 \begin{aligned}
	& |\phi_\ve(x,t)|\  \le  \ \ve^\sigma \omega_0 (x,t;\xi^0,\ve)   %\sum_{j=1}^N \frac 1{\ve^{2}} U_0 \left(\frac {x-\xi_j} {\ve}\right ),
\\ &|\psi_\ve(x,t)|\ +\ \ve \, |\nn \psi_\ve(x,t)| \ \le  \ C\, \ve^2 \, .\,
	  \end{aligned}	
\foral 	(x,t)\in \Omega \times [0,T] $$
\end{theorem}
We see that the solution $(\omega_\ve,\Psi_\ve)$ predicted above satisfies \equ{w*}, \equ{psi*} for $\xi=\xi^0$  and also \equ{energy}.

\medskip
The function $\Psi_0$ is a $\ve$-regularization of Green's function $G(x,\xi)$ where the fundamental solution
 $\log\frac 1{|x-\xi|^4}$ is replaced by
 $ \log\frac 1{(|x-\xi|^2 + \ve^2 )^2}$. There are many other ways to regularize producing a similar result as that of
 Theorem \ref{teo1}, but we have chosen this one since it is convenient for computations.
 The ultimate reason is that the function
\be\label{G0}
\Gamma_0(y) := \log \frac 8{(1+|y|^2)^2}
\ee
satisfies the classical Liouville equation \be \label{U0} -\Delta \Gamma_0   = e^{\Gamma_0}= \frac 8{(1+|y|^2)^2}=:U_0(y) \inn \R^2 .\ee

\medskip
The result of Theorem \ref{teo1} is connected with phenomena known for desingularized vortex-solutions of stationary Euler equation associated to
critical points of the Kirchoff-Routh energy \equ{K}.
It is well-known that a solution of a semilinear equation of the form
$$
-\Delta \Psi =  f(\Psi) \inn \Omega, \quad \Psi =0 \onn \pp \Omega
$$
corresponds to a steady state on \equ{11} setting $\omega = f(\Psi)$. In particular, it is known, see \cite{bp,dkm,egp}
that if $\xi= (\xi_1,\ldots,\xi_N)$ is a non-degenerate critical point of the functional \equ{K} for  $\kappa_j =1$ for all $j$,
then
the singularly perturbed Liouville equation
\begin{align*}
-\Delta \Psi_\ve  = &
\ve^2 e^{\Psi_\ve}=: \omega_\ve \ \hbox{in } \Omega \\
\Psi_\ve = &    0 \, \qquad\hbox{on } \pp\Omega
\end{align*}
has a solution  with
$
\omega_\ve(x) \rightharpoonup \sum_{j=1}^N 8\pi \delta(x-\xi_{j}).
$
Stationary desingularized vortex solutions associated to critical points of \equ{K} for general $\kappa_j$'s  have been found in \cite{wei,smets}.
Gluing methods like those leading to the above mentioned results have led to striking constructions in geometry and various asymptotically singular elliptic equations. The general scheme we follow in the current time-dependent setting is a gluing of a similar kind, however considerably  more delicate. As far a we know this is the first result where precise asymptotics are obtained for a problem of this kind in Euler flows.
	
\medskip
In the next section we explain the scheme of the proof of Theorem \ref{teo1}  which is carried out in the subsequent sections.

\section{Scheme of the construction in Theorem \ref{teo1}}
%\section{The outer-inner gluing system }

\subsection{Construction of an approximate solution}
We consider a solution $\xi^0(t)$ of \equ{sist} as in the statement of Theorem \ref{teo1} and a function $\xi(t)$ close to $\xi^0(t)$ that we leave
as a parameter to be adjusted.
 For the sake of notation we write in what follows the functions in \equ{piko1}-\equ{piko2}
 as $\Psi_0(x,t;\xi)$, $\omega_0(x,t;\xi)$ without making explicit their dependence on $\ve$.
 It is convenient to express them in the form

\begin{align}
\Psi_{0} (x,t;\xi) & =  \sum_{j=1}^N \kappa_j  \left[ \Gamma_0\left( \frac {x-\xi_j(t)} {\ve} \right ) -  H(x, \xi_j(t))  -   \log 8\ve^2 \right]
\label{psistar}
\\
\omega_{0} (x,t;  \xi) & =   \sum_{j=1}^N \frac {\kappa_j}{\ve^2}  U_0 \left( \frac {x-\xi_j(t)} {\ve} \right ).
%\label{omegastar}
\nonumber
\end{align}
where $\Gamma_0$ and $U_0$ are defined by \equ{G0}-\equ{U0}.

\medskip
We look for a solution of the Euler equation
\begin{align}
\label{euler}
\begin{aligned}
E(\omega, \Psi)\ := &\ \omega_t  +   \nn^\perp \Psi \cdot \nn \omega  \, =\, 0  \inn\quad  \Omega \times (0,T), \\
E_2(\omega, \Psi)\ := & \ \Delta \Psi + \omega  \, =\, 0\qquad \quad \inn\quad  \Omega \times (0,T) ,\\
& \qquad\Psi\,  =\, 0      \qquad\qquad  \onn\quad \pp\Omega \times (0,T) ,
\end{aligned}
\end{align}
of the form
\begin{align}
\omega(x,t)  = & \omega_0 (x,t;  \xi) +  \vp(x,t)\label{f1}\\
\Psi(x,t)  = & \Psi_{0} (x,t; \xi) +  \psi(x,t)\label{f2}
\end{align}
Let us compute the error of approximation for $(\omega_0(\cdot;\xi),\Psi_0(\cdot;\xi))$.
We get
$$
E(\omega_0(\cdot;\xi), \Psi_0(\cdot;\xi))  =
$$
$$
\, \ve^{-3}\sum_{j=1}^N [- \kappa_j \dot \xi_j  + \nn_x^\perp \big ( - \kappa_j^2 H(x,\xi_j) +  \sum_{i\ne j}\kappa_i\kappa_j  G(x,\xi_i )\big)  ] \cdot \nn_y U_0\left (\frac {x-\xi_j(t)}{\ve} \right )
$$
and therefore, setting  $y_j =\frac {x-\xi_j(t)}{\ve}$,
$$
\,E(\omega_0(\cdot;\xi), \Psi_0(\cdot;\xi))  \ =\   O(\ve^{-3} ) \sum_{j=1}^N \frac 1{ 1+|y_j|^{5} }
$$
so that in particular the error is of size $O(\ve^2)$ for $x$ away from all the vortices, and $O(\ve^{-3})$ very close to them. The choice
\be\label{vv}\xi(t) =\xi^0(t) + O(\ve)\ee
 in the $C^1$-sense, substantially reduces the error near the vortices. In fact
in that case we quickly see that
$$
\,E(\omega_0(\cdot;\xi), \Psi_0(\cdot;\xi))  \ =\   O(\ve^{-2 }) \sum_{j=1}^N
\frac 1{ 1+|y_j|^{4} }.
$$
The first step in the construction of a solution of the form \equ{f1}-\equ{f2} consists of finding an improvement of the approximation $ \vp_* (x,t;\xi)$, $ \psi_* (x,t;\xi)$
in such a way that
\be\label{w**}
\omega_* (\cdot ;\xi) := \omega_0 (\cdot ;\xi)+ \vp_*(\cdot ;\xi),\quad \Psi_* (\cdot ;\xi) := \Psi_0 (\cdot ;\xi)+ \psi_*(\cdot ;\xi),
\ee
satisfies for an arbitrarily  small $\sigma>0$,
\be \label{erroraprox}\begin{aligned}
 E_*(\ttt \xi) := E(\omega_* (\cdot ;\xi) , \Psi_* (\cdot ;\xi))\, =\,  & \, O(\ve^{1-\sigma} ) \sum_{j=1}^N
\frac 1{ 1+|y_j|^{3} }.\\
E_{2*}(\ttt \xi) :=E_2( \omega_* (\cdot ;\xi) , \Psi_* (\cdot ;\xi))\, =\, &\, O(\ve^{4-\sigma} )
\end{aligned}\ee
for  any $\xi$ of the form (more restrictive than \equ{vv})
$$\xi(t) = \xi^0(t) + \xi^1(t) + \ttt \xi(t) $$
where $\xi^1(t) $ is a certain  explicit $O(\ve^2 |\log \ve| )$ correction of $\xi^0(t)$ and $\ttt \xi(t) = O(\ve^{4-\sigma}) $
in the $C^1$-sense.  This correction corresponds to a substantial improvement of approximation, for the error is now  $O(\ve^{4-\sigma})$ away from the vortex locations and $O(\ve^{1-\sigma})$ close to them.
A good part of the remaining of this paper precisely corresponds to the construction of functions $(\omega_*,\Psi_*)$ as in \equ{w**}-\equ{erroraprox}.

\subsection{Construction of a solution}
We consider a smooth cut-off function $\eta(s)$ with
\be\label{eta} \eta (s)\ =\ \begin{cases} 1 \ \hbox{ for $s\le 1$},\\
  0 \ \hbox{ for $s\ge 2$}.
 \end{cases}\ee

We look for a solution of the Euler equation \equ{euler}
of the form
$$
\begin{aligned}
\omega(x,t)  = & \omega_* (x,t;  \xi) +  \vp(x,t)\\
\Psi(x,t)  = & \Psi_{*} (x,t; \xi) +  \psi(x,t)
\end{aligned}
$$
where $\vp$ and $\psi$ are small corrections of the previously found first approximations.
Each remainder is decomposed as the sum of a  quantity concentrated near the vortices, measured in the slow variables $y_j= \frac{x-\xi_j}{\ve}$ of the ``bubbles'' $\omega_0$, and conveniently cut-off at a small distance from  the vortices, plus a more regular term. More precisely, we set
\begin{align}
  \vp(x,t)= &  \sum_{j=1}^N  {\kappa_j}{\ve^{-2}}  \eta_{1R}^j\, \phi_j \left( \frac {x- \xi_j(t)} {\ve},t \right )\, +\, \phi^{out}(x,t)\label{f11}\\
\psi(x,t)  = &  \sum_{j=1}^N  {\kappa_j} \eta_{2R}^j \, \psi_j \left( \frac {x-\xi_j(t)} {\ve},t \right )\, +\,  \psi^{out}(x,t)\label{f22}
\end{align}
where we denote
%\be
\[
\eta_{mR}^j(x,t;\xi) = \eta\left ( \frac {|x- \xi_j(t)|}{m R\ve}    \right ).
%\label{etam}\ee
\]
for a large,  $\ve$-dependent number $R>0$ with $R\ve \ll 1$ that we will later specify. The {\em inner-outer gluing method} consists of finding
$(\vp , \psi)$ of the form \equ{f11}-\equ{f22}
such that the functions
\begin{align*}&\phi^{in}(y,t) = (\phi_1(y,t),\ldots, \phi_N(y,t)),
\ \psi^{in}(y,t) = (\psi_1(y,t),\ldots, \psi_N(y,t)),  &\quad (y,t)\in \R^2\times [0,T]  \\ &\psi^{out}(x,t),\quad \phi^{out}(x,t),\quad   \xi(t)=\xi^0(t)+\xi^1(t) +\ttt \xi(t),\quad\qquad& (x,t)\in \Omega \times [0,T]
\end{align*}
satisfy the following system of equations.
\begin{align}\label{inn1}\left\{ \begin{aligned}
&E_j(\phi^{in},\psi^{in}, \psi^{out},\ttt \xi)(y,t)\, :=\, \ve^2 \pp_t \phi_j\, \\ & + [ \nn_y^\perp (\Psi_* + \kappa_j \psi_j +\psi^{out}) -\ve\dot\xi  ]\cdot \nn_y\phi_j +  \ve^2\nn_y (\kappa_j \psi_j +\psi^{out}) \cdot \nn \omega_*
 \\ & +\,  E_*(\ttt\xi) \, =\, 0, \\
&-\Delta_y \psi_j = \phi_j, \quad (y,t)\in B_R(0)\times [0,T],
 \end{aligned}
\right. \end{align}
coupled with
\begin{align}\label{out1}\left\{ \begin{aligned}
&E^{out}_1(\phi^{in},\psi^{in}, \psi^{out},\ttt \xi)(x,t)\, :=\, \\ & \pp_t \phi^{out} \, + \nn_x^\perp [\Psi_* + \sum_{j=1}^N \kappa_j \eta_{2R,j} \psi_j +\psi^{out})]\cdot \nn_x\phi^{out} \\
&+  \ve^{-2}  \sum_{j=1}^N  \kappa_j \phi_j \big[ \pp_t \eta_{1R,j} + \nn_x^\perp (\Psi_* + \sum_{j=1}^N \kappa_j \eta_{2R,j} \psi_j +\psi^{out})) \cdot
 \nn_x \eta_{1R,j} \big ]
\\ &+  (1-  \sum_{j=1}^N\eta_{1R,j}) \nn_x ( \,  \sum_{l=1}^N \kappa_l \eta_{2R,l}\,\psi_l +\psi^{out}) \cdot \nn_x \omega_*
 \\ & +\,   (1- \sum_{j=1}^N \eta_{1R,j} ) E_*(\ttt\xi) \, =\, 0,\quad (x,t)\in \Omega \times [0,T]\end{aligned}
\right. \end{align}
and
\begin{align}\label{out2}\left\{ \begin{aligned}
&E^{out}_2(\psi^{in},\phi^{out}, \psi^{out},\ttt \xi)(x,t)\, :=\, \Delta_x \psi^{out}
   + \phi^{out}
\\ &
 +  \sum_{j=1}^N \kappa_j  (\psi_j \Delta_x\eta_{2j}  + 2 \nn_x \eta_{2j} \cdot \nn  \psi_j) +  E_{2*}(\ttt\xi)  = 0  ,
\quad (x,t)\in \Omega \times [0,T],
 \\ &\qquad \quad \psi^{out}=0, \quad (x,t)\in \pp\Omega \times [0,T]. \end{aligned}
\right. \end{align}
In the above expressions, for a function $g(x,t)$, when no confusion arises, we write $g(y,t)$ meaning
$ g ( \xi_j(t) + \ve y , t) $.

\medskip
A solution  $(\phi^{in}, \psi^{in}, \phi^{out}, \psi^{out}) $ of system \equ{inn1}-\equ{out1}-\equ{out2}
yields by simple addition a solution of \equ{euler}. Appropriate smallness of the remainders $\vp$ and $\psi$ in \equ{f11}-\equ{f22}  will indeed be
obtained if $\ttt\xi$ is in addition suitably adjusted.  In order to obtain the desired solution (with initial conditions equal to zero in all
the parameter functions) we will formulate the system as a fixed point problem for a compact operator in a ball of a suitable Banach space.
We will find a solution by means of a degree theoretical argument. That involves establishing a priori estimates for a homotopical  deformation of the problem into a linear one.

The rest of this paper is devoted to carrying out in detail the steps outlined above.

\section{Construction of a first approximation}\label{aproxsol}

This section will be devoted to the construction of the first approximation $(\omega_*,\Psi_*)$ as a perturbation of $(\omega_0,\Psi_0)$ in the form \equ{w*}.
 We will do this in the corresponding version of
System \equ{inn1}-\equ{out2} where $(\omega_*,\Psi_*)$ is replaced by
$(\omega_0,\Psi_0)$. More precisely,
we look for an approximate solution of the form
\begin{align}
  \omega_*(x,t;\xi) = & \omega_0(x,t;\xi )  + \sum_{j=1}^N  {\kappa_j}{\ve^{-2}}  \eta_{1R}^j\, \phi_j \left( \frac {x- \xi_j(t)} {\ve},t \right )\, +\, \phi^{out}(x,t)\label{f111}\\
\Psi_*(x,t;\xi)  = & \Psi_0(x,t;\xi) + \sum_{j=1}^N  {\kappa_j} \eta_{2R}^j \, \psi_j \left( \frac {x-\xi_j(t)} {\ve},t \right )\, +\,  \psi^{out}(x,t)\label{f222}
\end{align}
with a similar notation as in \equ{f11}-\equ{f22}.
In the remaining of this section we let
\be \label{RR}R = \frac{\delta}{\ve}
\ee where $\delta$ is a fixed, sufficiently small number. We denote
$$
\quad  \phi^{in}(y,t) = (\phi_1(y,t) , \ldots, \phi_N(y,t) ), \quad \psi^{in}(y,t)  = (\psi_1(y,t) , \ldots, \psi_N(y,t) ).
$$
Let  $E(\omega, \Psi)$, $E_2(\omega,\Psi)$ be the ``error'' operators defined in \equ{euler}.
Then for $(\omega_*, \Psi_*)$ given by \equ{f111}-\equ{f222}
 the following holds:
\be\label{error0}
\begin{aligned}
 E(\omega_*, \Psi_*) \ =&\    \ve^{-4}  \sum_{j=1}^N  \eta_{1R}^j E_{0j}(\phi_j, \psi_j, \psi^{out}; \xi)   +   E_0^{out} ( \phi^{in}, \psi^{in}, \phi^{out}, \psi^{out}; \xi)\\
  E_2(\omega_*, \Psi_*)\ =&\   E_{20}^{out}( \psi^{in}, \phi^{out},\psi^{out} ,\xi  ).
\end{aligned}
\ee
Here, for $j=1,\ldots, N$,
\be\label{Ej} \begin{aligned}
&E_{0j}(\psi_j,\phi_j, \psi^{out} ; \xi ) \ : =\ \\
& \ve^2 \pp_t \phi_j  +    \big [   \nn_y^\perp \big (\Psi_0(\cdot ;\xi)   + \kappa_j  \psi_j + \psi^{out}  \big ) - \ve  \dot { \xi_j} \big ]\, \cdot\, \nn_y (U_0 + \phi_j ),
 \end{aligned}
\ee
while
\begin{align}
\nonumber
E^{out}_0( \phi^{in},\psi^{in}, \phi^{out},\psi^{out} ,\xi  ) (x,t)
&  =
\phi^{out}_t    +\,  \nn^\perp _x\big (\Psi_0 (\cdot ; \xi)
+ \sum_{j=1}^N \eta_{2j}\psi_j  + \psi^{out} \big )  \cdot \nn \phi^{out}
\\
\nonumber
& \quad +\ \ve^{-2}  \sum_{j=1}^N \kappa_j\, \pp_t\eta_{1j}\phi_j  -\,\ve^{-2} \sum_{j=1}^N  \kappa_j(1-\eta_{1j} )  \ve^{-1}  \dot{\xi}_j \cdot  \nn_x U_{0j}
\\
\nonumber
& \quad +\,
\ve^{-2}  \sum_{j=1}^N \kappa_j\, [ \phi_j  \nn_x\eta_{1j} +   (1-\eta_{1j})  \nn_x U_{0j})\, ]\,  \\
&
\quad \cdot
    \nn^\perp \big (\Psi_0 (\cdot; \xi)  + \sum_{j=1}^N \eta_{2j} \psi_j + \psi^{out} \big )
    \inn \Omega\times [0,T],
\label{E1}\end{align}
where
 $U_{0j}(x,t)  =  U_0\left ( \frac {x- \xi_j(t)}{\ve}    \right )  $,
and
\be
   E_{20}^{out}( \psi^{in}, \phi^{out},\psi^{out} ,\xi  )\ = \
     \Delta_x  \psi^{out}    +
     \phi^{out}
      +  \sum_{j=1}^N \kappa_j ( \psi_j \Delta_x\eta_{2j}  + 2 \nn_x \eta_{2j} \cdot \nn_x   \psi_j)
\label{E2}\ee
and we assume
$$
\psi^{out} = - \Psi_0  \onn \pp\Omega\times [0,T].
$$
The main result of this section states the existence of an improvement of approximation of the form \equ{f111}-\equ{f222} so that in particular the bounds  \equ{erroraprox} hold.

\medskip
We will restrict ourselves to parameter functions $\xi(t)$ of the form
\be
\xi(t) = \xi^0(t)  + \xi^1(t) + \tilde \xi(t)
\label{xi0}\ee
where $\xi^1(t)$ is an explicit function  with size $O(\ve^2\log\ve)$, and for a fixed, arbitrarily small $\sigma>0$ we impose
$$
\begin{aligned}
\|\ttt\xi \|_{C^1[0,T]}\,:=\,\|\ttt \xi \|_\infty + \|\dot {\ttt \xi} \|_\infty \  & \le\  \ve^{4-\sigma}.
\end{aligned}
$$

\begin{prop}\label{aproxi}
There exists a function $\xi^1(t)$ as above such that that for any $\xi(t)$ as in $\equ{xi0}$ there exist  functions
$
\phi_{j}^*(y,t;\ttt\xi) ,  \ \psi_{j}^*(y,t;\ttt\xi) ,\
\phi^{out*}(x,t;\ttt\xi) , \  \psi^{out*}(x,t;\ttt\xi)
$
such that
\begin{align*}
(1+ |y|^2) \,|\phi^*_{j}(y,t;\ttt\xi)|\, +\,  |\psi_{j}^*(y,t;\ttt\xi)| \, \le  &\,  C\ve^2,  \inn B_R\times [0,T],    \\
|\phi^{out*}(x,t;\ttt\xi)|\,  +\, |\psi^{out*}(x,t;\ttt\xi)| \le  &\,  C\ve^2,      \inn \Omega \times [0,T]  \\
\end{align*}
and the following error bounds hold:
\begin{align*}
&E_{0j} (\phi_{j }^* , \psi_{j}^* , \psi^{out*},\xi)(y,t) = \\
 &  \ve  [ -\dot {\ttt \xi}(t) + \nn_{\xi_j} ^\perp K(\xi^0(t) +\xi^1(t) +\ttt \xi(t) ) -   \nn_{\xi_j} ^\perp K(\xi^0(t) +\xi^1(t) )] \cdot \nn_y U_0(y)  + E^*_j (\ttt \xi )(y,t)
\end{align*}
for $K$ given by $\equ{K}$, and
\begin{align*}
 \big|\,E^*_j (\ttt \xi )(y,t) \big| \ \le &\   \frac{ C}
 { 1 + |y|^{3}} \, \, \ve^5|\log\ve|^2 ,\\
 \big|\,E_*^{out}(\ttt\xi) (x,t) \big|\, +\, \big|\,E^{out}_{2*}(\ttt\xi) (x,t) \big|    \ \le & \ C \ve^4 |\log\ve|^2
 \end{align*}
and we have denoted
\be \begin{aligned}\label{roger}
E^{out}_* (\ttt\xi)\, :=& \,E^{out}_0  (\phi^{in*} ,\psi^{in*} , \phi^{out*}, \psi^{out*},\xi ) ,\\ E^{out}_{2*}(\ttt\xi) \, := &\, \,E^{out}_{20}  (\psi^{in*} , \phi^{out*}, \psi^{out*},\xi ).
\end{aligned}
\ee

\end{prop}

We observe that for $(\omega_*, \Psi_*)$ in $\equ{f111}$-$\equ{f222}$
 with the parameter functions considered above, we have the total errors in $\equ{error0}$ estimated as
$$
\begin{aligned}
| E(\omega_*, \Psi_*)(x,t)| \ \le &\     C\ve^{1-\sigma}\sum_{j=1}^N \frac 1 {1+ |y_j|^3} , \quad y_j = \frac {x-\xi_j(t)} {\ve}, \\
  |E_2(\omega_*, \Psi_*)(x,t)|\ \le &\   C\ve^{4-\sigma}.
\end{aligned}
$$
precisely as predicted in \equ{erroraprox}. The construction yields
uniform Lipschitz dependence on $\ttt \xi$ of this error of the same type as those above. The rest of this section will be devoted to
building this approximate solutions and to proving Proposition \ref{aproxi}.

%Our first step consists of finding functions  $\bar \phi_j(y,t)$, $\bar \psi_j(y,t)$, $\bar \xi(t)$ which suitably reduce these inner errors.Once that choice has been made, we find an improvement of the approximation $(\phi^{out*}, \psi^{out*})$ for the outer problem.In a last step we consider the effect in the inner problems of the latter choice of functions. We make last corrections  $(\phi_j^*,\psi_j^*)$ to the inner parameter functionswhose effect is negligible in the outer error and represent a global improvement of the approximation errors.

 %so that for all $j$ we have the bound
 %\begin{equation}\label{ester1}
%\left|  E_j(\bar \psi_j,\bar \phi_j,\bar \xi) (y,t)  \right|\ \leq\  C {\varepsilon^8}({1+|y|^2})\foral (y,t)\in B_R(0)\times[0,T].
% $\end{equation}
%for a constant $C$ independent of $\ve$.
 %We write
 %$$
 %E_j(\psi_j,\phi_j,\xi_j) =  \ve^2 \pp_t \phi_j    +  \nn_y^\perp \big [ - \ve  \dot { \xi_j} \cdot y +   \Psi_0  + \kappa_j  \psi_j \big ] \cdot
 %\nn_y (U_0 + \phi_j ).
 %$$

\subsection{Expression for the errors $E_{0j}$}
 We will first find a convenient expression for the $j$-inner error \equ{Ej}.  We write
\begin{align*}
\Psi_0 (\xi_j + \varepsilon y,t; \xi)
& =
\kappa_j
\log \frac{1}{ \varepsilon^4(1 + |y|^2)^2}
-  \kappa_j  H(\xi_j + \varepsilon y,\xi_j)
\\
&\quad
+
\sum_{i\not=j}
\kappa_i
\left(
\log \frac{1}{(\varepsilon^2 + |\xi_j-\xi_i + \varepsilon y|^2)^2}
-  H(\xi_j + \varepsilon y,\xi_i)
\right)
\\
&=
\kappa_j \big [\, \Gamma_0(y)
- \log(8\varepsilon^4)
+ \ttt \varphi_j(\xi_j + \varepsilon y;  \xi)
\, \big]
\end{align*}
where
$$
\ttt \varphi_j(x ;  \xi)= \varphi_j(x;  \xi)
+\varepsilon^2 \theta_{j,\varepsilon}(x;  \xi)
$$
and
%\xi = ( \xi_1,\ldots,\xi_k)$, and
\begin{align*}\varphi	_j(x ;  \xi) =&
-    H(x,\xi_j)
+
\sum_{i\not=j}
\kappa_j^{-1}\kappa_i G(x,\xi_i) ,\\
\theta_{j,\varepsilon}(x;\xi)
=&\frac{2}{\varepsilon^2}
\sum_{i\not=j}
\kappa_i\kappa_j^{-1}
\log \frac{ |x-\xi_i |^2}{\varepsilon^2 + |x-\xi_i|^2}.
\end{align*}
Let us set
\begin{align}
\label{defRa}
\RR_j (y,t;\xi) = \ve \kappa_j^{-1}\dot \xi_j^\perp \cdot y  +  \tilde \varphi_j(\xi_j+\ve y; \xi )  - \tilde \varphi_j (\xi_j ; \xi ).
\end{align}
For notational simplicity we omit below the subindex $j$ in $\phi_j$ and $\psi_j$.
\begin{align}
\kappa_j ^{-1} E_{0j} (\phi,\psi, \psi^{out};\xi) &=
\kappa_j ^{-1}\varepsilon^2 \partial_t  \phi
+ \nabla_y^\perp \left[   \Gamma_ 0 + \RR_j +    \psi + \kappa_j^{-1} \psi^{out}  \right]
\cdot
 \nabla_y   (U_0 + \phi ) \nonumber  \\
 %& = \varepsilon^2 \partial_t  \phi
%+ \kappa_j \left[ \nabla_y^\perp  \Gamma_ 0 \cdot   \nabla \phi + \nabla^\perp (\RR +    \psi ) \cdot \nabla U_0 +
%\nabla^\perp (\RR + \psi) \cdot \nabla \phi \right] \\
&=  \nabla ^\perp  \Gamma_ 0 \cdot   \nabla \phi + \nabla^\perp    \psi  \cdot \nabla U_0  \nonumber \\
&+  \nabla^\perp \RR_j  \cdot \nabla U_0  +
\nabla^\perp \RR_j  \cdot \nabla \phi +   \nabla^\perp \psi \cdot \nabla \phi  + \kappa_j ^{-1}\ve^2 \partial_t  \phi \nonumber \\ & +    \kappa_j^{-1}\nn^\perp \psi^{out}\nn (U_0 +   \phi).
\label{Ej1}\end{align}
 We find first functions $(\phi,\psi,\xi)$  that improve the first error
\be\label{e0}
\kappa_j^{-1}E_{0j}(0,0,0;\xi)(y,t) =   \nabla_y^\perp \RR_j(y,t)  \cdot \nabla U_0(y).
=   \frac {O(\ve )}{ 1+|y|^{5}} .  \ee
As in the statement of the proposition
we restrict ourselves to parameter functions $\xi(t)$ of the form
\be
\xi(t) = \xi^0(t)  + \xi^1(t) + \tilde \xi(t)
\label{xi}\ee
where for some fixed numbers $M>0$ and $\sigma \in (0,1)$ that we will fix later,
we have
\be
\begin{aligned}
\|\xi^{1} \|_{C^2[0,T]}\,:=\,\|\xi^{1} \|_\infty + \|\dot \xi^{1} \|_\infty + \|\ddot \xi^{1} \|_\infty\  & \le\  M\ve^2|\log \ve|, \\
\|\ttt\xi \|_{C^1[0,T]}\,:=\,\|\ttt \xi \|_\infty + \|\dot {\ttt \xi} \|_\infty \  & \le\  \ve^{4-\sigma}.
\end{aligned}
\label{cotasxi}\ee

%where $ \|\xi^{01} \|_{C^2[0,T] } \le M\ve^2 , \quad \|\ttt \xi^1 \|_{C^2[0,T] } \le M\ve^{2+\sigma} $

\medskip
The elimination of part of the error \equ{e0} will be done by solving
elliptic equations that involve the linear operator in the second line of formula \equ{Ej1}
$$
 L[\psi]\,  :=    \, \nabla^\perp  \Gamma_ 0 \cdot   \nabla \phi + \nabla^\perp    \psi  \cdot \nabla U_0, \quad \phi= -\Delta\psi .
$$
More precisely, for the large number $R$ given by \equ{RR}  we consider the general equation
\be\label{linear1}
L[\psi] + g =0\inn B_{8R}, \quad \psi = 0\onn \pp B_{8R}, \quad -\Delta \psi = \phi.
\ee
For a  bounded function $g: B_{8R}\subset \R^2  \to\R$. We solve this problem finding estimates for $\psi$  in terms of prescribed decay on $g$.

\subsection{Solving Equation \equ{linear1}}

 Using that $ -\Delta \Gamma_0 = f(\Gamma_0) = U_0 $ where $f(u) = e^u$ and $-\Delta \psi = \phi$ we see that
 \begin{align*}
L[\psi]\  := &  \ \  \ \nabla^\perp  \Gamma_ 0 \cdot   \nabla \phi + \nabla^\perp    \psi  \cdot \nabla U_0\\
=& - \nabla^\perp \Gamma_0 \cdot \nabla [ \Delta \psi + f'(\Gamma_0)\psi ].
\end{align*}
% For a given bounded function $g: B_R\subset \R^2  \to\R$  we consider the elliptic equation
%\be\label{linear1}
%L[\psi] + g =0\inn B_R, \quad \psi = 0\onn \pp B_R, \quad -\Delta \psi = \phi
%\ee
%for a large number $R>0$.
Let us consider polar coordinates $y=\rho e^{i\theta}$ in $\R^2$.
It is then easy to check that
$$
L[\psi]  =  - \frac 4{\rho^2 + 1} \frac{\pp}{\pp\theta}  [ \Delta \psi + f'(\Gamma_0)\psi ]
$$
Therefore a necessary condition for the solvability of \equ{linear1} is that
\be \label{ort0} \int_0^{2\pi} g(\rho e^{i\theta} )\, d\theta =0 \foral  \rho\in (0,8R). \ee
On the other hand we clearly have that
$$
L[Z_\ell] = 0, \quad Z_\ell(y) = \pp_{y_\ell} \Gamma_0(y) .
$$
In addition to \equ{ort0} we assume the orthogonality conditions
\be \label{ort2} \int_{B_R}(1+|y|^2)\, g(y)\, Z_\ell (y)\, dy\, =\, 0, \quad \ell=1,2. \ee
We consider right hand sides $g$ with a decay rate in $|y|$. We assume \be \label{decay}
|g(y)| \  \le\  (1+ |y|)^{-\alpha}
\ee
We have the validity of the following result.

\begin{lemma}\label{alpha} Assume that $3<\alpha \le 5$ There exists a constant $C>0$ such that for all $R$ sufficiently large and
 $g\in L^\infty (B_R)$ that satisfies conditions $\equ{ort0}$, $\equ{ort2}$ and $\equ{decay}$,
there exists a unique solution  $(\psi,\phi)$ of equation $\equ{linear1}$ that satisfies
$$ \int_0^{2\pi} \psi(\rho e^{i\theta})\, d\theta =0 \foral  \rho\in (0,8R). $$
 and the estimate
\begin{align}
& |\psi(y)| \, +\,  (1+|y|)\, |\nn \psi(y)|\,  + \,  (1+ |y|^2)\, | \phi(y)| %\, + \, (1+ |y|^4)\, |\nn^\perp \Gamma_0 \cdot \nn \phi(y) |
\nonumber
\\
& \le \  C{(1+|y|)^{4-\alpha}} \,\begin{cases} \log \Big ( \frac {16R} {|y|+1}  \Big) & \hbox{ if \quad} \alpha =5   \\ 1  & \hbox{ if \quad}  3< \alpha < 5   \end{cases}
\label {formula} \end{align}

\end{lemma}

\begin{proof}
%Under condition \equ{ort0}, Problem \equ{linear1} is  equivalent to

%Let
%\[
%h = -(\Delta \psi + f'(\Gamma_0) \psi ).
%\]
Let us decompose $\psi(y)$, $g(y)$ in Fourier series, using polar coordinates $ y =\rho e^{i\theta}$
\be \label{fourier}
\psi (\rho,\theta) = \sum_{k\in \Z}  p_k(\rho) e^{ik\theta} ,
\quad
g (\rho,\theta) = \sum_{k\in \Z}  g_k(\rho) e^{ik\theta} .
\ee
Condition \equ{ort0} amounts to $g_0 \equiv 0$. Imposing $p_0\equiv 0$,
equation \eqref{linear1} decouples into the infinitely many problems.
\be\label{pk}
\mathcal L_k [p_k] := \partial^2_{\rho} p_k + \frac{1}{\rho}\partial_{\rho} p_k
-\frac{k^2}{\rho^2} p_k + \frac{8p_k}{(1+\rho^2)^2}
= \frac{i(1+\rho^2)}{4k} g_k(\rho), \quad p_k(R) = 0.
\ee
For each $k\ne 0$, there exists a positive function $\zeta_k(\rho)$ such that $\mathcal L_{k} [\zeta_k] =0$ and
\[
\zeta_k(\rho) =  \rho^{|k|} (1+ o(1))  \quad \text{as }\rho \to 0
\]
For $k=\pm 1$ we explicitly have
$ \zeta_k(\rho) = \frac{\rho}{1+\rho^2} $, while for $|k|\ge 2$ we have
\[
\zeta_k(\rho) =  \rho^{|k|} (1+ o(1))  \quad \text{as }\rho \to +\infty.
\]
 Problem \equ{pk} is uniquely solved by the formula
$$
p_{k} (\rho)   =  \mathcal L_k^{-1}[g_k]\, :=    \frac {i}{4k}\zeta_k(\rho)  \int_\rho^{8R} \frac {dr}{ r\zeta_k(r)^2} \int_0^r (1+s^2) g_{k}(s)\zeta_k(s)\,s\, ds . \quad
$$
 Condition \equ{decay} corresponds to $$|g(\rho,\theta)|\ \le\    (1+ \rho)^{-\alpha} $$ with $3< \alpha \le 5$. Let us consider the case $|k|\ge 2$. We have
$$
|p_{k} (\rho)|\ \le\  \mathcal L_k^{-1}[(1+\rho)^{-\alpha} ]\le   c_k (1+\rho)^{|k|}  \int_\rho^{8R}   (1+ r)^{-|k|+3 -\alpha} dr
$$
so that
%\be \label{cotapk}
\[
|p_{k} (\rho)|\ \le\      c_k (1+\rho)^{4-\alpha}
%\begin{cases}%(1+\rho)^{|k|} R^{4-\alpha-|k|} &\hbox{ if } \alpha+ |k| < 4 \\
% (1+\rho)^{4-\alpha}\log\frac  R\rho &\hbox{ if } \alpha+ |k| = 4 \\  (1+\rho)^{4-\alpha}  &\hbox{ if }    \alpha+ |k| > 4 \end{cases}
%\ee
\]
since
$\alpha+ |k| > 4$.
Let us denote $ \bar P(\rho) =   \mathcal L_2^{-1}[ (1+ \rho)^\alpha ]$ we claim that for some $\gamma>0$ and all $|k|\ge 2$ we have the validity of
the estimate
$$
 |p_{k} (\rho)| \ \le \  \frac \gamma {k^3}\bar P(\rho) .
$$
That follows from the fact that the right hand side defines a positive supersolution for the real and imaginary parts of \equ{pk}. Indeed, if $\gamma$ is taken sufficiently large we get
$$
\mathcal L_k \big [ \frac \gamma {k^3}\bar P(\rho) \big ]  +  \frac {1+\rho^2}{4k} |g_k(\rho)|    \le     \gamma \frac {4-k^2} {10\rho^2} (1+ \rho)^{4+\alpha}
$$
Fourier modes $\pm 1$ need to be separately treated because $\zeta_1(\rho)$ decays at infinity.
At this point we observe that
$$
Z_1(y) = \zeta_1(\rho) \cos\theta , \quad Z_2(y) = \zeta_1(\rho) \sin\theta
$$
and that the orthogonality conditions \equ{ort2} assumed are equivalent to
$$
\int_0^{8R} (1+\rho^2)\,g_{\pm1}(\rho) \zeta_1(\rho)\, \rho\, d\rho = 0  .
$$
Therefore we can write
$$
p_{\pm 1} (\rho)  =    \mp \frac {i}{4}\zeta_1(\rho)  \int_\rho^{8R}  \frac {dr}{ r\zeta_1(r)^2} \int_r^{8R}  (1+s^2) g_{\pm 1}(s)\zeta_1(s)\,s\, ds . \quad
$$
and we obtain, if we now assume $3 < \alpha  \le 5$,
$$
|p_{\pm 1} (\rho) | \ \lesssim \    \begin{cases}   \rho^{-1} \log {\frac {16R}{\rho+1}}  &\hbox{ if } \alpha = 5 \\ (1+ \rho)^{4 -\alpha } &\hbox{ if }3<  \alpha < 5     . \end{cases}
$$
The desired result then follows from addition of the above estimates since
$$
|\psi(y)| \ \le\ \sum_{k\in \Z} |p_k(|y|)| .
$$
Finally,  since $\psi$ satisfies the equation
$$
 \Delta \psi + f'(\Gamma_0) \psi  =  -\frac 14 (1+\rho^2) \int_0^\theta g(\rho,\theta)\, d\theta, \inn B_{8R} , \quad \psi=0\onn\pp B_{8R} ,
$$
  the bounds for $\phi$ and $\nn \psi $ follow from standard elliptic estimates.

\end{proof}

\begin{remark}\label{remark100}{\em
We observe that if $2\le \alpha \le 3$ and $|k|\ge 2$ the estimate for $p_k$ given by \equ{pk} yields
\[
|p_{k} (\rho)|\ \le\      c_k (1+\rho)^{4-\alpha}
\begin{cases} %(1+\rho)^{|k|} R^{4-\alpha-|k|} &\hbox{ if } \alpha+ |k| < 4 \\
\log\frac  {16R}{\rho+1} &\hbox{ if } \alpha+ |k| = 4 \\  1  &\hbox{ if }    \alpha+ |k| > 4 \end{cases}
\]
Then if $g_0=g_{\pm 1}= 0$ and $2\le \alpha \le 3$   we have that the statement of Lemma \ref{alpha} holds with estimate \equ{formula} replaced with
$$
\begin{aligned}
& |\psi(y)| \, +\,  (1+|y|)\, |\nn \psi(y)|\,  + \,  (1+ |y|^2)\, | \phi(y)| %\, + \, (1+ |y|^4)\, |\nn^\perp \Gamma_0 \cdot \nn \phi(y) |
\nonumber
\\
& \le \  C{(1+|y|)^{4-\alpha}} \,\begin{cases} \log \Big ( \frac {16R} {|y|+1}  \Big) & \hbox{ if \quad} \alpha =2   \\ 1  & \hbox{ if \quad}  \alpha >2 \hbox{ or } \alpha =2  \hbox{ and } g_{\pm 2}=0.  \end{cases}
 \end{aligned}
$$

}
\end{remark}

%\proof

\subsection{Expansion of $\nn^\perp \RR_j$ }

We Taylor expand
\begin{align*}
\ttt \vp_j ( \xi_j+   \ve y ; \xi)   = &\  \ttt \vp_j ( \xi_j ; \xi ) + \ve \nn_x\vp_j ( \xi_j ; \xi )\cdot y  + \sum_{k=2}^4 \frac{\ve^2}{k!}
 D_x^2 \vp_j(\xi_j; \xi) [y]^k  \\
&\, +\,  \ve^3 D_x \theta_{j , \ve} (\xi_j; \xi ) [y]+{\ve^4 \over 2} D_x^2 \theta_j^\ve (\xi_j ; \xi)[y]^2+ Q(\ve y , \xi)
\end{align*}
where $Q(z,\zeta)$ is a function smooth in its arguments that satisfies
$$
  |\nn _zQ(z , \zeta)|\ \le \ C\ve |z|^4.
$$
Hence we find
\begin{align}
\nabla^\perp_y \RR_j(y,t;\xi )&=  \kappa_j^{-2}\ve  \big[ - \kappa_j\dot \xi_j + \nn _{\xi_j}^\perp K(\xi) \big ] + \ve^3 \nn_x^\perp \theta_{j  \ve} (\xi_j; \xi ) \nonumber \\
 %\\ &+\,  \kappa_j^{-2}\ve \big[-  \kappa_j\dot \xi_j^1  + (D^2_\xi K (\xi^0) [\xi^1] )_j^\perp \big ] + \ve^3 \nn_x^\perp \theta_{j  \ve} (\xi_j^0; \xi^0 ) \nonumber\\
&+   \sum_{k=2}^4 \frac {\ve^k }{k!} \nn _y ^\perp \big[  D_x^k\vp_j(\xi_j; \xi)  [y]^k \big]  +{\ve^4 \over 2}\nn _y ^\perp \big[ D_x^2 \theta_{j  \ve}  (\xi_j ; \xi)[y]^2 \big] \nonumber \\
%&  +\,  \frac {\ve^2}2 \nn _y ^\perp\big[ \,  D^{2}_{x} D_{\xi} \vp_j ( \xi^0_j ; \xi^0 )[y]^2[\xi^1]  +   D^{3}_{x}\vp_j( \xi_j^0 ; \xi^0 ) [y]^2[\xi_j^1] \, \big ] \nonumber\\
&+\mathcal Q(\ve y ; \xi)
\label{expR} \end{align}
 where
 $$
 |\mathcal Q(z ; \zeta )|\ \le\ C\ve\,  \, |z|^4 \, .
 $$ We recall that the basic assumption on $\xi^0$  is precisely that $ - \kappa_j\dot \xi_j^0 + \nn _{\xi_j}^\perp K(\xi^0) = 0 $, hence
 the first term in expansion \equ{expR} is actually of size $O(\ve^3)$.

\subsection{First improvement} We recall that $R= 6\delta\ve^{-1}$.
We will eliminate some of the terms in the first error. We let $y=\rho e^{i\theta}$ and compute
\begin{align*}
g^k(y,t;\xi)\  =&\  \frac 1{k!}   \nn _y ^\perp \big[  D_x^k\vp_j(\xi_j(t); \xi(t))  [y]^k \big] \cdot \nn U_0 (y)\\
=&  \ \frac 1{k!}   \frac {4U_0}{1+ \rho^2}  \frac{\pp}{\pp\theta} \,\big \{  D_x^k\vp_j(\xi_j(t); \xi(t))  [y]^k\big\}.
\end{align*}
We notice that the function $q_k(y,t)= D_x^k\vp_j(\xi_j(t); \xi(t))  [y]^k$ is a harmonic polynomial in $y$ since $\vp_j(x;\xi)$ is harmonic in $x$. Thus $$q_k(y,t) =   \rho^k (\alpha_k(t) \cos k\theta +\beta_k(t) \sin k\theta)$$
with $\alpha$ and $\beta$ $C^1$ functions. As a conclusion, the function $g^k(y,t)$ only has components in modes $k$ and $-k$ in its Fourier
expansion \equ{fourier}. We observe also that $$|g^k(y,t;\xi)| \le C (1+ \rho)^{-\alpha}, \quad \alpha = 6-k. $$
 We consider the cases $k= 2,3,4$.
By Lemma \ref{alpha} and Remark \ref{remark100} we find that there exists a unique solution $(\psi^{(k)} (y,t;\xi)),\phi^{(k)} (y,t;\xi))$ to the equation
$$
L[\psi^{(k)}  ]  + g^{(k)}(y,t; \xi) =0 \inn B_{8R} ,   \quad \psi^{(k)}  = 0 \onn \pp B_{8R} , \quad -\Delta \psi^{(k)}  = \phi^{(k)}
$$
that satisfies the estimates
\begin{align}
&|\psi^{(k)} (y,t)| \, +\,  (1+|y|)\, |\nn \psi^{(k)} (y,t)|\,+ \,  (1+ |y|^2)\, | \phi^{(k)} (y,t)|  \, |\nonumber
\\
& \quad \le   C(1+|y|)^{k-2}    .
\label {formula2}
\end{align}
A similar estimate is of course satisfied for the functions $(\pp_t \psi^{(k)} ,\pp_t \phi^{(k)})$.
In addition we consider the solutions $(\bar\psi_\ell^{(4)}, \bar\phi_\ell^{(4)})$, $\ell =1,2$  of the problems
$$
L[\bar\psi_1^{(4)}]  +    \kappa_j^{-1}  \pp_t \phi^{(2)}(y,t;\xi)   =0 \inn B_{8R} , \quad \bar\psi_1^{(4)} = 0 \onn \pp B_{8R}
$$
%(we take $\xi^0$ and not $\xi$ since we will need further differentiability in $t$ in a later argument)
and
\be\label{hh}
L[\bar\psi_2^{(4)}]  +     \frac 1{2}   \nn ^\perp \big[  D_x^2\vp_j(\xi_j; \xi)  [y]^2 \big] \cdot \nn \phi^{(2)}(y,t;\xi) =0 \inn B_{8R} , \quad \bar\psi^{(4)} = 0 \onn \pp B_{8R}  .
\ee
By construction, $ \pp_t \phi^{(2)}$ only carries modes $\pm 2$ in its Fourier expansion in $y$.
Besides $| \pp_t \phi^{(2)}(y,t)| \le C(1+|y|)^{-2}$, and hence $(\bar\psi_1^{(4)}, \bar\phi_1^{(4)})$ satisfies the
 estimate
$$
 \begin{aligned}
&|\bar\psi_1^{(4)} (y,t)| \, +\,  (1+|y|)\, |\nn \bar\psi_1^{(4)} (y,t)|\,+ \,  (1+ |y|^2)\, | \bar \phi_1^{(4)} (y,t)| \nonumber
\\
& \quad \le   C(1+|y|)^{2} \log \frac {16R}{ 1+ |y|    }     .
\end{aligned}
$$
To solve \equ{hh}, we claim that
$$
g(y,t;\xi) =  \frac 1{2}   \nn ^\perp \big[  D_x^2\vp_j(\xi_j; \xi)  [y]^2 \big] \cdot \nn \phi^{(2)}(y,t;\xi)
$$
carries only Fourier modes $\pm 4$ (so that $g_0\equiv 0$). Since  $|g(y,t)| \le C(1+|y|)^{-2}$, that implies that \equ{hh} is solved by a function
that again satisfies estimates \equ{formula2} with $k=4$.
To see the validity of the claim, we recall that
$$
D_x^2\vp_j(\xi_j; \xi)  [y]^2 =  \rho^2 (\alpha(t) \cos 2\theta +  \beta(t) \cos 2\theta).
$$
It follows that $$\psi^{(2)}(y,t)=  p(\rho) (\alpha(t) \cos 2\theta +  \beta(t) \cos 2\theta)$$ where $p(\rho)$ is the unique solution of
\equ{pk} with $k=2$ and $g_2(\rho)  =  -U_0(\rho)\rho^2 . $
Hence
$$\phi^{(2)}(y,t)=  q(\rho) (\alpha(t) \cos 2\theta +  \beta(t) \cos 2\theta)$$ for a certain smooth function $q(\rho)$. Then we compute
$$
\nn ^\perp \big[  D_x^2\vp_j(\xi_j; \xi)  [y]^2 \big] \cdot \nn \phi^{(2)}(y,t) = \rho q'(\rho) (\beta^2-\alpha^2) \sin 4\theta +  q(\rho) \alpha \beta \cos 4\theta,
$$
hence the Fourier expansion of this terms only involves modes $\pm 4$ as asserted.

\medskip

We let
$$
\begin{aligned}
% \label{barpsi1}
\bar \psi^1 (\cdot;\xi)  &= \ve^2 \psi^{(2)} + \ve^3 \psi^{(3)} +
\ve^4 \psi^{(4)} + \ve^4 \bar \psi_1^{(4)}+ \ve^4 \bar \psi_2^{(4)},
\\
\bar \phi^1(\cdot;\xi)   &= \ve^2 \phi^{(2)}  + \ve^3 \phi^{(3)}
+\ve^4 \phi^{(4)} +\ve^4 \bar \phi_1^{(4)}+ \ve^4 \bar \phi_2^{(4)} .
\end{aligned}
$$
We compute  the new error
\be
\kappa_j^{-1} E_j (\bar \phi^1 , \bar \psi^1 ,0, \xi) =  \bar E_j^1   + \bar E_j^2  + \bar E_j^3  +   \bar E_j^4,
\label{exp2}\ee
where the errors $\bar E_j ^l(y,t;\xi)$ are given by
\begin{align*}\bar E_j^1
& =
\left( \kappa_j^{-2}\ve  \big[ - \kappa_j\dot \xi_j + \nn _{\xi_j}^\perp K(\xi) \big ] +\ve^3 \nn_x^\perp \theta_{j  \ve} (\xi_j; \xi )\right)\cdot \nn U_0  ,
\\
\bar E_j^2 &= \ve^4 \nn^\perp \psi^{(2)} \cdot \nn \phi^{(2)}  ,
\\
\bar E_j^3 & =
 \ve^5 \kappa_j^{-1} \pp_t \phi^{(3)}
 \\
 & \ +
 \nabla^\perp \big [ \vp_j(\xi_j+\ve y; \xi) -\vp_j(\xi_j; \xi)- \ve  \nn \vp_j(\xi_j; \xi)\cdot y    \big ] \cdot \nabla (\bar\phi -\ve^2\phi^{(2)}  )
 \\ &\ + \mathcal Q(\ve y ; \xi )\cdot \nn (U_0 + \bar\phi^1)+  \ve^6\pp_t (\phi^{(4)}_1 +  \phi^{(4)}_2)
\\ &\ + \  \ve^2 \nabla^\perp \big [\vp_j(\xi_j+\ve y; \xi) -\vp_j(\xi_j; \xi)- \ve  \nn \vp_j(\xi_j; \xi)\cdot y - \frac 12 D_x^2\vp_j(\xi_j; \xi)  [y]^2    \big ]
\\
& \ \cdot \nabla \phi^{(2)} ,
\\\bar E_j^4 &=
\left( \kappa_j^{-2}\ve  \big[ - \kappa_j\dot \xi_j + \nn _{\xi_j}^\perp K(\xi) \big ] + \ve^3 \nn_x^\perp \theta_{j  \ve} (\xi_j; \xi )\right)\cdot \nn \bar\phi^1
\\ &\ + \nn^\perp (\bar\psi^1-\ve^2\psi^{(2)}    ) \cdot \nn \bar\phi^1 + \nn^\perp \bar\psi^1 \cdot \nn (\bar\phi^1- \ve^2 \phi^{(2)}  )
\end{align*}
%where
%$$
%A(\xi)(y,t) = \kappa_j^{-1}  \pp_t \phi^{(2)}(y,t;\xi) +  \frac 1{2}   \nn_y^\perp \big[  D_x^2\vp_j(\xi_j; \xi)  [y]^2 \big] \cdot \nn \phi^{(2)}(y,t;\xi).
%$$

The (directly checked) relevant characteristics of each of the above terms are the  following:
\begin{itemize}
\item   $\bar E_j^1$ is  $O( \ve^3 \rho^{-5} )$ at Fourier mode $\pm 1$.

\item  $\bar E_j^2$ is  $O( \ve^4 \rho^{-4} )$   at Fourier mode $\pm 4$.

\item $\bar E_j^3$ is  $O(\ve^5 \rho^{-1}  )$.

\item $\bar E_j^4$ is $O(\ve^5 \rho^{-3}    )$.

\end{itemize}

Our next step is the elimination of  the term  $\bar E_j^3$. To get this, rather than solving  an elliptic problem
 we solve the transport equation

\begin{align}
\left\{
\begin{aligned}
\ve^2 \phi_t
+ \nabla_y^\perp(\Gamma_0(y) +  \RR_j^0( y,t;\xi) ) \cdot \nabla_y \phi
& =  \bar E_j^3(y,t;\xi) ,
\quad \text{in } B_R\times [0,T]
\\
\phi(y,0) & =0 , \quad \text{in } B_R
\end{aligned}
\right.
\label{pico}\end{align}
Here $\RR^0_j( y,t;\xi)$ is a slight modification of the potential  $\RR_j (y,t;\xi)$ in \equ{defRa}.
We recall that we are choosing $\xi = \xi^0 + \xi^{1} + \ttt \xi$ as in \equ{xi}.
We take
\begin{align}
\label{defRa1}
\RR^0_j (y,t;\xi) = &\ve \kappa_j^{-1} (\dot\xi^0_j + \dot\xi_j^{1})^\perp \cdot y  +  \tilde \varphi_j(\xi_j+\ve y; \xi )  - \tilde \varphi_j (\xi_j ; \xi )\nonumber\\ =&\RR_j (y,t;\xi) - \ve \kappa_j^{-1} {\dot{\ttt\xi}_j}^\perp\cdot y .
\end{align}
The reason for this modification is that we will need uniform differentiability in $t$ of this coefficient.

%{
%We have $	 \ve \kappa_j^{-1} {\dot{\ttt\xi}_j}^\perp\cdot y = O (\varepsilon^{3+\sigma}) |\rho| )$. This appears multiplied by $\nabla \bar \phi^1$.
%}

%$$\RR^0_j( y,t) =   \vp_j(\xi_j(t)+ \ve y ; \xi(t)) -   \vp( \xi_j(t); \xi(t)  ) - \ve \nn_x\vp( \xi_j(t); \xi(t)  )\cdot y .$$

We consider a smooth cut-off function $\eta(s)$ as in \equ{eta}, and extend $\RR^0_j$ and $E_j^3$ to entire space by setting
$$
\ttt\RR_{0j} (y,t;\xi) =
\RR_{0j} (y,t;\xi) \eta \Big( \frac{|y|}{2R}\Big) ,
\quad
\ttt E_j^{30}(y,t)  =  \bar E_j^{3}(y,t;\xi_0) \eta \Big( \frac{|y|}{2R}\Big).
$$
Note that in $\ttt E_j^{30}$ we have frozen $\xi = \xi^0$.

We will then have a solution to \equ{pico} by restricting to $B_R$ the solution of the Cauchy problem
\begin{align}
\left\{
\begin{aligned}
\ve^2 \phi_t
+ \nabla_y^\perp(\Gamma_0(y) + \ttt \RR_j^0( y,t;\xi) ) \cdot \nabla_y \phi
& =\ttt E_j^{3}(y,t;\xi^0) ,
\quad \text{in } \R^2\times [0,T]
\\
\phi(y,0) & =0 , \quad \text{in } \R^2
\end{aligned}
\right.
\label{pico0}\end{align}
To solve \equ{pico0} we consider a slightly more general equation of that form and state estimates that will also be used later, whose proofs
we postpone to  \S~\ref{secProofTransportInner}.

\subsection{The inner transport equation}
\label{secInnerTransport}
We consider a transport equation of the form
\begin{align}
\left\{
\begin{aligned}
\ve^2 \phi_t
+ \nabla_y^\perp(\Gamma_0(y) +  \RR( y,t) ) \cdot \nabla_y \phi
& = E(y,t) ,
\quad \text{in } \R^2\times [0,T],
\\
\phi(y,0) & =0 , \quad \text{in } \R^2
\end{aligned}
\right.
\label{pico1-new}\end{align}
where $\RR$ is a perturbation term on which we assume that it is defined in $\R^2$ and $\RR(y,t)=0$ for $|y|\geq 4R$,
$R=\frac{4\delta}{\varepsilon}$ and the estimate
\be
|\nabla_y \RR( y,t) | \ \le \ M \ve^2 (1+|y|) .
\label{papa-new}\ee
Here $\delta>0$ is fixed.
It is convenient to let $t=\ve^2\tau $ and represent the equation for $\phi= \phi(y,\tau)$ as
\begin{align}
\left\{
\begin{aligned}
\phi_\tau
+ \nabla_y^\perp(\Gamma_0(y) +  \RR( y,\ve^2\tau ) ) \cdot \nabla_y \phi
& =E(y,\ve^2\tau ) ,
\quad \text{in } \R^2\times [0,\ve^{-2} T],
\\
\phi(y,0) & =0 , \quad \text{in } \R^2
\end{aligned}
\right.
\label{pico2}\end{align}
We solve \equ{pico1-new} by the method of characteristics.
For  definiteness of the characteristics, we
 also assume that $\nn_y \RR(y,t) $ is  log-Lipschitz in $y$ uniformly in $t$, that is,
\[
| \nn_y \RR(y_1,t) - \nn_y \RR(y_2,t) | \leq L |y_1-y_1| ( 1+ |\log|y_1-y_2| \, | )
\]
and continuous in its two variables, but no uniform estimate on the constant $L$ will for the moment be assumed ($L$ is allowed to depend on $\varepsilon$).

\medskip
The characteristic curve $\bar y(s;\tau,y)$  is by definition the solution
$\bar y (s)$ of the ODE system
\begin{align}
\label{characteristic1-new}
\left\{
\begin{aligned}
\frac{d \bar y }{d s} (s) & =
\nabla_y^\perp (\Gamma_0+\RR)(\bar y(s) , \varepsilon^2 s )
\\
 \bar y(\tau) &= y  .
\end{aligned}
\right.
\end{align}
This system has indeed a unique solution  defined on the entire interval $[0,\varepsilon^{-2} T]$, see for instance \cite{notes-krr}.
For a locally bounded function $E$, the unique solution of \equ{pico2} is then represented by the formula
\begin{align}
\label{phi}
\phi(y,\tau) =
\int_0^\tau E(\bar y(s;\tau,y) , \varepsilon^2 s)\,ds .
\end{align}

\begin{lemma}\label{transport1-new}
Let us assume the validity of $\equ{papa-new}$. Let $1\le p\le +\infty$ and $\alpha\in \R$. There exists a number $C>0$ such that
for any function  $E(y,t)$ that satisfies
%\be\label{cotaE}
\[
\sup_{t\in [0,T]}
\| (1+|\cdot |)^{-\alpha } E(\cdot , t)\|_{L^p(\R^2)}  \ <\ +\infty
%\ee
\]
we have that for all sufficiently small $\ve$, the solution of $\equ{pico1-new}$ satisfies
$$
\sup_{t\in [0,T]}
\| (1+|\cdot |)^{-\alpha } \phi (\cdot , t)\|_{L^p(\R^2)} \ \le \
C \ve^{-2} \sup_{t\in [0,T]}
\| (1+|\cdot |)^{-\alpha } E(\cdot , t)\|_{L^p(\R^2)}  .
$$
\end{lemma}

%\medskip
%A property that will be useful for the analysis of the outer problem is the fact that if the spacial support of the function $E$ stays
%at a uniform large distance of the origin then so does the solution of \equ{pico1-new}.
%
%\begin{lemma}\label{transport2}
%There exist numbers $R_0>0$, $\beta>0$  such that for any sufficiently small $\ve$, all $R>R_0$ and any locally bounded function $E$ such that
%$$E(y,t) = 0 \foral (y,t) \in B_{R}(0)\times [0,T] $$
%we have that the solution of $\equ{pico1-new}$ satisfies
%$$\phi(y,t) = 0 \foral (y,t) \in B_{\beta  R}(0)\times [0,T]. $$
%\end{lemma}
%For the proof see \S~\ref{secProofTransportInner}.

\medskip

\noindent{\bf Gradient estimates.}
It is natural to think that an estimate for the gradient of $E(y,t)$ leads to such an estimate for $\phi(y,t)$.
This will be true under some additional assumptions on $\RR(y,t)$.
Let us assume that the function $\RR(y,t)$ satisfies the following bounds
\begin{align}
\nonumber
 |\RR(y,t)| +  |\partial_t \RR(y,t)|  \ \le  & \ C \varepsilon^2 (1+|y|^2),
\\
%\label{boundDyR}
|\nn_y \RR(y,t)|+ |\partial_t \nn_y \RR(y,t)|  \ \le  & \ C \varepsilon^2 (1+|y|),
\nonumber\\
\label{pico3-new}
| D^2_y \RR(y,t)| \ \le  &\  C \varepsilon^2.
\end{align}
We find a corresponding estimate for the gradient. %$|y|\le \delta\ve^{-1}$ where $\delta>0$ is fixed and sufficiently small.
\begin{lemma}\label{transport3-new}
Let us assume that $\RR$ satisfies $\equ{pico3-new}$.
Then there exist numbers $C,\delta>0$ such that for all sufficiently small $\ve$  and any function $E(y,t)$ that satisfies for some $A,\alpha\in \R$
%\be\label{cotin}
\[
(1+|y|) |\nn_y E (y,t)| + | E (y,t)| +|E_t(y,t)| \ \le \ A\, (1+ |y|)^\alpha ,
%\ee
\]
the solution of $\equ{pico1-new}$ satisfies
\be\label{cotin1}
(1+|y|)|\nn_y \phi(y, t)    |  + |\phi(y, t)| \  \le\  C \ve^{-2} A (1+|y|)^\alpha \foral y\in \R^2, \ |y|< \delta\ve^{-1}
\ee
\end{lemma}
For the proof see \S~\ref{secProofTransportInnerGradient}.

%It follows that
%\begin{align*}
%|\frac{d}{ds} H(\bar y(s),s) | \leq C \varepsilon^4 (   |\bar y(s)|^2 + 1)

\subsection{Second improvement of the approximation}
We let $\bar\phi^{2} (y,t;\xi)$  be the solution of  Problem \equ{pico0}. Thanks to Lemmas~\ref{transport1-new} and \ref{transport3-new},  $\bar\phi^{2} $ satisfies
$$
|\bar\phi^{2} (y,t)|  + (1+ |y|) |\nn_y\bar\phi^{2} (y,t)| \ \le\ C \frac{\ve^3}{ 1+|y|}.
$$
We remark that the hypothesis \eqref{pico3-new} holds since at the functions involved in the definition \eqref{defRa1}  are smooth and of the correct order in $\varepsilon$.

\medskip

Let us consider the unique solution $\bar\psi^{2}(y,t)$ of
$$
\begin{aligned}
-\Delta_y \bar\psi^{2}  = \bar\phi^{2} \inn B_{4R}, \quad
\bar\psi^{2}  = 0 \onn \pp B_{4R} .
\end{aligned}
$$
By standard elliptic theory, we get the following estimate in  $\bar\psi^{2}$.
$$
|\bar\psi^{2} (y,t)|  + (1+ |y|) |\nn_y\bar\psi^{2} (y,t)|+ (1+ |y|^2) |D^2_y\bar\psi^{2} (y,t)| \le C\ve^3\,(1+|y|)\,\log\frac {16R}{1+|y|} .
$$

Introducing back the subindex $j$ we then let
\be\label{barphi}\begin{aligned}
\bar \phi_j(y,t;\xi) = & \bar\phi^1_j(y,t;\xi)+ \bar\phi^2_j(y,t;\xi) , \\ \bar \psi_j(y,t;\xi) = & \bar\psi^1_j(y,t;\xi)+ \bar\psi^2_j(y,t;\xi) .
\end{aligned}
\ee
We compute the associated error
\begin{align}
\kappa_j ^{-1} E_{0j} (\bar\phi_j , \bar\psi_j, 0;\xi)
\nonumber &  =
\kappa_j ^{-1} E_{0j} (\bar\phi^1_j, \bar\psi^1_j, 0 ;\xi)
- \bar E_j^{3}(\cdot;\xi^0)
\, -\, \ve \dot{\ttt\xi}\cdot   \nabla \bar\phi^2_j
\\
\nonumber
& \quad + \nn^\perp \bar\psi^2_j\cdot \nn U_0
+   \nabla^\perp \bar \psi^1_j \cdot \nabla \bar \phi^2_j
\\
& \quad +   \nabla^\perp \bar \psi^2_j \cdot \nabla \bar \phi_j^1 + \nabla^\perp \bar \psi_j^2 \cdot \nabla \bar \phi_j^2 .
\label{exp3}
\end{align}
The terms in the last two rows of the above expansion are all of the order $O(\ve^5\rho^{-3}\log\ve )$. Observe also that
$\ve \dot{\ttt\xi} \cdot   \nabla \bar\phi^2 = O (\ve^{6+\sigma}\rho^{-2})$.
Therefore from \eqref{exp3} we get
\[
\kappa_j ^{-1} E_{0j} (\bar\phi_j , \bar\psi_j, 0;\xi) = \kappa_j ^{-1} E_{0j} (\bar\phi^1_j, \bar\psi^1_j, 0 ;\xi) - \bar E_j^3(\cdot;\xi^0)  + \nn^\perp \bar\psi^2_j\cdot \nn U_0+ O(\ve^5\rho^{-3}\log\ve   ).
\]

\begin{remark}
	%\label{suave}
{\em
For a later step in the construction it is useful to point out that the directional derivative $\pp_\xi \phi_{j} $ given by
$$
\pp_\xi \bar \phi_j(y,t ,\xi) [\zeta]\, :=\,  \frac {d}{d s} \bar\phi_j(y,t, \xi + s\zeta)\big |_{s=0}
$$
has uniform bounds. Indeed, we recall that $\bar \phi_j = \bar\phi_j^1 +\bar\phi_j^2 $. From the definition of $\bar\phi_j^1$ we clearly have
$$
\big| \pp _\xi \bar \phi^1_j(y,t,\xi) [\zeta]\big| \ \le  \  C\frac{\ve^2}{1+|y|^{2}} \|\zeta\|_{C^1[0,T]},
$$
while $\Phi (y,t) =  \pp _\xi \bar \phi^2_j(y,t,\xi) [\zeta]$ solves the transport equation
$$
\left\{
\begin{aligned}
\ve^2 \Phi_t
+ \nabla_y^\perp(\Gamma_0(y) + \ttt \RR_j^0( y,t;\xi) ) \cdot \nabla_y \Phi + E(y,t)
& = 0 ,
\quad \text{in } \R^2\times [0,T]
\\
\phi(y,0) & =0 , \quad \text{in } \R^2
\end{aligned}
\right.
$$
where
$$
 E (y,t) = \nabla_y^\perp( \pp_\xi \ttt \RR_j^0( y,t;\xi)[\zeta] ) \cdot \nabla_y \bar \phi^1_j(y,t; \xi)\ = \ O(\ve^5 \rho^{-1})\,\|\zeta\|_{C^1[0,T]}.
$$
Using Lemma \ref{transport1-new} we find that
\[
\pp _\xi \bar \phi^2_j(y,t,\xi) [\zeta]
=
\Phi (y,t) = O(\ve^3 \rho^{-1})\, \|\zeta\|_{C^1[0,T]} .
\]
As a result we get
\[
\big| \pp _\xi \bar \phi_j(y,t;\xi) [\zeta]\big| \ \le  \  C\frac{\ve^2}{1+|y|^{2}} \|\zeta\|_{C^1[0,T]}
\]
The previous estimates and the definition of $\bar \psi^1_j$ lead to
\begin{align*}
\nonumber
& \big| \pp _\xi \bar \psi^1_j(y,t;\xi) [\zeta]\big|
+ (1+ |y|) \big|\nn_y \partial_\xi \bar \psi^1_j(y,t;\xi) [\zeta]\big|
+ (1+ |y|^2) \big| D^2_y \partial_\xi \bar \psi^1_j(y,t;\xi) [\zeta]\big|
\\
& \quad \leq
C\ve^2\,
\big (\,\|\zeta\|_\infty + \|\dot\zeta\|_\infty\big) .
\end{align*}
%{  De donde sali\'o el $\log^2R$? }
Similarly, for  $\bar \psi^2$, we get
\begin{align*}
& \big| \pp _\xi \bar \psi^2_j(y,t;\xi) [\zeta]\big|
+ (1+ |y|) \big|\nn_y \partial_\xi \bar \psi^2_j(y,t;\xi) [\zeta]\big|
+ (1+ |y|^2) \big| D^2_y \partial_\xi \bar \psi^2_j(y,t;\xi) [\zeta]\big|
\\
& \quad \leq
C\ve^3\,(1+|y|)\,\log\frac {16R}{1+|y|}
 \big (\,\|\zeta\|_\infty + \|\dot\zeta\|_\infty\big) .
\end{align*}
Therefore, for $\bar \psi_j = \bar\psi_j^1 + \bar\psi_j^2$, we get
\begin{align}
\nonumber
& \big| \pp _\xi \bar \psi_j(y,t;\xi) [\zeta]\big|
+ (1+ |y|) \big|\nn_y \partial_\xi \bar \psi_j(y,t;\xi) [\zeta]\big|
+ (1+ |y|^2) \big| D^2_y \partial_\xi \bar \psi_j(y,t;\xi) [\zeta]\big|
\\
\label{derBarPsij}
& \quad \leq
C\ve^2 \,\log\frac {16R}{1+|y|}
\big (\,\|\zeta\|_\infty + \|\dot\zeta\|_\infty\big) .
\end{align}

}

%{  no creo que lleve logs cuando $|y|\sim R$ }
\end{remark}

Next we shall improve the approximation error in the outer problem, considering the values \equ{barphi} for the functions $\phi_j$, $\psi_j$.

\subsection{The outer approximation}
\label{secOuter}
We will improve the outer errors
\be\label{E0}
\begin{aligned}
  E_0(x,t;\xi) \,  &:= \, E_0^{out}(\bar\phi^{in}, \bar\psi^{in},0,0,\xi )\, \\
& =   \ve^{-2}  \sum_{j=1}^N \kappa_j\, \pp_t\eta_{1j}\bar \phi_j  -\,\ve^{-2} \sum_{j=1}^N  \kappa_j(1-\eta_{1j} )  \ve^{-1}  \dot{\xi}_j \cdot  \nn_x U_{0j}
\\
& \quad +  \ve^{-2}  \sum_{j=1}^N \kappa_j\, [ \bar \phi_j  \nn_x\eta_{1j} +   (1-\eta_{1j})  \nn_x U_{0j})\, ]\, \cdot   \nn^\perp \Psi_0  , \\
&= O (   \ve^2 )\, , \\
E_{20}(x,t;\xi) \,  &:= \, E_{20}^{out}( \bar\psi^{in},0,0,\xi )  \\
 &:= \, \sum_{j=1}^N \kappa_j\,\big [\,\bar \psi_j\Delta_x\eta_{2j} + 2\nn_x \eta_{2j}\cdot \nn_x \bar \psi_j \, \big ] \,
 \\ &= O (   \ve^2 )\, .
\end{aligned}
\ee
%\begin{align}
%\nonumber
%E^{out}_1( \phi,\psi, \phi^{out},\psi^{out} ,\xi  )
%&  =
%\phi^{out}_t    +\,  \nn^\perp _x\big (\Psi_0 (\cdot; \xi)+ \sum_{j=1}^N \eta_{2j}\psi_j + \psi^{out} \big )  \cdot \nn \phi^{out}\\\nonumber& \quad +\ \ve^{-2}  \sum_{j=1}^N \kappa_j\, \pp_t\eta_{1j}\phi_j  -\,\ve^{-2} \sum_{j=1}^N  \kappa_j(1-\eta_{1j} )  \ve^{-1}  \dot{\xi}_j \cdot  \nn_x U_{0j}\\\nonumber& \quad +\,\ve^{-2}  \sum_{j=1}^N \kappa_j\, [ \phi_j  \nn_x\eta_{1j} +  (1-\eta_{1j})  \nn_x U_{0j})\, ]\,  \\&\quad \cdot
    %\nn^\perp \big (\Psi_0 (\cdot; \xi)  + \sum_{j=1}^N \eta_{2j} \psi_j + \psi^{out} \big ) \inn \Omega\times [0,T], \label{E1}\end{align} \be   E_2^{out}( \psi, \phi^{out},\psi^{out} ,\xi  )\ = \=    \Delta_x  \psi^{out}    +\phi^{out}+  \sum_{j=1}^N \kappa_j ( \psi_\Delta_x\eta_{2j}  + 2 \nn_x \eta_{2j} \cdot \nn_x   \psi_j)\label{E2}\ee$$\psi^{out} = - \Psi_0  \onn \pp\Omega\times [0,T]$$
First we consider the solution of the elliptic equation
\begin{align}
\label{psi1out}
\Delta_x\psi^{out}_1(\cdot ;\xi)   = 0 \inn \Omega\times [0,T],\quad
\psi_1^{out} = - \Psi_0 (\cdot, \xi) \onn \pp\Omega\times [0,T].
\end{align}
which we readily see is of size $O(\ve^2)$ in $C^2$-topology uniformly in $t$.
Let us consider the transport equation

\be
\left \{
\begin{aligned}
  \phi^{out}_t    +\,  \nn^\perp _x (\Psi_0  +\psi_1^{out})  \cdot \nn \phi^{out}  + E_0(x,t;\xi)\ = &\ 0 \inn \Omega \times [0,T],  \\
  \phi(\cdot ,0)\ = &\ 0 \inn \Omega.\end{aligned}\right.
  \label{pico5} \ee
We solve \equ{pico5} by considering the transport equation in a slightly more general setting and state some preliminary results that we will use to conclude the construction and also later.

 \subsection{The outer transport equation}
 We consider a transport equation of the form
\begin{align}
\left\{
\begin{aligned}
\phi_t  + \nn_x^\perp(\Psi_0(\cdot;\xi) + \QQ )\cdot \nn \phi  =  &E(x,t)
\quad \text{in } \Omega\times [0,T],
\\
\phi(x,0) & =0 , \quad \text{in } \Omega
\end{aligned}
\right.
\label{pico3}\end{align}
We assume that the function  $\nn_x \QQ(x,t)$ is continuous and log-Lipschitz in $x$ uniformly in $t$, and that $E\in L^\infty(\Omega\times [0,T])$.
In addition, we assume that
\be\label{cero}
\Psi_0(x;\xi(t)) + \QQ(x,t) = 0\foral (x,t)\in \pp\Omega \times [0,T].
\ee
Again we can represent the solution of \equ{pico3} by means of Duhamel's principle
\be \label{phi3}
\phi(x,t)  =  \int_0^t  E( \bar x(s; t, x) ,  s)\, ds
\ee
where the characteristics  $\bar x(s)= \bar x(s; t, x) $ correspond to the solution of
\begin{align*}
%\label{characteristic2}
\left\{
\begin{aligned}
\frac{d \bar x }{d s} (s) & =
\nabla_x^\perp (\Psi_0 + \QQ)(\bar x(s) , s ), \quad s\in [0,t],
\\
 \bar x(t) &= x  ,
\end{aligned}
\right.
\end{align*}
which exist and are unique

\begin{lemma}\label{int1}
Under the above assumptions we have that for $x\in \Omega$ the characteristics satisfy
$\bar x(s; t, x)\in \Omega$ for all $0\le s\le t$.  The solution of $\equ{pico3}$ given by $\equ{phi3}$ satisfies for any $1\le p\le +\infty$,
\begin{equation*}
%\label{aaa}
\|\phi(\cdot ,t)\|_{L^p(\Omega)}  \le     t\sup_{0\le s\le t} \|E(\cdot, s)\|_{L^p (\Omega)} .
\end{equation*}
\end{lemma}

\medskip
The context that we need to consider in the outer problem \equ{pico3} is that of a right hand side $E$ that its supported away from the vortices $\xi_j(t)$. Let us assume that for some fixed number $\delta$ we have
\be
E\equiv 0  \inn  \Big \{ (x,t) \in \Omega \times [0,T] \ /\  x\in  \bigcup_{j=1}^N B_\delta (\xi_j^0(t))   \Big \}.
\label{assE}\ee
We make the following assumption on the perturbation term $\QQ$. For a constant $M>0$ we have
\be
|\nn \QQ(x,t)| \ \le \  M \sum_{j=1}^N (|x-\xi_j(t)|+\ve) \foral (x,t) \in \Omega\times [0,T].
\label{assQ}\ee

\begin{lemma} \label{transport5}
Assume that $\equ{assQ}$ holds. Then
there exists a number $\beta>0$ independent of  $\delta$ and $\ve $ such that if $E$ satisfies $\equ{assE}$ then the solution of $\equ{pico3}$
satisfies
\[
\phi \equiv 0  \inn  \Big \{ (x,t) \in \Omega \times [0,T] \ /\  x\in  \bigcup_{j=1}^N B_{\beta \delta} (\xi_j^0(t))   \Big \}.
\]
\end{lemma}

Next we will get a gradient estimate under some further assumptions.
Let us assume that the function $\QQ$ satisfies
\be\label{MM}
 |D^2_x \QQ(x,t) | % + |\nn \QQ_t(x,t) | + |\QQ_t(x,t) |
 \le M \foral (x,t) \in \Omega \times [0,T].
\ee

\begin{lemma}\label{transport4}
Let us assume that $\QQ(x,t)$ satisfies $\equ{assQ}$ and  $\equ{MM}$.
There exists a constant $C>0$ such that for any $E$ satisfying $\equ{assE}$
and
$$
|\nn_x E(x,t) | + |E(x,t) |  \le A \foral (x,t) \in \Omega \times [0,T],
$$
 the solution of $\equ{pico3}$ satisfies
the estimate
\be\label{poto1} |\nn_x \phi(x,t)| + |\phi_t(x,t)|  + |\phi(x,t) |\ \le\ C\, A. \ee
\end{lemma}
For the proof of Lemmas~\ref{int1}, \ref{transport5} and \ref{transport4} see \S~\ref{secProofTransportOuter}.

\subsection{Improvement of the outer approximation}
 Expressed in the variable $x$, we have
$$\bar\phi_j(x,t) =  \bar\phi_j \left (y , t \right ), \quad y= \frac{x-\xi_j(t)}\ve .$$
By construction, we immediately check that
$$
|\bar \phi_j(x,t)| + |\nn_x\bar \phi_j(x,t)|    \  =\  O(\ve^4)
$$
 in the region where $\nn \eta_j$ is supported. In fact, we check that, globally
 $$
 |E_0(x,t;\xi)| + |\nn_x E_0(x,t)| \  =\  O(\ve^2) .
 $$
 Since the support of $E_0$ is away from the vortex points $\xi_j(t)$ (see \eqref{E0}), so is the case for the solution $\phi^{out}_1(x,t)$ of \equ{pico5}
 and satisfies
\begin{align}
\label{estPhi1Out}
 |\phi^{out}_1(x,t)| + |\nn_x \phi^{out}_1(x,t)| \  =\  O(\ve^2) .
\end{align}
These facts follow from Lemmas \ref{transport5} and \ref{transport4} in the next section.
Now we solve
\begin{align}
\label{psi2out}
\left\{
\begin{aligned}
- \Delta_x  \psi_2^{out}    & =
\phi_1^{out} + E_{20}(x,t;\xi)
\inn \Omega\times [0,T], \\
\psi_2^{out} & =   0 \onn \pp\Omega\times [0,T].
\end{aligned}
\right.
\end{align}
with $E_{20}(x,t;\xi)$ as in \equ{E0}.
By construction both $\bar\psi_j$ and $\nn \phi_j$ are of size $O(\ve^2)$, so are their gradients in $x$. It follows that $\bar\psi_2$ are of class
$C^{2,\alpha}$ in space variable, with corresponding norm $O(\ve^2)$.
Let us define
$$
\phi^{out*} :=   \phi^{out}_1, \quad \psi^{out*}:=  \psi_1^{out}+ \psi_2^{out}.
$$
Now we observe that
\begin{align*}
& E^{out}_0( \bar\phi^{in},  \bar \psi^{in}, \phi^{out*},\psi^{out*} ,\xi  )
\\
&  \quad =
    \ve^{-2}  \sum_{j=1}^N \kappa_j\, [ \bar \phi_j  \nn_x\eta_{1j} +   (1-\eta_{1j})  \nn_x U_{0j})\, ]\, \cdot \,
    \nn^\perp \big ( \sum_{j=1}^N \eta_{2j} \bar \psi_j + \psi^{out*} \big )
\\
& \qquad + \nn^\perp _x\big (  \sum_{j=1}^N \eta_{2j}\bar \psi_j + \psi_2^{out} \big )  \cdot \nn \phi_1^{out}
\, =\,  O(\ve^4) ,  \\
& E^{out}_{20}(\bar \psi^{in}, \phi^{out*},\psi^{out*} ,\xi  )\, =\,  0.
\end{align*}
and we have therefore obtained an improvement in $\ve^2$ for both error sizes.

\medskip
%We will need the dependence on $\xi$ of the function  $\nabla_x\psi^{out*}$.
%  which we express in terms of the function
% \[
%D(y,t;\xi)  = \nabla_x\psi^{out*} ( \xi_j + \varepsilon y , t ;\xi) .
% \]
%\begin{lemma}
%We have
%\[
%|  \nabla_x\psi^{out*} ( \xi_j  , t ;\xi)  - \nabla_x\psi^{out*} ( \xi_j^0  , t ;\xi^0) |
%\lesssim \| \dot \xi - \dot \xi^0\|_{L^\infty(0,T)} +  \|  \xi -  \xi^0\|_{L^\infty(0,T)}  .
%\]
%\end{lemma}
We will need a bound on $\partial_\xi \nabla_x \psi^{out*}$ given by
\[
\partial_\xi \nabla_x \psi^{out*}(x,t;\xi)[\zeta]
=\frac{d}{ds} \nabla_x \psi^{out*}(x,t;\xi +s\zeta) \big |_{s=0}  .
\]
Recall that $\psi^{out*} = \psi_1^{out} + \psi_2^{out} $.
From the definition \eqref{psistar} of $\Psi_0 $  we readily see that
$$ \partial_\xi \Psi_0  (x,t;\xi)[\zeta])(x,t) \ = \ O(\varepsilon^2 |\zeta(t)|) $$
uniformly on $t\in [0,T]$, and  in the $C^{3,\alpha}$-sense on  $x\in \partial \Omega$. Hence
\[
\partial_\xi \nabla_x \psi^{out}_1 (\cdot,t;\xi)[\zeta](x,t)= O ( \varepsilon^2 |\zeta(t)|)
\]
in the $C^{2,\alpha}$ in space, uniformly in $t$, thanks to equation \eqref{psi1out}.

\medskip
To estimate $\partial_\xi  \psi^{out}_2(x,t)$ we let  $ \Phi (x,t) := \partial _\xi \phi_1^{out}   (x,t;\xi)[\zeta]$.
From \eqref{pico5}, $\Phi (x,t)$   solves the transport equation
\[
\begin{aligned}
\Phi_t    + \nabla_x^\perp  (\Psi_0  +\psi_1^{out})  \cdot \nabla_x \Phi + E(x,t)
& = 0 \inn \Omega \times [0,T],
\\
\Phi(\cdot ,0) & = 0 \inn \Omega,
\end{aligned}
\]
where
\begin{align*}
E(x,t) &=
\partial_\xi E_0(x,t;\xi)[\zeta]
+ \nabla_x^\perp ( \, \partial_\xi \psi_1^{out} (x,t;\xi)[\zeta]  \, ) \cdot \nabla_x \phi_1^{out} (x,t;\xi)
\end{align*}
with $E_0(x,t)$ given in \eqref{E0}. From this formula and \eqref{estPhi1Out} we get that
$$E(x,t) = O(\varepsilon^2 \| \zeta \|_{C^1[0,T]}    ). $$
Using Lemma~\ref{int1} we then get
\[
%\big| \partial _\xi \phi_1^{out}   (x,t,\xi)[\zeta] \big|  =
|\Phi (x,t)|
\leq C \varepsilon^2  \| \zeta \|_{C^1[0,T]}  .
\]
Combining this estimate with \eqref{psi2out} and \eqref{derBarPsij}, we obtain
\[ %
|\partial_\xi  \psi^{out}_2(x,t;\xi)[\zeta]= O ( \varepsilon^2 ) \|\zeta\|_{C^1[0,T]} ,
\]
in $C^{2}$ norm in $x$, in a small fixed neighborhood of the points $\xi_j(t)$, uniformly in $t$.
Here we use Lemma~\ref{transport5} and  equation \eqref{psi2out}, which imply that $\psi_2^{out}(\cdot,t;\xi)$ is harmonic in some fixed neighborhood of the points $\xi_j(t)$. We conclude that
\begin{align}
\label{derPsiOutStar}%
\big |\partial_\xi  \nn_x\psi^{out*} (x,t;\xi)[\zeta]\,\big | \,+\, \big | \partial_\xi  \psi^{out*} (x,t;\xi)[\zeta]\big | \, \le \, C\varepsilon^2  \|\zeta\|_{C^1[0,T]}. ,
\end{align}
in the $C^{1,\alpha}(\bar\Omega)$-sense uniformly in $t$. Moreover, in the
in $C^{2}(\bar\Omega)$-sense norm in some fixed neighborhood of the points $\xi_j(t)$, uniformly in $t$.

 \subsection{Conclusion of the proof of Proposition \ref{aproxi}}
 Let us consider the effect of the function $\psi^{out*}$ in the inner problems.
We compute
\begin{align*}
&
\kappa_j ^{-1} E_{0j} (\bar\phi_j +\phi, \bar\psi_j + \psi, \psi^{out*};\xi)   \\
& = \nabla ^\perp  \Gamma_ 0 \cdot   \nabla \phi + \nabla^\perp    \psi  \cdot \nabla U_0\,\\
&
\qquad +\,   \kappa_j ^{-1} E_{0j} (\bar\phi_j, \bar\psi_j, 0 ;\xi)\, +\,
\kappa_j^{-1}\nn^\perp \psi^{out}\nn U_0 + \kappa_j^{-1}\nn^\perp
\psi^{out*}\nn\bar   \phi_j \\
&\qquad  +\nabla^\perp (\RR_j +\bar\psi) \cdot \nabla \phi \nonumber \\
& \qquad  +   \nabla^\perp \bar \psi \cdot \nabla \phi+   \nabla^\perp \psi \cdot \nabla \bar \phi +  \nabla^\perp \psi \cdot \nabla \phi  + \kappa_j ^{-1}\ve^2 \partial_t  \phi \nonumber \\
& \qquad  +    \kappa_j^{-1}\nn^\perp \psi^{out*}\nn  \phi.
\end{align*}
The largest new error term created is
$\nn_y^\perp \psi^{out*}\nn U_0 = O( \ve^3 \rho^{-5}) $. Next we will produce $\psi$, $\phi$ that eliminate the cubic and quartic terms in
$\ve$ in the full error by solving the corresponding elliptic equation of the form \equ{linear1}. We will be able to do so after a convenient choice of the function $\xi^{1}(t)$ in \equ{xi}, leaving the remainder $\ttt \xi$ as a free parameter.  Taking into account expansions \equ{exp2} and \equ{exp3}
we collect those terms  and solve the elliptic equation
\be \begin{aligned}
&\nabla ^\perp  \Gamma_ 0 \cdot   \nabla \phi + \nabla^\perp    \psi  \cdot \nabla U_0 \,+\,  g(y,t) \,=\, 0\inn B_{4R},
\\& -\Delta \psi \,=\, \phi,\inn B_R, \quad \psi = 0 \onn \pp B_{4R}  .
\end{aligned}
\label{poto100}\ee
where we choose
\begin{align*}
g(y,t) & =    \kappa_j^{-2}\ve  \big[ - \kappa_j\dot \xi^{1}_j + \nn _{\xi_j}^\perp K(\xi^0 + \xi^{1})- \nn _{\xi_j}^\perp K(\xi^0) \big ]\cdot \nn U_0
\\ & \quad +
\big[  \ve^3 \nn_x^\perp \theta_{j \ve} (\xi^0_j; \xi^0 )
+ \nn_y^\perp( \bar\psi^2(y,t; \xi^0 ) + \psi^{out*}(\xi_j^0(t) + \ve y, t;\xi^0 ) )\big ] \cdot \nn U_0(y)
\\
& \quad +  \ve^4 \nn^\perp _y\psi^{(2)} (y,t; \xi^0 ) \cdot \nn_y \phi^{(2)} (y,t; \xi^0 ).
\end{align*}
We recall that the last term in the above expression only carries Fourier modes $\pm 4$ and it is of size $O(\ve^4 \rho^{-4})$. The other terms do not involve mode zero.

The solvability condition \eqref{ort2} needed  for $g$ amounts to a system of differential equations for $\xi^1$ of the form
\begin{align}
\label{systemXi1}
- \kappa_j\dot \xi^{1}_j(t) + \nn _{\xi_j}^\perp K(\xi^0(t) + \xi^{1}(t))- \nn _{\xi_j}^\perp K(\xi^0(t))
= \mathcal B_j(t), \quad t\in [0,T].
\end{align}
where $\mathcal B_j(t) $ is a smooth function  in  $[0,T]$ (independent of $\ttt \xi$) such that
$$\|\mathcal B\|_{C^1[0,T]}   =O (\varepsilon^2\log\ve ). $$
We let $\xi^1(t)$ be the unique solution of system \equ{systemXi1} such that $\xi^1(0)= 0 $. It is directly checked that
$$
\|\xi^1 \|_{C^2[0,T]} \ \le \ C \ve^2|\log\ve| .
$$
%\begin{align*}
%| \mathcal B ( \mathtt{x}_1 ; \xi^1 ) - \mathcal B ( \mathtt{x}_2 ; \xi^1 ) |
%& \lesssim \varepsilon^2 |\mathtt{x}_1-\mathtt{x}_2|
%\\
%| \mathcal B ( \mathtt{x} ; \xi^1 ) - \mathcal B ( \mathtt{x}_2 ; \xi^2 ) |
%& \lesssim \varepsilon^2 |\log \varepsilon|^2 \| \xi^1- \xi^2\|_{C^1[0,T]} ,
%\end{align*}
Using Lemma \ref{alpha}, we find a solution $(\bar \phi^3_j ,\bar \psi^3_j)$ of equation \equ{poto100} which  satisfies
\begin{align*}
& |\bar \psi^3_j(y,t)| \, +\,  (1+|y|)\, |\nn_y \bar \psi^3_j(y,t)|\,  + \,  (1+ |y|^2)\, | \bar \phi^3_j(y,t)| %\, + \, (1+ |y|^4)\, |\nn^\perp \Gamma_0 \cdot \nn \phi(y) |
\nonumber
\\
& \le \  C\,\ve^3|\log \ve|^2 {(1+|y|)^{-1}}
 \end{align*}
At last we define
$$
\phi_j^*(y,t;\xi ) = \bar\phi_j(y,t;\xi ) + \bar\phi^3_j(y,t ) , \quad \psi_j^*(y,t;\xi ) = \bar\psi_j(y,t;\xi ) + \bar\psi^3_j(y,t ).
$$
We get
\be \label{E2*}
\begin{aligned}
E_{2*}^{out} := & E_{20}^{out}( \psi^{in *}, \phi^{out*},\psi^{out*} ,\xi  )\
\\=& \
     \sum_{j=1}^N \kappa_j ( \bar \psi_j^3 \Delta_x\eta_{2j}  + 2 \nn_x \eta_{2j} \cdot \nn_x  \bar \psi_j^3 ) \, = \, O(\ve^4 |\log \ve |^2 \,).
      \end{aligned}
      \ee
Similarly we check
\begin{align*}
&E_j (\psi^*_j,\phi^*_j ,\psi^{*out},\xi) = \\
 &  \ve  [ -\dot {\ttt \xi} + \nn_{\xi_j} ^\perp K(\xi^0+\xi^1+\ttt \xi) -   \nn_{\xi_j} ^\perp K(\xi^0+\xi^1)] \cdot \nn U_0  + E^*_j (\ttt \xi )
\end{align*}
where
\begin{align*}
 \big|\,E^*_j (\ttt \xi )(y,t) \big| \ \le &\ C \ve^5|\log\ve|^2 (1 + |y|)^{-3} .
\end{align*}
We also observe that
$$
E^{out}_0( \phi^{in *},\psi^{in *}, \phi^{out *},\psi^{out *} ,\xi  )\, = \, O(\ve^4 |\log \ve |^2\, ).
$$
The proof of the proposition is concluded. \qed

\medskip
As for Lipschitz estimates for the errors, we directly obtain the following

\begin{corollary}
For the approximation constructed in Proposition \ref{aproxi}, we have the error Lipschitz estimates
\begin{align*}
  \big |\,\pp_{\ttt\xi} E_{0j} (\ttt \xi )[\zeta] (y,t)\big| \ \le &\ C \ve^5|\log\ve|^2\|\zeta\|_{C^1[0,T]} (1 + |y|)^{-3}\|\zeta\|_{C^1[0,T]} ,
\end{align*}
and
\begin{align*}
\big|\, \pp_{\ttt\xi}E^{out}_* (\ttt\xi) [\zeta](x,t) \big|\, +\, \big|\, \pp_{\ttt \xi}E^{out}_{2*}(\ttt\xi) [\zeta](x,t) \big|\ \le & \ C \ve^4 |\log\ve|^2\|\zeta\|_{C^1[0,T]}.
\end{align*}

\end{corollary}
\begin{proof}
 The proof follows by a straightforward verification term by term of the error, using
 estimates   \equ{derBarPsij} and \equ{derPsiOutStar}.
 \end{proof}

\section{Proof of the estimates for transport equations}

In this section we prove the results stated  in \S ~\ref{secInnerTransport} and
\S~\ref{secOuter}.

%deal with the solution of the initial value problem for the transport equation naturally associated to our problem, both in the inner and outer regimes.

\subsection{The inner transport equation}
\label{secProofTransportInner}

\begin{proof}[Proof of Lemma~\ref{transport1-new}]

We let $\bar y(s)= \bar y(s;\tau,y)$ be the unique solution of  System \eqref{characteristic1-new}.
We readily see that
$$
  \frac d{ds} \frac {|\bar y(s)|^2}2 =    \bar y(s)\cdot  (\nn^\perp_y\RR) (\bar y(s), \ve^2 s) = O( \ve^2 (1+|\bar y(s)|^2) )
$$
hence
$$
\frac d{ds} \log (1+ |\bar y(s)|^2)  =  O(\ve^2) .
$$
Therefore
$$
\log (1+ |\bar y(s)|^2)  =   \log (1+ |y|^2)  + O(\ve^2(\tau-s) ) \foral s\in (0,\tau),
$$
so that for some positive constants $a,b$ independent of $\tau\in (0, \ve^{-2}T)$
\be\label{bounds1}
a \, (1+ |y|^2)\,  \le 1+ |\bar y(s)|^2 \ \le\    b (1+ |y|^2)\, \foral s\in (0,\tau).
\ee
We recall the representation formula \equ{phi}
$$
\phi(y,\tau) = \int_0^ \tau  E( \bar y (s;\tau,y) ,\ve^2 \tau) \, d\tau  .
$$
Using \equ{bounds1} we readily get
$$
\sup_{t\in [0,T]}
\| (1+|\cdot |)^{-\alpha } \phi (\cdot , t)\|_{L^p(\R^2)} \ \le \
C \ve^{-2} \sup_{t\in [0,T]}
\| (1+|\cdot |)^{-\alpha } E(\cdot , t)\|_{L^p(\R^2)}  .
$$
for any $1\le p\le +\infty$, as desired.

\end{proof}

\medskip
A property that will be useful for the analysis of the outer problem is the fact that if the spacial support of the function $E$ stays
at a uniform large distance of the origin then so does the solution of \equ{pico1-new}.

\begin{lemma}\label{transport2}  There exist numbers $R_0>0$, $\beta>0$  such that for any sufficiently small $\ve$, all $R>R_0$ and any locally bounded function $E$ such that
$$E(y,t) = 0 \foral (y,t) \in B_{R}(0)\times [0,T] $$
we have that the solution of $\equ{pico1-new}$ satisfies
$$\phi(y,t) = 0 \foral (y,t) \in B_{\beta  R}(0)\times [0,T]. $$
\end{lemma}

\begin{proof}
This is an immediate consequence of Estimates \equ{bounds1} for the characteristics and the representation formula \equ{phi}.
\end{proof}

%\medskip
\begin{remark}
	%\label{rem1}
{\em
It is relevant to remark that the results of Lemmas \ref{transport1-new} and \ref{transport2} remain valid if equation \equ{pico1-new} is not defined
in entire $\R^2$ but only
in a domain $\Lambda \subset \R^2\times [0,T]$ that contains a cylinder of the form
$
B_{ m\ve^{-1}}(0) \times [0,T]
$, $m>0$, provided that we only considered points $(y,t)\in B_{ \delta\ve^{-1}}(0) \times [0,T]$ where $\delta>0$ is small and fixed
independently of $\ve$. }
\end{remark}

\subsection{Gradient estimates for the inner transport equations}
\label{secProofTransportInnerGradient}

%It is natural to think that an estimate for the gradient of $E(y,t)$ leads to such an estimate for $\phi(y,t)$.
%This will be true under some additional assumptions on $\RR(y,t)$.
%Let us assume that the function $\RR(y,t)$ satisfies the following bounds
%\begin{align}
%\nonumber
% |\RR(y,t)| +  |\partial_t \RR(y,t)|  \ \le  & \ C \varepsilon^2 (1+|y|^2),
%\\
%%\label{boundDyR}
%|\nn_y \RR(y,t)|+ |\partial_t \nn_y \RR(y,t)|  \ \le  & \ C \varepsilon^2 (1+|y|),
%\nonumber\\
%\label{pico3}
%| D^2_y \RR(y,t)| \ \le  &\  C \varepsilon^2.
%\end{align}
%We will be able to obtain a corresponding estimate for the gradient provided that $|y|\le \delta\ve^{-1}$ where $\delta>0$ is fixed and sufficiently small.

%Next we prove Lemma~\ref{transport3-new}.

%\begin{lemma}\label{transport3}
%Let us assume that $\RR$ satisfies $\equ{pico3}$.
%Then there exist numbers $C,\delta>0$ such that for all sufficiently small $\ve$  and any function $E(y,t)$ that satisfies for some $A,\alpha\in \R$
%\be\label{cotin}
%(1+|y|) |\nn_y E (y,t)| + | E (y,t)| +|E_t(y,t)| \ \le \ A\, (1+ |y|)^\alpha ,
%\ee
%the solution of $\equ{pico1}$ satisfies
%\be\label{cotin1}
%(1+|y|)|\nn_y \phi(y, t)    |  + |\phi(y, t)| \  \le\  C \ve^{-2} A (1+|y|)^\alpha \foral y\in \R^2, \ |y|< \delta\ve^{-1}
%\ee
%\end{lemma}

\begin{proof}[Proof of Lemma~\ref{transport3-new}]
Formally differentiating Formula \equ{phi} with respect to $y_i$, $i=1,2$ we obtain
\begin{align}
\label{phi1}
\pp_{y_i}\phi(y,\tau) =
\int_0^\tau \nn _{\bar y} E(\bar y(s;\tau,y) , \varepsilon^2 s) \cdot \bar y_{y_i}(s;\tau,y) \,ds .
\end{align}
Hence we need suitable estimates for $\bar y_{y_i}(s)$ where $\bar y(s) = \bar y(s;\tau,y)$.
Below we derive various estimates that we need to establish \equ{cotin1}.
Let us set
\[
H(y,s) = \Gamma_0(y) + \RR(y,\varepsilon^2 s).
\]
From Equation \equ{characteristic1-new} we see that
\[
\frac{d}{ds} H(\bar y(s),s) = \varepsilon^2 \frac{\partial \RR}{\partial t}(\bar y(s), \varepsilon^2 s).
\]
Integrating this relation we obtain
\begin{align*}
\Gamma_0( \bar y(s) ) + \RR (\bar y(s), \varepsilon^2 s)
= -
\varepsilon^2
\int_s^\tau
\frac{\partial \RR}{\partial t}(\bar y(\zeta), \varepsilon^2 \zeta)\,d\zeta
+ \Gamma_0(y)  + \RR(y,\ve^2\tau)
\end{align*}
that is
\begin{align}
\nonumber
2 \log( 1+|\bar y(s)|^2)
& =
2 \log( 1+|y|^2)
+ \RR(\bar y(s), \varepsilon^2 s)
-  \RR(y, \varepsilon^2 \tau)
\\
\label{H}
& \quad
- \varepsilon^2 \int_s^\tau \RR_t(\bar y(\zeta) , \varepsilon^2 \zeta)\,d\zeta  .
\end{align}
It is convenient  write $\bar y(s)$ in complex polar form
$$
\bar y(s) = \rho(s) e^{i\theta(s)},
$$
and also denote alternatively
$$
\hat \rho(s) = e^{i\theta(s)}, \quad  \hat \theta(s) = ie^{i\theta(s)}= -\hat \rho(s)^\perp ,
$$
Differentiating \eqref{H} with respect to the initial condition $y_i$  we get
\begin{align*}
&- \frac{4\rho \rho_{y_i}}{1+\rho^2}
+ \nabla \RR (\bar y,\varepsilon^2 s) \cdot ( \hat \rho \rho_{y_i}(s) + \rho \hat \theta \theta_{y_i})(s)
\\ = & -\varepsilon^2
\int_s^\tau \nabla \partial_t \RR(\bar y(\zeta),\ve^2 \zeta) \cdot ( \hat \rho \rho_{y_i}(\zeta) + \rho \hat \theta \theta_{y_i} (\zeta))
\,d\zeta
\\
&
-\frac{4 y_i}{1+|y|^2}
+ \nabla \RR (y,\varepsilon^2 \tau)\cdot e_i.
\end{align*}
Setting \be\label{formula-alpha}\alpha(s) :=  | \rho_{y_i}(s)| + \rho(s)  |\theta_{y_i}(s) |\ee and using  that $\rho(s) \sim |y| $  we then find
\begin{align}
\frac{|\rho_{y_i} | }{\rho^2} (s)
& \leq\ C\,  \big [\, \varepsilon^2 \alpha(s)
+  \varepsilon^4 \int_s^\tau \alpha(\zeta)\, d\zeta\,
+ \frac{1}{|y|^2}\big ]
\label{ineq1}\end{align}
On the other hand, Equation \eqref{characteristic1-new} is  explicitly given by
\begin{align}
\label{eqBarY}
\frac {d\bar y}{ds} (s) =  -4 \frac{\bar y^\perp}{1+|\bar y|^2} + \nabla_y^\perp R(\bar y,  \varepsilon^2 s).
\end{align}
which in terms of $\rho$ and $\theta$ corresponds to the system
\begin{align*}
\dot {\rho} (s) \ = &\    \frac 1\rho \nabla \RR(\bar y ,  \varepsilon^2 s)\cdot \bar y^\perp  \\
\dot \theta(s)\ =&  \    \frac {4} {1+\rho^2}    -   \frac 1{\rho^2}  \nabla \RR(\bar y ,  \varepsilon^2 s)\cdot \bar y \
\end{align*}
%where the initial condition is $\rho(\tau)e^{i\theta(\tau)} = y = \rho_0e^{i\theta_0} $
This immediately yields
\be\label{rels1}
\dot \rho (s) =  O (\ve^2 \rho(s)) , \quad \dot \theta(s) = \frac 4{1+\rho(s)^2} + O(\ve^2).
\ee
Incidentally,  we observe that $$\frac 1{1+\rho(s)^2}\ \le\  \dot \theta(s)\  \le\  \frac 5{1+\rho(s)^2}\quad \hbox{  $|y|< \delta\ve^{-1}$}$$
 for a sufficiently small, fixed $\delta$, thanks to the
established bounds \equ{bounds1}.

\medskip
Differentiating we also see that
\begin{align*}
\ddot \theta(s)\ =&  \   - \frac {8\rho\dot \rho } {(1+\rho^2)^2}  \\ &\ -
    \left (\frac 1{\rho^2} \nn_{\bar y} [\nabla \RR(\bar y ,  \varepsilon^2 s) \cdot \bar y^\perp ]  - \frac {2 \hat\rho }{\rho^3}\nabla \RR(\bar y ,  \varepsilon^2 s) \cdot \bar y \right ) \cdot \dot {\bar y}\\ &\ -   \frac {\ve^2}{\rho^2}  \nabla \RR_t(\bar y ,  \varepsilon^2 s)\cdot \bar y .
\end{align*}
Now, from \equ{eqBarY} we see that $\dot {\bar y} = O(\rho^{-1})$ and from \equ{rels1} we conclude
%\be\label{rels2}
\[
\ddot \theta(s)\ =\  O(\ve^2\rho^{-2}).
%\ee
\]
Next, differentiating with respect to the coordinate $y_i$, we also obtain
\begin{align*}
\dot \rho_{y_i} \ = &\   (\frac 1{\rho} \nn_{\bar y} [\nabla \RR(\bar y ,  \varepsilon^2 s) \cdot \bar y^\perp ]  - \frac { \hat\rho }{\rho^2}\nabla \RR(\bar y ,  \varepsilon^2 s) \cdot \bar y^\perp ) \cdot \bar y_{y_i}    \\
\dot \theta_{y_i}  \ =&  \  -\frac {8 \rho \rho_{y_i} } {(1+\rho^2)^2} -
    (\frac 1{\rho^2} \nn_{\bar y} [\nabla \RR(\bar y ,  \varepsilon^2 s) \cdot \bar y^\perp ]  - \frac {2 \hat\rho }{\rho^3}\nabla \RR(\bar y ,  \varepsilon^2 s) \cdot \bar y ) \cdot \bar y_{y_i}  .    \
\end{align*}
We recall that $
\bar y_{y_i} =  \rho_{y_i}\hat \rho  +  \rho \theta_{y_i}\hat \theta $ hence
$|\bar y_{y_i}(s)| \le \alpha(s)$ with $\alpha$ defined in \equ{formula-alpha}.
Thus
\be\label{relations}
\begin{aligned}
\dot \rho_{y_i} \ = &\   O(\ve^2 \alpha)    \\
\rho \dot \theta_{y_i}  \ =&  \  -\frac {8 \rho^2 \rho_{y_i} } {(1+\rho^2)^2} +
O(\ve^2\alpha )   ,   \\
\dot\rho \theta_{y_i}\ =&\  O( \ve^2 \alpha   )
\end{aligned}
\ee
Using \equ{ineq1} we then get
$$
|\dot \alpha|      \leq\ C\,  \big [\, \varepsilon^2 \alpha(s)
+  \varepsilon^4 \int_s^\tau \alpha(\zeta)\, d\zeta\,
+ \frac{1}{|y|^2}\big ]
$$
 Hence integrating we get
$$
\alpha (s)  \le    C\big [1 +  \varepsilon^2 \int_s^\tau \alpha(\zeta )\, d\zeta \, +\,  \frac {(\tau-s)}{|y|^2} \big]
$$
and using Gronwall's inequality we obtain
\be\label{cotaa}
\alpha (s)  \le C[ 1+ \frac {(\tau-s)}{|y|^2} ]
\ee
In particular, it follows from \equ{relations} that
$$
\dot \theta_{y_i}  =  O( \rho^{-3}) .
$$

\medskip
Next we prove the gradient estimate using the estimates collected above.
We can rewrite Formula \equ{phi1} in the form
\begin{align*}
\pp_{y_i}\phi (y,\tau)
& =  \int_0^\tau D E( \rho e^{i\theta } , \ve^2 s ) [ \rho_{y_i} \hat \rho   +    \rho \theta_{y_i} \hat \theta ] ds
\\
& =
\underbrace{ \int_0^\tau \nn E( \rho e^{i\theta } , \ve^2 s )\cdot \hat\rho  \rho_{y_i} dx}_{I(\tau)}    \ +\     \underbrace{   \int_0^\tau    \nn E( \rho e^{i\theta } , \ve^2 s )\cdot  \hat \theta \, \rho \theta_{y_i} ds}_{II(\tau)} .
\end{align*}
The first integral is bounded by
\be\label{a1}
|I(\tau)| \ \le \ C\, \tau\, (1+|y|)^{\alpha-1}.
\ee
We will see that $II(\tau)$ has a similar estimate. Our first consideration is that
$$
\frac d{ds}  E( \rho e^{i\theta } , \ve^2 s ) =   \nn E( \rho e^{i\theta } , \ve^2 s )\cdot \hat \rho\, \dot \rho  +
\nn E( \rho e^{i\theta } , \ve^2 s )\cdot \hat \theta\, \rho\, \dot \theta   +  \ve^2 E_t( \rho e^{i\theta } , \ve^2 s ).
$$
Hence if we let $$\gamma(s) = \frac {\theta_{y_i}(s) }{\dot \theta(s)}$$ we get
\begin{align*}
II(\tau)
& =
\underbrace{ \int_0^\tau   \frac d{ds}  E( \rho e^{i\theta } , \ve^2 s )\, \gamma(s) \, ds}_{III(\tau)}
  -
  \underbrace{ \int_0^\tau  \gamma(s) \, \nn E( \rho e^{i\theta } , \ve^2 s )\cdot \hat \rho\, \dot \rho  \, ds}_{IV(\tau)}
\\
& \quad -  \underbrace{ \int_0^\tau \, \gamma(s) \,  \ve^2 E_t( \rho e^{i\theta } , \ve^2 s )\, ds}_{V(\tau)} .
\end{align*}
From \equ{rels1} and \equ{cotaa} we see that
$$ \gamma(s) = O(\rho^{-1}\ve^{-2})    . $$
Using this and \equ{rels1} and our assumptions in $E$ we obtain that $III(\tau)$ and $IV(\tau)$ have a bound as \equ{a1}.
On the other hand, we have that
\be\label{a3}
III(\tau)=  E( y , \ve^2 \tau )\, \gamma(\tau)  - E(\bar y(0),0) \gamma(0)  - \int_0^\tau    E( \rho e^{i\theta } , \ve^2 s )\, \dot \gamma(s) \, ds.
\ee
The first two terms in the above expansion  again have an estimate like \equ{a1}.
Now,
$$
\dot \gamma (s) =  \frac {\dot \theta_{y_i}(s)}{\dot\theta(s)}  - \frac {\ddot\theta \theta_{y_i}(s)}{\dot\theta(s)^2}
$$
and using the above estimates
$$
\frac {\dot \theta_{y_i}(s)}{\dot\theta(s)} =  O (\rho^{-1}) , \quad   \frac {\ddot\theta \theta_{y_i}(s)}{\dot\theta(s)^2}  = O(\rho^{-1}).
$$
and the desired estimate for the last term in \equ{a3} follows. The proof is concluded.
\end{proof}

\medskip

\begin{remark}\label{nota}{\em
It is worth noticing that differentiability of the functions $\mathcal R$ and $E$ in the space variable $y$, with no assumptions in their $t$ dependence except continuity, yields uniform control on space derivatives of the solution of \equ{pico2}  (with possibly poor dependence on $\ve$).

\medskip
In fact if in addition to assumption \equ{papa-new} we assume
$$
|D^2_y \mathcal R (y,t)|  +   (1+|y|)\,|D^3_y \mathcal R (y,t)|  \le  M \ve^2 $$
for some $M>1$.
then the equation for $p(s) := \bar y_{y_i}(s; \tau, y) $ is
$$
 \dot p(s) =  D_y \nn_y^\perp (\Gamma_0 + \mathcal R)(y, \ve^2 s)[p(s)], \quad  p(\tau) = e_i , s\in [\tau, T\ve^{-2}] .
$$
Since
$$
|D_y \nn_y^\perp (\Gamma_0 + \mathcal R)(y, \ve^2 s)| \le     C M(|y|^{-2}+  \ve^2).
$$
from where, after an application of  Gronwall's inequality we get
$$
|D_y \bar y(s;\tau,y)| \ \le \  \exp \big (\, C T M(\ve^{-2}|y|^{-2} + 1) \,   \big)
$$
Similarly ${ q}(s) =  \bar y_{y_iy_j} (s;\tau, y)$ satisfies
$$
 \dot q(s) =  D_y \nn_y^\perp (\Gamma_0 + \mathcal R)(y, \ve^2 s)[q(s)] + b(s), \quad  q(\tau) = e_i , s\in [\tau, T\ve^{-2}] ,
$$
where
$$
|b(s)|  \le     (|y|^{-3} +  M\ve^2|y|^{-1} )\exp \big (\,  2CTM(\ve^{-2}|y|^{-2} + 1) \,   \big)
$$
and therefore we get
$$
|D_y^2 \bar y(s;\tau,y) | \ \le \    |y|^{-1} \exp \big (\, C TM(\ve^{-2}|y|^{-2} + 1) \,   \big)
$$
From the representation formula \equ{phi} we  get that if
$$
 |\nn_y E(y,t)| +  (1+|y|) |\nn_y E(y,t)| + (1+|y|)^2 |D^2_y E(y,t)|\ \le \ A(1+|y|^\alpha)
$$
then the solution of \equ{pico1-new} satisfies
$$
|D^2_y\phi(y,t)| \ \le\  C\ve^{-2} A (1+|y|)^{\alpha-2}   \exp \big (\,  2CTM(\ve^{-2}|y|^{-2} + 1) \,   \big).
$$
We also remark that a uniform estimate of the same kind on $\pp_\tau \nn_y\phi (y,\tau) $ can also be derived, following a very similar proof.

}\end{remark}

\subsection{The outer transport equation}\
\label{secProofTransportOuter}

\begin{proof}[Proof of Lemma \ref{int1}]
Let us prove that $\bar x(s)=\bar x(s; t, x)\in \Omega$ for all $0\le s\le t$.
In fact if there was a $s_0\in [0,t]$ such that $\bar x(s)\in \Omega$ for all $s\in (s_0,t]$ but along some sequence $s_n\downarrow s_0$ we have
$x(s_n)\to \bar x\in \pp\Omega$ then from the ODE satisfied we will have $x(s)\to  x_0$ as $s\to  x_0$. But since
$\nabla_x^\perp (\Psi_0 + \QQ)( x_0 , s ) = 0 $ for all $s$ thanks to \equ{cero},
uniqueness for ODE would yield $x(s)=x_0$ for all $s$, a contradiction.
Formula \equ{phi3} is then well-defined and we readily get the bound
\[
 |\phi(x,t)| \le     t\|E\|_{L^\infty (\Omega\times [0,t])} .
\]
We also have that for $1\le p <+\infty $,
$$
\int_\Omega | E(\bar x(s;t,x),s) |^p\,dx = \int_\Omega | E( x; t,s) |^p\,dx
$$
since the map $x\mapsto \bar x(s;t,x)$ is area-preserving. From this we readily get
\begin{equation*}
%\label{eee}
\| \phi(\cdot ,t)\|_{L^p(\Omega)} \ \le \ t \sup_{s\in [0,t]} \| E(\cdot ,s)\|_{L^p(\Omega)}
\end{equation*}
The proof is complete.
\end{proof}

\begin{proof}[Proof of Lemma \ref{transport5}]
We observe that the change of variables $y= \frac {x-\xi_j(t)}{\ve}$ leads us to the fact that an equation of the type \equ{pico1-new}
is satisfied where the assumption \equ{assQ} translates precisely into \equ{papa-new} for the coefficient $\RR(y,t)$. The result then follows from
Lemma \ref{transport2}.
\end{proof}

\begin{proof}[Proof of Lemma \ref{transport4}]
Under the assumptions made, $\phi$ is identically zero near the vortices. This implies that $\phi$ satisfies an equation of the form
\begin{align*}
\left\{
\begin{aligned}
\phi_t  + \nn_x^\perp H \cdot \nn \phi  =  &E(x,t)
\quad \text{in } \Omega\times [0,T],
\\
\phi(x,0) & =0 \quad \text{in } \Omega,
\end{aligned}
\right.
\end{align*}
where  $$ H(x,t)= \Big(1- \sum_{j=1}^N\chi_j(x,t) \Big ) \Psi_0(x;\xi(t)) + \QQ(x,t)$$
and
$\chi_j$ is a smooth function with $\chi_j(x,t) =1$ whenever $|x-\xi_j(t)|<\beta \delta$  and $=0$ if  $|x-\xi_j(t)|>\beta \delta$.
Then
%\be \label{phi4}
\[
\phi(x,t)  =  \int_0^t  E( \bar x(s; t, x) ,  s)\, ds
\]
where   $\bar x(s)= \bar x(s; t, x) $ solves
\begin{align*}
%\label{characteristic3}
\left\{
\begin{aligned}
\frac{d \bar x }{d s} (s) & =
\nabla_{\bar x}^\perp H(\bar x(s) , s ), \quad s\in [0,t],
\\
 \bar x(t) &= x  .
\end{aligned}
\right.
\end{align*}
The derivative $\bar X(s;t,x) =  D_x\bar x(s)$ solves
 \begin{align*}
%\label{characteristic4}
\left\{
\begin{aligned}
\frac{d  \bar X }{d s} (s) & =
D_{\bar x}\nabla_{\bar x}^\perp H(\bar x(s) , s )\,  \bar X(s), \quad s\in [0,t],
\\
 \bar X(t) &= I
\end{aligned}
\right.
\end{align*}
where $I$ is the identity matrix. It follows that $|\bar X(s)| \le C$ for some uniform constant $C$.
Hence
%\be \label{phi4}
\[
D_x\phi(x,t)  =  \int_0^t  D_{\bar x} E( \bar x(s; t, x) ,  s) \bar X(s;t,x) \, ds
%\ee
\]
and estimate \equ{poto1} follows.
\end{proof}

\subsection{Uniform continuity}

Next we make some elementary comments on the uniformity of the modulus of continuity of the characteristics and the solutions of the considered transport equations.

Let us consider an equation of the form
\be\label{ff}
\left\{
\begin{aligned}
\phi_t  + \nn_x^\perp H\cdot \nn \phi  =  &E(x,t)
\quad \text{in } \Omega\times [0,T],
\\
\phi(x,0) & =0 , \quad \text{in } \Omega ,
\end{aligned}
\right.
\ee
with $H(\cdot , t) = 0 $ on $\partial \Omega$. We know that the condition $\Delta H \in L^\infty (\Omega \times [0,T])$ suffices for the well-definiteness of the characteristics $\bar x = \bar x (s; t,x)$ for \eqref{ff}. In fact, the  modulus of continuity in their parameters
depends only on $\| \Delta H \|_{L^\infty (\Omega \times [0,T])}$.

\begin{lemma} \label{modc1}

For all $\varrho>0$ there exists a positive number
$$\delta = \delta (\varrho,\|\Delta H\|_{L^\infty(\Omega\times [0,T])}, T, \Omega) $$
such that
for all $(x_1,t_1  ), \, (x_2,t_2  ) \in \bar \Omega \times [0,T]$ we have
$$
|t_1 - t_2 | + |x_1 - x_2 | < \delta \implies  \left| \bar x  (s; t_1 , x_1 ) - \bar x (s; t_2 , x_2) \right| < \varrho.
$$
\end{lemma}

\begin{proof}
By definition
$$
\dot {\bar x} (s; t_i ,x_i ) = \nabla^\perp H\left( \bar x (s;t_i ,x_i ) \right) , \quad \bar x (t_i ; t_i , x_i ) = x_i
$$
for $i=1,2$. Let $h(s) = \bar x (s; t_1 , x_1 ) - \bar x (s; t_2 , x_2 )$. Then
$$
| \dot h (s) | = | \nabla^\perp H \left( \bar x (s;t_1 ,x_1 )\right) - \nabla^\perp H \left( \bar x (s;t_2 ,x_2 )\right) |
\leq C \| \Delta  H \|_\infty \, | h(s) | \, |\log (|h (s) |) |.
$$
Setting $\beta (s) := |h(s) |^2$, we can assume that $0< \beta (s) <1$. Then we get
$$
\left| {d \over ds} \left( \log \left( \log {1\over \beta (s) }\right) \right) \right| \leq C  \| \Delta  H \|_\infty .
$$
Integrating, we  obtain
%$$
%-C \| \Delta  H \|_\infty  \, T  \, + \log \left( \log {1\over \beta (t_1) }\right)
%\leq \log \left( \log {1\over \beta (s) }\right) \leq C \| \Delta H \|_\infty \, T \, + \log \left( \log {1\over \beta (t_1) }\right),
%$$
%from which we can conclude
\be \label{cott}
e^{-C \| \Delta  H \|_\infty  \, T} \, \log {1 \over \beta (t_1 ) } \leq \log {1\over \beta (s) } \leq  e^{ C \| \Delta H \|_\infty \, T} \log {1\over  \beta (t_1 ) }.
\ee
Observe now that
\be \label{cott1}
|h (t_1)|  = |x_1 - x_2 | + |x_2 - \bar x (t_1 ; t_2 , x_2 ) | \leq C \, \| \Delta H \|_\infty \left( |x_1 - x_2 | + |t_1 - t_2| \right) .
\ee
Combining \eqref{cott} and \eqref{cott1}, we obtain the result.
\end{proof}
Let  $E(x,t)$ be a bounded function that  satisfies
\be\label{mumu}
 \sup_{t\in [0,T], \ |x_1-x_2|< \mu} |E(x_1,t)-E(x_2,t)|\le \Theta (\mu)
\ee
for a certain function $\Theta$ with  $\Theta(\mu)\to 0$ as $\mu\to 0$.

\begin{corollary}

Let $E$ satisfy $\equ{mumu}$. Then
for all $\varrho>0$ there exists a positive number
$$\delta = \delta (\varrho,\|\Delta H\|_{L^\infty(\Omega\times [0,T])},\|E\|_\infty, \Theta, T, \Omega) $$ such
that the solution $\phi(x,t)$ of  \equ{ff} satisfies that
for all $(x_1,t_1  ), \, (x_2,t_2  ) \in \bar \Omega \times [0,T]$ we have
$$
|t_1 - t_2 | + |x_1 - x_2 | < \delta \implies  \left| \phi( x_1,t_1 ) - \phi (x_2,t_2) \right| < \varrho.
$$

\end{corollary}

\begin{proof}
This is a direct consequence of Lemma \ref{modc1} and the representation formula
$$
\phi(x,t)  = \int_0^t E(\bar x(s; t,x) , s) \, ds\
$$
for the solution of \equ{ff}.
\end{proof}

\medskip
Let us now consider
\be\label{ff1}
\left\{
\begin{aligned}
 \phi_\tau  +  \nn_x^\perp  \left( \Gamma_0 (y) + b(y,t) \right) \cdot \nn \phi  =  &E(y,t)
\quad \text{in } \R^2 \times [0, \ve^{-2} T],
\\
\phi(y,0) & =0 , \quad \text{in } \R^2 ,
\end{aligned}
\right.
\ee
with $b(y,t) = 0= E(y,t)$ for $|y| >R$. As in the previous problem, the modulus of continuity for the characteristics
$\bar y = \bar y (s; \tau , y)$ for \eqref{ff1} depends only on $\| \Delta_y  b \|_{L^\infty (\Omega \times [0,T] ) }$, and that of the solution
only on a uniform bound for $E$ and for its  modulus of continuity

\begin{lemma} \label{modc2}
Assume that
$$
 \sup_{\tau \in [0,\ve^{-2} T], \ |y_1-y_2|< \mu} |E(y_1,\tau )-E(y_2,\tau )|\le \Theta (\mu)
$$
for a certain function $\Theta$ with  $\Theta(\mu)\to 0$ as $\mu\to 0$. Then
for each $\varrho>0$ there exists a positive number
$$\delta = \delta (\varrho,\|\Delta b\|_{L^\infty(\R^2 \times [0,\ve^{-2} T])},\|E\|_\infty, \Theta, T, \ve) $$ such
that the solution $\phi(y, \tau )$ of  \equ{ff1} satisfies that
for all $(y_1,\tau_1  ), \, (y_2,\tau_2  ) \in  \R^2 \times [0,\ve^{-2} T]$ we have
$$
|\tau_1 - \tau_2 | + |y_1 - y_2 | < \delta \implies  \left| \phi( y_1,\tau_1 ) - \phi (y_2,\tau_2) \right| < \varrho.
$$
\end{lemma}
\begin{proof}
 We can estimate the modulus of continuity of the characteristics $\bar y (s;\tau,y)$ in the same way as in the proof of Lemma \ref{modc1}: for each $\varrho >0$ there exists $\delta = \delta (\varrho , \| \Delta b \|_{L^\infty (\R^2 \times [0,\ve^{-2} T])} , \ve, T ) $ such that
$$
\foral (y_i,\tau_i  ) \in \bar \R^2 \times [0,\ve^{-2} T], \quad i=1,2 \, : \, |\tau_1 - \tau_2 | + |y_1 - y_2 | < \delta
$$
then
$
\quad \left| \bar y  (s; \tau_1 , y_1 ) - \bar y (s; \tau_2 , y_2) \right| < \varrho.
$
The only difference with the proof of Lemma \ref{modc1} is that inequality \eqref{cott} is replaced by
$$
e^{-C \| \Delta  b \|_\infty  \, \ve^{-2} T} \, \log {1 \over \beta (\tau_1 ) } \leq \log {1\over \beta (s) } \leq  e^{ C \| \Delta b \|_\infty \, \ve^{-2} T} \log {1\over  \beta (\tau_1 ) }.
$$
The conclusion of Lemma \ref{modc2} then follows from the representation formula for the solution of \eqref{ff1}.
\end{proof}

\section{Setting up the full problem}
Let us consider the approximate solution %$(\phi^*,\psi^*,\phi^{out*},\psi^{out*})$ to \equ{g1}-\equ{g4} built in the previous section and the associate approximate solution
$(\omega_*,\Psi_*)$ built in \S \ref{aproxsol}.
 We look for a solution $(\omega,\Psi)$ of  the Euler equation \equ{euler} as
$$
\begin{aligned}
\omega(x,t)  = & \omega_* (x,t;  \xi) +  \vp(x,t)\\
\Psi(x,t)  = & \Psi_{*} (x,t; \xi) +  \psi(x,t)
\end{aligned}
$$
with corrections $(\vp,\psi)$ of the form \equ{f11}-\equ{f22} for functions
\begin{align*}
\phi^{in}(y,t) &= (\phi_1(y,t),\ldots,\phi_N(y,t)) , \quad \phi^{out} (x,t),\\
  \psi^{in}(y,t) &= (\psi_1(y,t),\ldots,\psi_N(y,t)),  \quad \psi^{out}(x,t),\\
\xi(t)  & =\xi^{0}(t) + \xi^1(t) + \ttt \xi(t), \quad y\in \R^2, \ x\in \Omega, \ t\in[0,T],
\end{align*} where $\ttt \xi(t)$ satisfies  \equ{cotasxi2}.
At this point we make the following choice of the large number $R$ in System \equ{inn1}-\equ{out2}.
$$
R =   \frac 1{\ve|\log\ve|} .
$$
We shall make a further decomposition of the functions $\phi_j$, $\psi_j$, as follows. We introduce the functions ${\bf Z}_\ell$,  $\ell=0,1,2,3,$ given by
\be\label{bernieb}
\begin{aligned}
 {\bf Z}_0(y) &= 1, \quad  {\bf Z}_1(y) = y_1\chi_{B_{5R}}(y)  , \\
  {\bf Z}_2(y) &= y_2\chi_{B_{5R}}(y) , \quad   {\bf Z}_3(y,t) = \frac {1- |y|^2}{1+|y|^2} + b_3(y,t) ,
\end{aligned}
\ee
where $b_3(y,t) = O(\ve^2)$ is an explicit function that we will later specify.
%\be\label{Zl}
% Z_0(y) = 1, \quad Z_1(y) = y_1, \quad Z_2(y) = y_2, \quad  Z_3(y) = \frac {1- |y|^2}{1+|y|^2} ,
%\ee
We assume that $\phi_j(y,t)$ has sufficient decay in the $y$-variable and write it in the form
\be\label{desco1}
\phi_j(y,t) =  \bphi_j(y,t) + \sum_{l=0,3}  \alpha_{jl}(t)  \mathcal Z_{1l}(y)
\ee
where
\begin{align}
{\mathcal Z}_{10}(y) = U_0(y),\quad
%& \hskip-1cm
%{\mathcal Z}_{11}(y) & = \pp_{y_1} U_0(y),
%{\mathcal Z}_{12}(y) &= \pp_{y_1} U_0(y),
%& \hskip-1cm
 {\mathcal Z}_{13}(y)  = 2U_0(y) + \nn_y U_0(y)\cdot y ,
\label{Z1l}\end{align}
and we impose the orthogonality conditions
\be
\int_{\R^2} \bphi_j {\bf Z} _\ell \, dy \, = \, 0, \quad \ell=0,1,2,3 \foral t\in [0,T] ,
\label{orto1}\ee
Accordingly, we let
\be\label{desco2}
 \psi_j(y,t) =  \bpsi_j(y,t) + \sum_{l=0,3}  \alpha_{jl}(t) \mathcal Z_{2l}(y)
\ee
where $-\Delta \bpsi_j = \bphi_j$. Precisely, we take
\be \label{bpsidef}
\bpsi_j(y,t) =  \mathcal N ( \bphi_j ) :=  \frac 1{2\pi} \int_{\R^2}\log \frac 1{|z-y|} \,  \bphi_j(z,t)\, dz.
\ee
and
\begin{align*}
\mathcal Z_{20}(y) = \Gamma_0(y),\quad
%&   \hskip-1cm
%\mathcal Z_{21}(y) & = \pp_{y_1} \Gamma_0(y),
%\mathcal Z_{22}(y) & = \pp_{y_2} \Gamma_0(y),
%&   \hskip-1cm
\mathcal Z_{23}(y)  = 2\frac {1-|y|^2}{1+|y|^2} ,
\end{align*}
so that indeed  we have
$
\psi_j= \mathcal N(\phi_j)$
and hence
$$
-\Delta  \psi_j =  \phi_j \inn \R^2\times [0,T].
$$
The terms $\alpha_l \mathcal Z_{1l}(y)$ correspond at main order to infinitesimal perturbations of the parameters involved in the main part of the ansatz
$$
\omega (x; \kappa, \ve, \xi) :=  \frac{\kappa} {\ve^{2}}U_0\left( \frac{x-\xi}{\ve}\right)
$$
Indeed, setting $y=\frac{x-\xi}{\ve}$ we see that
$$ \begin{aligned}
\ve^2 \pp_\kappa \omega (x; \kappa, \ve, \xi) = &  \mathcal Z_{10} (y) , \\
%\ve^3 \pp_{\xi_l} \omega (x; \kappa, \ve, \xi)=  & \mathcal Z_l (y), \quad l=1,2
%\\
\ve^3 \pp_{\ve} \omega (x; \kappa, \ve, \xi)=  & \mathcal Z_{13} (y).
\end{aligned}$$ For later reference we also define
\be \label{Z1l2}
\mathcal Z_{1l}(y) \, = \,  \pp_{y_l} U_0(y), \quad l=1,2
\ee
so that
$$
\ve^3 \pp_{\xi_l} \omega (x; \kappa, \ve, \xi) =   \mathcal Z_{1l} (y) , \quad l=1,2.
%\ve^3 \pp_{\xi_l} \omega (x; \kappa, \ve, \xi)=  & \mathcal Z_l (y), \quad l=1,2
%\\
$$

Conditions \equ{orto1} for $\ell =0,1,2$  correspond to fixing respectively center of mass in $B_{5R}$  and total mass of $\bphi_j(\cdot,t)$ to be zero for they amount to
$$
\int_{B_{5R}} y\,\bphi_j(y,t)\, dy \, = \, 0, \quad \int_{\R^2} \,\bphi_j(y,t)\, dy \, = \, 0.
$$
The reason why we fix the center of mass condition in $B_{5R}$ rather than in entire $\R^2$ is that in the natural space considered for $\bphi_j$,  the function $y\,\bphi_j(y,t)$ will be barely non-integrable in entire space.
We introduce the following operators, depending on a homotopy parameter $\la\in [0,1]$
\begin{align}
{\EE }_{j,\la} & (  \bphi_j ,  \alpha_j, \psi^{out};\ttt \xi) \ :=\ \label{Ejla} \\
&
 \ve^2 \pp_t \bphi_j    +  \kappa_j \nn_y^\perp \big [ \Gamma_0    +   \la {a}_{j}( \bpsi_j , \alpha_j, \psi^{out},\ttt\xi )  \big ]   \cdot
 \nn_y \bphi_j \, +\, \kappa_j  \nabla^\perp_y \bpsi_j   \cdot \nabla_y (U_0 +\la \eta_4\phi_j^*) \nonumber
 \\
 & \quad +\,  \la\,\ttt \EE_{j}( \bpsi_j , \alpha_j,  \psi^{out},\ttt \xi ) \nonumber \\
&\quad  +\, \ve  [ -\dot {\ttt \xi}_j +  (D_{\xi_j}\nn^\perp_{\xi} K)(\xi_0+ \xi_1) [\ttt\xi] + \la\,\eta_4\, \nabla^\perp_y \psi^{out} ] \cdot \nn U_0  +  \ve^2\sum_{l=0,3}  \dot \alpha_{lj}  \mathcal Z_{1l}  \inn \R^2\times [0,T].\nonumber
\end{align}
where $\phi_j^*$ is the function built in Proposition \ref{aproxi},
$\bpsi_j = \mathcal N(\bphi_j)$ as in \equ{bpsidef}
and we have defined

\be
\begin{aligned}
&\alpha_j\ :=  \ (\alpha_{0j}, \alpha_{3j}),\\
&\kappa_j{a}_{j}( \bpsi_j , \alpha_j,  \psi^{out},\ttt \xi )\ := \  \eta_{4} \big[ \, \RR_j^*(\cdot ;\ttt \xi) +  \kappa_j  \bpsi_j + \kappa_j\sum_{l=0,3} \alpha_{lj}\mathcal Z_{2l} +  \psi^{out} \big ] \\
&\ttt \EE_{j}( \bpsi_j , \alpha_j, \psi^{out},\ttt \xi )\ :=
 \sum_{l=0,3} \alpha_{lj} \kappa_j \nn_y^\perp  [{a}_{j}( \bpsi_j , \alpha_j, \psi^{out},\ttt \xi )  ] \cdot
 \nn_y  \mathcal Z_{1l}
 \\
   &+ \ve \big [\nn_{\xi} ^\perp K(\xi^0+\xi^1+\ttt \xi) -   \nn_{\xi} ^\perp K(\xi^0+\xi^1)- (D_{\xi_j}\nn^\perp_{\xi} K)(\xi_0+ \xi_1) [\ttt\xi]\big ] \cdot \nn U_0\\
&+  \,\nabla^\perp_y ( \eta_4\psi^{out} ) \cdot \nabla_y \phi_j^*  +  \sum_{l=0,3} \alpha_{lj}  \nabla^\perp_y   \mathcal Z_{2l}   \cdot \nabla_y (\eta_4\phi_j^*) + \,\eta_4E^*_j (\ttt \xi ).
\end{aligned}\label{defaj}\ee
Here $E^*_j (\ttt \xi )$ is the remainder in Proposition \ref{aproxi}, $K$ is the function defined in \equ{K}  and
\begin{align*}
\RR_j^*(y ,t;\ttt \xi) =  & \kappa_j \RR_j(y,t ;\xi ) + \kappa_j \psi_j^*(y,t;  \xi) + \psi^{out*}( \xi_j(t)+ \ve y,t; \xi) ,\\
%\xi(t)= & \xi^0(t)+ \xi^1(t)+  \ttt \xi(t),\\
\eta_4(y) =& \eta \left ( \frac{|y|}{4R} \right ),
\end{align*}
where $\RR_j$ was defined in \equ{defRa}. For  $\phi_j$ as in \equ{desco1} and $\psi_j =\mathcal N(\phi_j)$ so that \equ{desco2} holds,
and letting
$ \bphi =(\bphi_1,\ldots,\bphi_N), %\quad \bpsi =(\bpsi_1,\ldots,\bpsi_N),
\quad \alpha =(\alpha_1,\ldots,\alpha_N), $
we define
\begin{align*}
\EE_{1,\la}^{out}( \ttt \phi ,  \alpha, \phi^{out},  \psi^{out},\ttt \xi)
& \, = \,   \pp_t \phi^{out}   \, + \, \nn^\perp _x\big (\Psi_* (\cdot;  \xi)
+ \la\sum_{j=1}^N \eta_{2j}  \psi_j + \la  \psi^{out} \big )  \cdot \nn\phi^{out}  \\ &\
+ \la \ttt \EE_{1}^{out}( \ttt \phi ,\alpha,     \psi^{out},\ttt \xi)
\inn \Omega \times [0,T].
\end{align*}
where
\be\label{EE1}
\begin{aligned}
& \ttt \EE_{1}^{out}( \ttt \phi , \alpha,  \psi^{out},\ttt \xi)\, :=\, \\
&  \ve^{-2}  \sum_{j=1}^N  \kappa_j \phi_j \big[ \pp_t \eta_{1R,j} + \nn_x^\perp (\Psi_* + \sum_{j=1}^N \kappa_j \eta_{2R,j} \psi_j +\psi^{out})) \cdot
 \nn_x \eta_{1R,j} \big ]
\\ &+  (1-  \sum_{j=1}^N\eta_{1R,j}) \nn_x ( \,  \sum_{l=1}^N \kappa_l \eta_{2R,l}\,\psi_l +\psi^{out}) \cdot \nn_x \omega_*
 \\ & +\,   (1- \sum_{j=1}^N \eta_{1R,j} ) E_*(\ttt\xi),
 \end{aligned}\ee
We also define
\begin{align*}
 &  \EE_{2,\la}^{out}( \ttt\psi, \alpha , \phi^{out}, \psi^{out},\ttt \xi )\ =\\ \
  &   \Delta_x  \psi^{out}    +
     \la\,  \phi^{out}
      +  \la\, \sum_{j=1}^N \kappa_j ( \psi_j \Delta_x\eta_{2j}  + 2 \nn_x \eta_{2j} \cdot \nn_x   \psi_j) + \la\,E_{2*}^{out}(\ttt\xi)
\end{align*}
$$
\psi^{out} = 0 \onn \pp\Omega\times [0,T].
$$
The functions $E^{out}_* (\ttt \xi )$ and $E_{2*}^{out}(\ttt\xi)$ are defined in \eqref{roger}.
The key observation is that for $\la=1$ these operators recover $E_j$, $E_1^{out}$ and $E_2^{out}$ given by  \equ{Ej}, \equ{E1} and \equ{E2}, whose annihilation corresponds to System
\equ{inn1}-\equ{out1}-\equ{out2}. Indeed, for
$\phi_j$, $\psi_j$ given by \equ{desco1}-\equ{desco2} we have the identities

\begin{align*}
{\EE }_{j,1}  (  \bphi_j , \alpha_j, \psi^{out};\ttt \xi)  \, = &\, E_j( \phi_j ,  \psi_j  , \psi^{out} ;\xi)\inn B_R\times [0,T] , \\
\EE_{1,1}^{out}( \ttt \phi ,\alpha,   \phi^{out},  \psi^{out};\ttt \xi) \, = &\, E_1^{out}( \phi^{in} , \psi^{in}  , \phi^{out}, \psi^{out};\xi),
\\
 \EE_{2,1}^{out}( \ttt\phi ,\alpha, \phi^{out},  \psi^{out};\ttt \xi )\, = &\, E_2^{out}( \psi^{in}  ,  \phi^{out},  \psi^{out} ;\xi ).
\end{align*}
Problem \equ{inn1}-\equ{out2} amounts to finding $( \ttt \phi , \alpha, \phi^{out}, \psi^{out},\ttt \xi)$ that make the three quantities above equal to zero.  We will do this by a continuation argument that involves finding uniform a priori estimates for the corresponding equations along the deformation parameter $\la$  imposing in addition initial condition $0$ for all the parameter functions.

\medskip
We consider the functions $\alpha_j, \psi^{out},\xi$ as given
and require that $\bphi_j$ satisfies an initial value problem of the form

\be\left\{ \begin{aligned}
%\ve^2 \pp_t \bphi    +  \kappa_j \nn_y^\perp \big [ \Gamma_0    + &  \la {a}_{j}( \bpsi , \alpha_j, \ttt \psi^{out},\ttt\xi )  \big ]   \cdot
% \nn_y \bphi_j \, +\, \kappa_j  \nabla^\perp_y \bpsi_j   \cdot \nabla_y (U_0 +\la \eta_R\phi_j^*)\\
{\EE }_{j,\la}  (  \bphi_j ,  \alpha_j, \ttt \psi^{out};\ttt \xi)   =&   \sum_{l=0}^3 c_{lj}(t)\mathcal Z_{1l}(y)
  \inn \R^2\times [0,T] , \\
\bphi_j(y,0 )  &= 0\inn \R^2\\
%\int_{B_{5R}} \bphi_j(y,t)Z_l(y)\, dy \, &= \, 0, \quad l=1,2 \foral t\in [0,T] ,\\
%\int_{\R^2} \bphi_j (y,t)Z_l(y)\, dy \, &= \, 0, \quad l=0,3 \foral t\in [0,T].
\end{aligned}\right. \label{equinner} \ee
where $\bpsi_j$ is given in terms of $\bphi_j$ by \equ{bpsidef}.
 and
the four orthogonality conditions \equ{orto1} are imposed on $\bphi_j$.  The functions  $c_{lj}(t)$ do take
explicit values which are linearly dependent on  $\bphi_j$. They are computed after integrating the equation against ${\bf Z}_\ell$ in space variable

\medskip
We will later obtain through integrations by parts a ``workable'' expression for the functionals $c_{\ell j}$
 which in particular will not depend on any
differentiability of  $\bphi_j$. To annihilate $\EE_{\la,j}$
we impose the initial value problems
\be\label{equc} \begin{cases} c_{lj}[  \bphi_j ,  \alpha_j,  \psi^{out},\ttt \xi,\la](t) =0 \foral t\in [0,T],\\
\ttt \xi_j(0)=  \alpha_j(0) =0  \end{cases}\ee
for all $l$ and $j$.
We require \be\label{equout1} \left\{\begin{aligned} \EE_{1,\la}^{out}( \ttt \phi ,  \alpha ,  \phi^{out},  \psi^{out};\ttt \xi) &=0 \inn  \Omega\times [0,T],\\
 \phi^{out}(\cdot ,0) &= 0 \inn \Omega.
 \end{aligned} \right.\ee

\be\label{equout2} \left\{\begin{aligned} \EE_{2,\la}^{out}( \ttt\phi,\alpha , \phi^{out},  \psi^{out};\ttt \xi ) &=0 \inn \ \ \Omega\times [0,T],\\
 \psi^{out} &= 0 \onn \pp\Omega\times [0,T].
\end{aligned} \right.\ee
and we recall that, as in \equ{desco1}-\equ{desco2}
\begin{align*}
\phi_j(y,t) \, = & \, \bphi_j(y,t) + \sum_{l=0,3}  \alpha_{jl}(t)  Z_{1l}(y),\\
\psi_j(y,t)\, = & \, \bpsi_j(y,t) + \sum_{l=0,3}  \alpha_{jl}(t) \mathcal Z_{2l}(y).
\end{align*}

Let us explain the strategy of the rest of the proof.
We consider the vector of parameter functions
$\vec { p} = (\bphi, \alpha,  \phi^{out},\psi^{out}, \ttt\xi) $. We shall define a Banach space $(X, \|\cdot \|_X)$ where
these functions belong
and set the system of equations  \equ{bpsidef}, \equ{equinner}, \equ{equc}, \equ{equout1}, \equ{equout2},
in the fixed point form
\be\label{equp}
\vec { p}  = {\mathcal F} ( \vec { p} , \la ), \quad \vec { p} \in \OO\, .   %\|\vec { p}\|_X\le r,
\ee
Here $\OO$ designates a bounded open set in $X$ with $\vec p= \vec 0\in \OO$ and
 ${\mathcal F}( \cdot  , \la )$  is a homotopy of nonlinear compact operators on $\bar \OO$
with ${\mathcal F} ( \cdot  , 0)$ linear.

\medskip
We shall prove that a suitable choice of a small $\OO$ yields that for all $\la\in [0,1]$
no solution of \equ{equp} with  $\vec p\in \pp \OO $   exists.
Existence  of a solution of \equ{equp} for $\la=1$ thus follows from standard degree theory. But this precisely corresponds to a solution of the original problem. The definition of the norm and the set $\OO$ will yield the desired properties of the solution of Euler equation thus obtained.

%where $\vec p = (\bphi^{in} ,\bpsi^{in} ,  \alpha, \ttt \phi^{out} , \ttt \psi^{out},\ttt\xi )\,$

\medskip
In order to find the desired a priori estimates we need several preliminary considerations which we make in the next section.

\section{Preliminaries for a priori bounds}

\subsection{The Poisson equation}

Let us consider for the solution of the Poisson equation
\be\label{poisson}
-\Delta_y \psi(y) = \phi(y)  \inn \R^2
\ee
given by the Newtonian potential
\be\label{newtonian}
\psi(y) =  \frac 1{2\pi} \int_{\R^2}  \log \frac 1{|y-z|} \phi(z)\, dz.
\ee
Our basic assumptions are
$$
\int_{\R^2} \phi^2(y) U_0(y)^{-1}dy \, <\,+\infty , \quad \int_{\R^2} \phi(y)\, dy\, =\, 0  .
$$
Equation \equ{poisson} can be pulled back into the sphere $S^2$ by means of  the stereographic
\be \label{stereo}
\Pi(z) = (    \frac{z_1}{1-z_3},\frac{z_2}{1-z_3}), \quad z \in S^2 -\{(0,0,1)\},
\ee
whose inverse is given by
$$
\Pi^{-1} (y) =   (\frac{2y_1}{1+ |y|^2}, \frac{2y_2}{1+ |y|^2},\frac{|y|^2-1 }{1+|y|^2}), \quad y\in \R^2.
$$
For a function  $h(y):\R^2\to \R$ we denote by $h(z)$ the function $h(\Pi(z))$ defined on $S^2$. Then we get
\begin{align*}
%\label{eq-}
\Delta_{\R^2} h(y) = U_0  (\Delta_{S^2}  h ) (z) ,  \quad  y = \Pi(z).
\end{align*}
and
$$
\int_{\R^2}h(y)\, dy \  =\ \int_{S^2}  h(z)\,  U_0^{-1}(z) \, d\sigma(z)  .
$$
Let us denote
$$
\ttt \phi(y) = \frac {\phi(y)}{U_0(y)}
$$
Then we have
$$
\int_{S^2} \ttt \phi^2 (z) \, d\sigma(z) \ =\ \int_{\R^2}  \phi^2(y) U_0(y)^{-1}\, dy , \quad \int_{S^2} \ttt \phi (z) \, d\sigma(z) \ =\ 0,
$$
and Equation \equ{poisson}  gets transformed into
\be\label{poisson1}
-\Delta_{S^2} \ttt \psi  = \ttt \phi \inn S^2  .
\ee
The mean value zero condition implies the existence of a unique solution of \equ{poisson1} with mean value zero, which we denote in what follows as $  (-\Delta_{S^2})^{-1} \ttt \phi $. This solution is in $H^2(S^2)$, hence it is H\"older continuous of any order. Setting $P=(0,0,1)$ then
the function
$$ \ttt \psi(z) = (-\Delta_{S^2})^{-1} \ttt \phi (z)- (-\Delta_{S^2})^{-1} \ttt \phi (P)$$
is the only one that vanishes at $P$. Pulling back this function to a $\psi(y)$ defined in $\R^2$ we see that
it satisfies equation \equ{poisson} and it is the only solution that vanishes as $|y|\to \infty$. In fact that condition is precisely satisfied by
\equ{newtonian}.  This and the H\"older condition yields that $\psi$ given by \equ{newtonian} satisfies
\be\label{cota1psi}
|\psi(y)| \ \le \  \frac {C}{1+ |y|^{1-\sigma}}\,  \| \phi\, U_0^{-\frac 12} \|_{L^2(\R^2)}
\ee
for an arbitrarily small $\sigma>0$. Moreover, for all $p>2$ we have an $L^p(S^2)$-gradient estimate for $\ttt \psi(z)$ of the form
$$
\|\nn_{S^2} \ttt\psi \|_{L^p(S^2)} \ \le \ C  \| \ttt \phi \|_{L^2(S^2)}
$$
which yields for $\psi$ in \equ{newtonian}
\be\label{gradp}
\| U_0^{-\frac 12 + \frac 1{2p}} \nn \psi \|_{L^p(\R^2)} \ \le \ C  \| \phi\, U_0^{-\frac 12} \|_{L^2(\R^2)}.
\ee
If we in addition have that $\ttt\phi\in L^q(S^2)$ for some $q>2$, then a solution of
\equ{poisson1} satisfies
$$
\|\nn_{S^2} \ttt\psi \|_{C^{0,\alpha} (S^2)} \ \le \ C  \| \ttt \phi \|_{L^q(S^2)}.
$$
for some $0<\alpha<1$.
This estimate translates for $\psi $ in \equ{newtonian} into
\be \label{gradtotal}   |\psi(y)| \, + \,  (1+|y|) \, |\nn \psi(y) | + (1+|y|)^{1+\alpha} [\nn \psi ]_\alpha (y)    \ \le \frac C{1+|y|}  \| U_0^{\frac 1q -1} \phi \|_{L^q(\R^2)}, \ee
where
$$
[\nn \psi ]_\alpha (y) = \sup_{y_1,y_2\in B_1(y)} \frac {|\nn \psi (y_1) -\nn \psi (y_2) |}  {|y_1-y_2|^\alpha} .
$$
A useful corollary \equ{gradtotal} is the following estimate.
For $a>0$ let us denote
$$
\|\phi\|_a = \sup_{y\in\R^2} |(1+|y|)^a\phi(y)|.
$$

\begin{lemma}\label{interpolacion}
Let $0<\beta<1$ be fixed.
Then given $0<\sigma<1$ there exist numbers $C_\sigma >0$ and  $\alpha \in (0,1)$ such that for any function $\phi(y)$ with $\|\phi\|_{3+\beta}<+\infty$ and
$\int_{\R^2}\phi =0 $ we have
$$
\begin{aligned}
& |\psi(y)| \, +\, (1+|y|) |\nn_y \psi (y)|  + (1+|y|)^{1+\alpha}  [\nn \psi ]_\alpha (y) \\ \, & \le \, \frac {C_\sigma }{1+|y|}   \, \|\phi\|_{3+\beta}^\sigma\|U_0^{-\frac 12} \phi\|_{L^2(\R^2)}^{1-\sigma}  \foral y\in \R^2.
\end{aligned}
$$

\end{lemma}

\begin{proof}

Let us fix a number $p$ with $2<p< \frac 2{1-\beta}$. Then we have that
$$
  \| U_0^{\frac 1p -1} \phi \|_{L^p(\R^2)} \ \le \ C \,\|\phi\|_{3+\beta}.
$$
We write a number $q \in (2,p)$, which we are interested in taking it arbitrarily close to 2,
in the form   $$q= 2 (1-\la ) + \la p , \quad \la \in (0,1).$$
Using H\"older's inequality we check that
\begin{equation*}\begin{split}
\| U_0^{{1\over q}-1} \phi \|_{L^q (\R^2)}^q &= \int_{\R^2} U_0^{1-q} |\phi |^q \, dy \\
%&=\int_{\R^2}  ( U_0^{-(1-\la)} |\phi|^{2(1-\la)} )\, ( U_0^{-(p-1) \la } |\phi|^{p\la} \, dy \\
&\leq \left(\int_{\R^2}  U_0^{-1} |\phi|^2 \, dy \right)^{(1-\la)} \, \left( \int_{\R^2}  U_0^{1-p} |\phi|^p \, dy \right)^\la \\
&= \| U_0^{-{1\over 2}} \phi \|_{L^2 (\R^2) }^{2(1-\la)} \, \| U_0^{{1\over p}-1} \phi \|_{L^p (\R^2)}^{p \la},
\end{split}
\end{equation*}
and then
$$
\| U_0^{{1\over q}-1} \phi \|_{L^q (\R^2)} \leq \| U_0^{-{1\over 2}} \phi \|_{L^2 (\R^2) }^{{2\over q} \, (1-\la)} \, \|  \phi \|_{3+\beta }^{{ p \over q} \,  \la} .
$$
Letting $$\sigma = \frac {p}{q}\la =  \frac   {p\la}{ 2(1-\la) + p\la } $$
the desired result readily follows from \eqref{gradtotal}.
\end{proof}

\subsection{A quadratic form}

Let us consider functions  $\phi$ that satisfy
\be\label{assphi1}
\| \phi\, U_0^{-\frac{1}{2} } \|_{L^2(\R^2)}<\infty
\ee
and the orthogonality conditions
\be
\begin{aligned}
\int_{\R^2} \phi(y){\bf Z}_\ell(y)\, dy \, = \, 0, \quad \ell=0,1,2,3.
\quad
%\int_{\R^2} \phi (y)Z_l(y)\, dy \, = \, 0, \quad l=0,3 .
\end{aligned}
\label{orto40}\ee
where, consistently with \equ{bernieb},  we denote
\be\label{bbernie}
\begin{aligned}
 {\bf Z}_0(y) &= 1, \quad  {\bf Z}_1(y) = y_1\chi_{B_{5R}}(y)  , \\
  {\bf Z}_2(y) &= y_2\chi_{B_{5R}}(y) , \quad   {\bf Z}_3(y) = \frac {1- |y|^2}{1+|y|^2} + b_3(y) ,
\end{aligned}
\ee
where $R$ is a large positive number
and $b_3(y)$ satisfies
\be\label{bernie0} |b_3(y)| \le R^{-\nu} \quad \hbox{for some $\nu>0$. } \ee
We have the validity of the following key estimate for functions with the above properties.

\begin{lemma}\label{passl}
There exists a number $\gamma>0$ such that for any  sufficiently large $R$  and all $\phi$ satisfying conditions $\equ{assphi1}$-$\equ{orto40}$, the following holds:
let $g$ be given by
\[
g =   U_0^{-1} \phi  - \psi, \quad \psi(y) = \frac 1{2\pi}\int_{\R^2} \log \frac 1 {|y-z |} \phi(z)\, dz
\]
%where $\psi$ is the solution of $-\Delta\psi = \phi$  given by  \eqref{newtonian}.
we have
\begin{equation}\label{pass}
 \int_{\R^2} \phi g \geq \frac{\gamma}{|\log R|}
\int_{\R^2} \phi^2 U_0^{-1}.
\end{equation}
\end{lemma}
\begin{proof}
Let us set
$ \ttt \phi = U_0^{-1}\phi $. We recall that, after stereographic projection, we have that
$$
\int_{S^2} \ttt \phi^2  =  \int_{\R^2} \phi^2 U_0^{-1}, \quad
\int_{S^2} \ttt\phi = \int_{\R^2}\phi = 0 .
$$
Besides
$$
\int_{\R^2} \phi g =  \int_{S^2}  \ttt \phi ( \ttt \phi - 2(-\Delta_{S^2})^{-1} \ttt \phi ) .
$$
Expanding $ \ttt \phi$ in the orthonormal basis in $ L^2(S^2)$ of spherical harmonics   we get
$$
\ttt \phi =   \sum_{j=0}^\infty  \ttt\phi_j e_j (z) = \sum_{j=0}^3 \tilde  \phi_j e_j (z)  +\ttt\phi^\perp ,
$$
where $ -\Delta_{S^2} e_j = \la_j e_j$.
Here $ \la_0=0$ and $ e_0$ is constant, while  $ \la_1=\la_2=\la_3 = 2$, with $ e_j(z) =z_j$.
From \eqref{orto40}, we get $ \ttt \phi_0=0$. %and also $ \ttt \phi_3 =0$.
Thus
\be \int_{\R^2} \phi g
 =    \sum_{j=4}^\infty \left ( 1- \frac 2{\la_j} \right )\ttt\phi_j^2  \ \ge \  c_1\|\ttt\phi^\perp \| _{L^2(S^2)}^2
\label{bernie}\ee
for some uniform $c_1>0$. We also have
$$
0 = \int_{B_R} \phi y_j   =  c_2 \ttt\phi_j  +     O( \|\ttt \phi^\perp\|_{L^2(S^2)}) \, |\log R|^{\frac 12 }
$$
for some uniform $c_2>0$.  On the other hand, we have
$$
0 = \int_{\R^2} \phi {\bf Z}_3   =  \ttt\phi_3   +    O(R^{-\nu})  \|\ttt\phi \| _{L^2(S^2)}.
$$
From the above relations we get that for some $c>0$ independent of $R$,
\begin{align*}
\|\ttt \phi^\perp\|_{L^2(S^2)}^2 \, & \ge  \,  \|\ttt \phi\|^2_{L^2(S^2)}  - c \sum_{\ell = 1}^3 |\ttt \phi_\ell|^2 ,  \\
  & \ge\,  (1- O(R^{-2\nu}) ) \,\|\ttt \phi\|_{L^2(S^2)}^2 \, - \, c |\log R|\, \|\ttt \phi^\perp\|_{L^2(S^2)}^2.
\end{align*}
From here and \equ{bernie} it follows that
$$
\int_{\R^2} \phi g  \ge \frac{\gamma}{  |\log R|} \int_{S^2} \ttt \phi^2
$$
for some uniform $\gamma>0$, as desired.\end{proof}

\subsection{An $L^2$-weighted a priori estimate}
We let $f(v)= e^v $
and consider a linear transport equation of the form
\be\label{innerL}\left \{
\begin{aligned}
%\label{innerL}
\varepsilon^2 \phi_t +  \nabla_y^\perp( \Gamma_0 + a_* +a)\cdot  \nabla_y ( &\phi  - f'( \Gamma_0 +a_*) \psi )
%\ = \ \sum_{l=0}^3 c_{l}(t) \mathcal  Z_{1l}(y)  \inn \R^2 \times (0,T)
 \\ &\quad\ + \, E(y,t)\, = \, 0  \inn \R^2 \times (0,T),\\
%\label{innerInitial}
&\qquad\quad\phi(\cdot,0)\,  =\, 0 \inn \R^2
\end{aligned}\right.
\ee
where
\be\label{psi1}
\psi(y,t) =  \frac 1{2\pi} \int_{\R^2}  \log \frac 1{|y-z|} \phi(z,t)\, dz.
\ee
and $\phi(y,t)$ satisfies the orthogonality conditions
\be
\begin{aligned}
\int_{\R^2} \phi (y,t){\bf Z}_\ell(y)\, dy \, &= \, 0, \quad l=0,1,2,3  \end{aligned}
\label{orto4}\qquad\foral t\in [0,T].\ee
where ${\bf Z}_\ell$ is defined in \equ{bbernie} with
\begin{equation}
\label{defR}
R= {1 \over \ve |\log \ve| }
\end{equation}
and $b_3=b_3(y,t) $ satisfying \equ{bernie0} for some $\nu>0$.
On the functions
  $a_*(y,t)$, and $a(y,t)$  we assume
\begin{equation}\label{nu0}
a_*(y,t), \, a (y,t) \, =0 \quad \hbox{for } |y|\ge 4R, \quad
\Delta_y (a + a_*) \in L^\infty (\R^2\times (0,T))
\end{equation}
and
for some numbers $C>0$, $\nu >0$,
\begin{align}
   |\pp_t a_*(y,t)| + | (1+ |y|) \,  \nn_y a_* (y,t)|  \ \le & \  C \ve^2(1+ |y|^2) \nonumber\\
|\nn_y a(y,t)|   \ \le & \   \ve^{2+\nu}.\label{nu}\end{align}

\begin{lemma}\label{lin}
There exists a constant $C>0$ such that for any
$a$, $a_*$ satisfying $\equ{nu}$-$\equ{Enu}$,  $R$ given by \eqref{defR}, all sufficiently small $\ve$
and any solution $\phi$ of $\equ{innerL}$-$\equ{orto4}$ with \be\label{aaa}\sup_{t\in [0,T]} \|U_0^{-\frac 12 } \phi(\cdot , t)\|_{L^2(\R^2)} \ < \ +\infty\ee
we have
\be \label{pass1}
\sup_{t\in [0,T]} \|U_0^{-\frac 12 } \phi(\cdot , t)\|_{L^2(\R^2)} \ \le \ C\,\ve^{-2} \, |\log \ve |^{1\over 2} \, \sup_{t \in [0,T]} \, \| E (\cdot,t)\,  U_0^{-1/2} \|_{L^2 (\R^2) }.
\ee

\end{lemma}

\begin{proof} Let us assume that
$$ \sup_{t \in [0,T]} \, \| E (\cdot,t)\,  U_0^{-1/2} \|_{L^2 (\R^2) } \, <\, +\infty $$
and define the functions
\[
U_1  = f'(\Gamma_0 + a_*) , \quad  U_1  g_1 =
\phi  - f'(\Gamma_0 + a_*) \psi  .
\]
We claim that, after multiplying equation \eqref{innerL} against $g_1$ and integrating in $\R^2$, we get
\begin{equation}\begin{split}\label{uu00}
{\ve^2\over 2} {d \over dt} &\left( \int_{\R^2} \phi g_1 \, dy \right) + {\ve^2 \over 2} \int_{\R^2} \phi^2  \,  {(U_1)_t \over U_1^2}  \, dy + \frac{1}{2}
\int_{\R^2}
\frac{\nabla U_1}{U_1} (\nabla_y^\perp a ) \,  U_1 g_1^2 \, dy\\
&
+ \int_{\R^2} E  g_1 \, dy =0.
\end{split}
\end{equation}
To prove \eqref{uu00}, we first observe that
\begin{equation}\begin{split}\label{uu0}
0&=\ve^2 \int_{\R^2} \phi_t  g_1 \, dy  +
\int_{\R^2} U_1^{-1} \nabla_y^\perp ( \Gamma_0 + a_* + a) \nabla(\frac{U_1^2 g_1^2}{2})  \, dy
+ \int_{\R^2} E   g_1 \, dy
\\
& =\ve^2 \int_{\R^2} \phi_t  g_1 \, dy  +
\frac{1}{2}
\int_{\R^2}
\frac{\nabla U_1}{U_1} (\nabla_y^\perp a ) \,  U_1 g_1^2  \, dy\\
&
+ \int_{\R^2} E  g_1 \, dy
\end{split}
\end{equation}
%For any function $\Phi$ write
%\be \label{uu1}
%\int_{\R^2} \phi   g_1 \, dy =
%\int_{\R^2} \phi [ {\phi \over U_0} - \psi ]  \, dy  + \int_{\R^2} \phi^2 \, {U_0 - U_1 \over U_0 \, U_1 } \, dy.
%\ee
As before, we set  $\tilde \phi  = \phi U_0^{-1} $, and we recall that, after the stereographic projection \eqref{stereo}
$$
\int_{\R^2} \phi [ {\phi \over U_0} - \psi ] (\cdot ,t) \, dy  = \int_{S^2} \tilde \phi [ \tilde \phi - 2 (- \Delta_{S^2})^{-1} \tilde \phi ] (\cdot ,t) d\sigma .
$$
Direct computations give
\begin{equation*}\begin{split}
\int_{\R^2} \phi_t \, g_1  \, dy &=
\int_{\R^2} \phi_t  [ {\phi \over U_0} - \psi ] \, dy  +
\int_{\R^2}
\phi\,\phi_t \,
{U_0 - U_1 \over U_0 \, U_1 }  \, dy\\
&= \int_{S^2} \tilde \phi_t
[ \tilde \phi -2 (- \Delta_{S^2})^{-1} \tilde \phi ] \, d\sigma \, +\,  \int_{\R^2}  \phi\,\phi_t
\, {U_0 - U_1 \over U_0 \, U_1 }  \, dy\\
%&= {1\over 2} {d \over dt}
%\left(\int_{S^2} \tilde \phi
% [ \tilde \phi + \Delta_{S^2}^{-1} \tilde \phi ]
%  d\sigma \right) +
%\int_{\R^2} {d \over dt} \left( {\phi^2 \over 2} \right)  \,
%{U_1 - U_0 \over U_0 \, U_1 }\, dy\\
&= {1\over 2} {d \over dt}
 \left(\int_{\R^2}  \phi  [  \phi + \Delta_{\R^2}^{-1} \phi ]  d y \right) +
\int_{\R^2} {d \over dt}
\left( {\phi^2 \over 2} \right)  \,
{U_0 - U_1 \over U_0 \, U_1 }\, dy\\
&= {1\over 2} {d \over dt}
\left( \int_{\R^2} \phi g  \, dy \right) +
\int_{\R^2} {\phi^2 \over 2} \, {d \over dt}
\left({U_1 - U_0 \over U_0 U_1} \right) \, dy\\
&= {1\over 2} {d \over dt} \left( \int_{\R^2} \phi g  \, dy \right) + \int_{\R^2} {\phi^2 \over 2} \,  {(U_1)_t \over U_1^2}  \, dy.
\end{split}
\end{equation*}
Replacing this term in \eqref{uu0} we obtain \eqref{uu00}.

\medskip
Next, we estimate the last three terms in \eqref{uu00}.
 Since $U_1 = (1+ o(1)) \, U_0$, then from Lemma \ref{passl} and \eqref{pass} we  obtain
\begin{equation*}\begin{split}
\left| \int_{\R^2} E g_1 \, dy \right| &\leq \| E U_0^{-{1\over 2}} \|_{L^2(\R^2)} \, \| U_0^{1\over 2} g_1 \|_{L^2(\R^2)} \\&
\leq C \| E U_0^{-{1\over 2}} \|_{L^2(\R^2)} \, \| U_0^{1\over 2} g \|_{L^2(\R^2)} \\
&\leq C |\log R |\,  \| E U_0^{-{1\over 2}} \|_{L^2(\R^2)} \, \left( \int_{\R^2} \phi g_1 \right)^{1\over 2}.
\end{split}
\end{equation*}
On the other hand, using \eqref{nu0} and \eqref{nu}, we get
\begin{equation*}\begin{split}
\left| \int_{\R^2} \phi^2 {(U_1)_t \over U_1^2}  \, dy \right| &\leq C \int_{\R^2} \phi^2 U_0^{-1} |\pp_t a_*| \, dy \\
&\leq C\ve^2 (1+ R^2)\, \| \phi U_0^{-{1\over 2} } \|_{L^2(\R^2)}.
\end{split}
\end{equation*}
Using again Lemma \ref{passl}, \eqref{pass} and assumptions \eqref{nu}, we finally get
\begin{equation*}\begin{split}
\left| \int_{\R^2}
\frac{\nabla U_1}{U_1} (\nabla_y^\perp a ) \,  U_1 g_1^2 \, dy \right|&\leq C \ve^{2+\nu}  |\log R| \int_{\R^2} \phi^2 U_0^{-1} \, dy.
\end{split}
\end{equation*}

\medskip
Then
\begin{equation*}\begin{split}
\varepsilon^2
\frac{d}{d t} \int_{\R^2} g_1 \phi \, dy&
\leq \,C \ve^2 \left( \ve^2 (1+R^2) \log R + \ve^\nu |\log R|^2 \right)
\int_{\R^2} g_1 \phi \, dy \\
&+\, \| E \,  U_0^{-1/2} \|_{L^2 (\R^2)} \left( \int_{\R^2} g_1 \phi \, dy \right)^{1/2}\\
&\leq\,  C \ve^4 \, R^2 \, \log R
\int_{\R^2} g_1 \phi \, dy + \| E \,  U_0^{-1/2} \|_{L^2 (\R^2)} \left( \int_{\R^2} g_1 \phi \, dy \right)^{1/2}.
\end{split} \end{equation*}
Let us set  $$\alpha(t)\, =\,  \left (\int_{\R^2} g_1 \phi \, dy \right)^{1\over 2}, \quad  A_\ve = C\,  \ve^2 \, R^2 \,  \log R  .  $$
Then we find
\[
\varepsilon^2 \frac{d \alpha}{d t }
\, \leq\,  \ve^2 A_\ve \alpha + \| E (\cdot,t)\,  U_0^{-1/2} \|_{L^2}.
\]
Since $R$ satisfies \eqref{defR} we have that $A_\ve =o(1)$ as $\ve \to 0$.
Gronwall's inequality then yields
\[
\alpha(t) \,\leq \,  C\,\ve^{-2} \,
\sup_{0\le t \leq T}  \| E (\cdot,t)\,  U_0^{-1/2} \|_{L^2}.
\]
This inequality and  Lemma \ref{passl} yield \eqref{pass1}.
\end{proof}

\medskip
We consider a class of functions $E(y,t)$ such that  for a number $0<\beta<1$ we have
\be\label{Enu}
\|E\|_{3+\beta}  :=   \sup_{(y,t)\in \R^2\times [0,T]} | (1+|y|)^{3+\beta} E(y,t) \| < + \infty .
\ee
We observe that
$$
   \sup_{t \in [0,T]} \, \| E (\cdot,t)\,  U_0^{-1/2} \|_{L^2 (\R^2) } \ \le \  C\,\|E\|_{3+\beta} ,
$$
hence Lemma \ref{lin} is applicable.

\begin{lemma}\label{lin1} Under the assumptions of Lemma \ref{lin}, given an arbitrarily small $\sigma >0$ we have that for some $0<\alpha<1$ and  all small $\ve$,
%$a$, $a_*$ satisfying $\equ{nu}$-$\equ{Enu}$,  $R$ given by \eqref{defR}, all sufficiently small $\ve$
%and any solution $\phi$ of $\equ{innerL}$ with $$\sup_{t\in [0,T]} \|U_0^{-\frac 12 } \phi(\cdot , t)\|_{L^2(\R^2)} \ < \ +\infty$$
%we have

\be\label{pass2}
\begin{aligned}
& |\psi(y,t)| \, +\, (1+|y|) |\nn_y \psi (y,t)|  + (1+|y|)^{1+\alpha}  [\nn \psi(\cdot,t) ]_\alpha (y) \\ \, & \le \, \frac {\ve^{-2-\sigma} }{1+|y|}   \, \|E\|_{3+\beta} \foral (y,t) \in \R^2\times[0,T].
\end{aligned}
\ee
where $\psi(y,t)$ is given by $\equ{psi1}$ where $\phi(y,t)$ is a solution of $\equ{innerL}$-$\equ{orto4}$ satisfying  $\equ{aaa}$.
\end{lemma}

\begin{proof}

We have that $\phi(y,t)$ satisfies the transport equation
\be\label{innerL1}\left \{
\begin{aligned}
\varepsilon^2 \phi_t +  \nabla_y^\perp( \Gamma_0 + a_* +a)\cdot  \nabla_y  &\phi + \, \ttt E(y,t)\, = \, 0  \inn \R^2 \times (0,T),\\
%\label{innerInitial}
&\quad\quad\ \phi(\cdot,0)\,  =\, 0 \inn \R^2
\end{aligned}\right.
\ee
where
$$
\ttt E(y,t) \,= \,  E(y,t)
- \nabla_y^\perp( \Gamma_0 + a_* +a)\cdot  \nabla_y ( f'( \Gamma_0 +a_*) \psi ).
$$
Let us fix a number $p$ with $2<p< \frac 2{1-\beta}$. Then we have that
$$
\sup_{t\in [0,T]}  \| U_0^{\frac 1p -1} E (\cdot, t)\|_{L^p(\R^2)} \ \le \ C \,\|E\|_{3+\beta}.
$$
We claim that
\begin{equation}\label{uff2}
\sup_{t\in [0,T]} \| U_0^{{1\over p}-1} \phi(\cdot, t) \|_{L^p (\R^2 )} \leq C \ve^{-4} \,  |\log \ve |^{1\over 2} \,\sup_{t\in [0,T]}  \| U_0^{-{1\over 2}} E(\cdot, t)  \|_{L^2 (\R^2) } .
\end{equation}
Let us postpone the proof of \eqref{uff2}. Interpolating as in the proof of
Lemma \ref{interpolacion} we see that for
some $q= q(\mu)>2$ we have
$$
\| U_0^{{1\over q}-1} \phi(\cdot, t)  \|_{L^q (\R^2)} \leq \| U_0^{-{1\over 2}} \phi(\cdot, t)  \|_{L^2 (\R^2) }^{1-\mu} \, \| U_0^{{1\over p}-1} \phi(\cdot, t)  \|_{L^p (\R^2)}^{\mu} .
$$
Using estimates \eqref{uff2} and \eqref{pass1} in Lemma \ref{lin}, we conclude that
$$
\sup_{t\in [0,T]} \| U_0^{{1\over q}-1} \phi \|_{L^q (\R^2)} \,\leq \, \ve^{-2-\frac \sigma 2 } \, |\log \ve |^{1\over 2} \sup_{t\in [0,T]} \| U_0^{-{1\over 2}} E \|_{L^2 (\R^2)} ,
$$
and then \equ{pass2} directly follows from \eqref{gradtotal}.
Finally, let us prove estimate
%\medskip
%\noindent
%{\it Proof of \eqref{uff2}.}\ \
\eqref{uff2}.  Since $\phi$ solves equation\eqref{innerL1}, Lemma \ref{transport1-new} applies to yield
\begin{equation}\label{uff0}
\| U_0^{{1\over p}-1} \phi \|_{L^p (\R^2) } \leq C \ve^{-2} \| U_0^{{1\over p}-1} \ttt E \|_{L^p (\R^2) }.
\end{equation}
Let us estimate this weighted $L^p$ norm for the second term in $\ttt E$. We have
$$
\big |\nabla_y^\perp( \Gamma_0 + a_* +a)\cdot  \nabla_y ( f'( \Gamma_0 +a_*) \psi )\big |\ \le \  C\, \big [ \frac 1{1+ |y|^5} |\nn \psi| + \frac 1{1+ |y|^6}|\psi|\, \big ].
$$
Since
$$
\left \|\, U_0^{{1\over p}-1}  \frac {|\nn \psi|}{1+ |y|^5} \, \right \|_{L^p (\R^2) }\, \leq\, C \left \| U_0^{-{1\over 2} + {1\over 2p} } |\nabla \psi | \, \right \|_{L^p (\R^2) },
$$
using  \eqref{gradp}  we get
$$
\left \| U_0^{{1\over p}-1}  \frac {|\nn \psi|}{1+ |y|^5}\, \right  \|_{L^p (\R^2) }\,  \leq\,  C \,
\left \| U_0^{-{1\over 2}} \phi  \right \|_{L^2 (\R^2) }
$$
and from \eqref{cota1psi},
$$
\left \| U_0^{{1\over p}-1}  \frac {| \psi|}{1+ |y|^6} \, \right \|_{L^p (\R^2) } \,\leq\, C \left
\| U_0^{-{1\over 2}} \phi  \right \|_{L^2 (\R^2) }.
$$
Combining the above estimates, we conclude
$$
\sup_{t\in [0,T]}  \| U_0^{\frac 1p -1} \ttt E (\cdot, t)\|_{L^p(\R^2)} \ \le C \left( \| E \|_{3+\beta} +
 \sup_{t\in [0,T]} \| U_0^{-{1\over 2}} \phi | \|_{L^2 (\R^2) } \right).
$$
From this last estimate, together with \eqref{uff0} and
the result of Lemma \ref{lin}, we get \eqref{uff2}. The proof is concluded.
\end{proof}

As a consequence of the above result we can also get an $L^\infty$-weighted estimate for $\phi$.
\begin{corollary} \label{co71} Under the assumptions of Lemma \ref{lin1}, we also have the estimate
\be \label{pass3}
|\phi(y,t)| \ \le \  C\, \Big[  \frac {\ve^{-2}}  {1+|y|^{3+\beta}}  +  \frac {\ve^{-4-\sigma}}  {1+|y|^{7}}   \Big ]\,  \|E\|_{3+\beta}.
\ee
\end{corollary}
\begin{proof}
Recall that $\phi(y,t)$ solves \equ{innerL1}. From Lemma \ref{lin1}, we get
$$
|\ttt E (y,t) | \le C \Big[  \frac   1{1+|y|^{3+\beta}}  +  \frac {\ve^{-2-\sigma}}  {1+|y|^{7}}   \Big ]\,  \|E\|_{3+\beta}.
$$
Estimate \eqref{pass3} then follows as a direct application of Lemma \ref{transport1-new} for $p=+\infty$.
\end{proof}

\subsection{Estimates for a projected problem}
Here we consider the ``projected version'' of Problem  \equ{innerL},
\be\label{innerL2}\left \{
\begin{aligned}
%\label{innerL}
\varepsilon^2 \phi_t +  \nabla_y^\perp( \Gamma_0 + a_* +a)\cdot  \nabla_y ( \phi & - f'( \Gamma_0 +a_*) \psi )\, + \, E(y,t)
\\ &  = \, \sum_{l=0}^3 c_{l}(t) \mathcal  Z_{1l}(y)  \inn \R^2 \times (0,T)\\
% \\ &\quad\ + \, E(y,t)\, = \, 0  \inn \R^2 \times (0,T),\\
%\label{innerInitial}
&\qquad\quad\phi(\cdot,0)\,  =\, 0 \inn \R^2
\end{aligned}\right.
\ee
under the same assumptions on $a$ and $a_*$ in \equ{innerL}.
Here $\psi$ is given by \equ{psi1} and $\phi$ satisfies the orthogonality conditions \equ{orto4}.
We recall, from \equ{Z1l}, \equ{Z1l2},
\begin{align*}
{\mathcal Z}_{10}(y) = U_0(y),& \quad
{\mathcal Z}_{11}(y)  = \pp_{y_1} U_0(y),\\
{\mathcal Z}_{12}(y) = \pp_{y_2} U_0(y), &\quad
 {\mathcal Z}_{13}(y)  = 2U_0(y) + \nn_y U_0(y)\cdot y .
\end{align*}

The functions $c_l(t)$ are precisely those compatible with relations \equ{orto4}, that is
 \begin{align}
\gamma_l c_l(t)\, &=  \,  \int_{\R^2}  \big [ E(\cdot ,t)+
\nabla_y^\perp( \Gamma_0 + a_* + a)\cdot  \nabla_y ( \phi - f'(\Gamma_0 + a_*) \psi )\big ](\cdot, t) \, {\bf Z}_\ell\, dy, \nonumber \\  & \qquad \ell =0,1,2,\label{cll}  \\
 \gamma_3c_3(t)\, &=\, - \ve^2   \int_{\R^2} (\pp_t b_{3}) \phi \, dy  - \sum_{\ell=0}^2 c_\ell (t) \int_{\R^2} b_3 \mathcal Z_{1\ell} \, dy \label{c33} \\ &\ \ +   \int_{\R^2}  \big [ E(\cdot ,t) +
\nabla_y^\perp( \Gamma_0 + a_* + a)\cdot  \nabla_y ( \phi - f'(\Gamma_0 + a_*) \psi )\big ](\cdot, t) \, {\bf Z}_3\, dy \, .
\nonumber \end{align}
where
$$
\gamma_\ell\, =\, \int_{\R^2} \mathcal Z_{1\ell} {\bf Z}_\ell \, dy, \quad l=0,1,2,3.\\
$$
Our purpose is to give  alternative expressions for the functions $c_l(t)$ that do not involve derivatives of $\phi$, always assuming that $\phi$ has sufficient decay in the $y$-variable. We consider the case in which for some $0<\beta<1$ we have a solution of \equ{innerL2} with
\be
\|E\|_{3+\beta}\ <\ +\infty, \quad   \|\phi\|_{3+\beta} <+\infty.
\label{EEnu}\ee
where this norm was defined in \equ{Enu}.

It is useful to notice that thanks to Corollary \ref{co71} we have the validity of the pointwise estimate
$$
|\phi(y,t)| \ \le \  C\, \Big[  \frac {\ve^{-2}}  {1+|y|^{3+\beta}}  +  \frac {\ve^{-4-\sigma}}  {1+|y|^{7}}   \Big ]\, \mathcal M
$$
where
$$
\mathcal M\, :=\,
 \|E\|_{3+\beta} +  \sum_{\ell =0}^3 \| c_\ell \|_{L^\infty (0,T)} \,.
$$
An observation that will later be useful on the behavior of $\psi$ is the following. Since $\psi$ satisfies
$$
-\Delta \psi = \phi \inn \R^2 \times [0,T], \quad \psi(y,t)\to 0 \ass |y|\to +\infty,
$$
 then decomposing $(\psi,\phi)$ in Fourier series as
 \be\label{fou1}
 \psi(y,t)  =    \sum_{k=0}^\infty \psi_k(r,t) e^{ ik\theta} , \quad \phi(y,t)  =    \sum_{k=0}^\infty \phi_k(r,t) e^{ ik\theta}
\ee
we have, using the variation of parameters formula for the corresponding ODEs for modes $k=\pm 1$,
$$
\psi_{\pm 1} (r,t)  =   \frac 1r  \int_0^r \rho\, {d\rho}\int_\rho ^\infty \phi_{\pm 1} (s,t) \, ds
$$
\be \label{72}
\psi_{\pm 1} (r,t)  =   \frac {A_{\pm 1}(t)} r  +   \mathcal M \big (  r^{-5}  O( \ve^{-4-\sigma} )   +    r^{-1-\beta} O(\ve^{-2}) \big )
\ee
for certain numbers $A_{\pm 1}(t)$.

\medskip
At this point we will be more explicit in the choice of the function $a_*(y,t)$ satisfying \equ{nu} and $b_3(y,t)$ in the definition of $\bf Z_3$.
We let
$$
{\bf B} (y,t) :=    \Delta_y(\Gamma_0 + a_*)  +  f(\Gamma_0 + a_*)
$$
and assume that for some $\nu>0$,
\be
|{\bf B} (y,t)|  +  (1+|y|)  |\nn_y {\bf B} (y,t)|  \ \le\  \ve^{2+ \nu} (1+|y|)^{-1-\nu}  \foral (y,t)\in \R^2\times [0,T].
\label{uuu} \ee
Let $Z_3(y) =  \frac {|y|^2 -1} {1+ |y|^2} $. Then for a certain function $G(s)$ we can write
$
Z_3(y) =     G( \Gamma_0(y) ).
$
$b_3$ will be chosen in such a way that the relation   $$G(\Gamma_0 + a_*) \approx  {\bf Z}_3 :=  Z_3+ b_3 $$holds, which we assume by setting
$$
b_3(y,t)\, := \,  G'(\Gamma_0(y))\, a_*(y,t)
$$
We observe that $G'(\Gamma_0(y)) \sim |y|^{-2}$ for large $|y|$, hence   $$b_3(y,t)= O(\ve^2), \quad  \nn_yb_3(y,t)= O(\ve^2 |y|^{-1})$$
and  $${\bf Z}_3(y,t) =  G(\Gamma_0(y) + a_*(y,t)\, )  + O(\ve^4|y|^2 ) ,  $$
$$\nn_y {\bf Z}_3(y,t) =  \nn_y G(\Gamma_0(y) + a_*(y,t)\, )  + O(\ve^4|y| ) . $$
After testing Equation \equ{innerL2} against the functions  ${\bf Z}_\ell$ and integrate in space variable we arrive at the expressions
 \equ{cll}-\equ{c33} for the functions $c_l(t)$.
To estimate $c_3(t)$ we we integrate by parts the last quantity in \equ{c33}  and get
\begin{align*}
 \int_{\R^2}
\nabla_y^\perp( \Gamma_0 + a_* + a)\cdot  \nabla_y ( \phi -f'(\Gamma_0 + a_*) \psi )(\cdot, t) \, {\bf Z}_3 \, dy \, .
\\
 = \ - \int _{\R^2}  ( \phi  - f'(\Gamma_0 + a_*) \psi )(\cdot, t) \nn_y {\bf Z}_3\cdot \nabla_y^\perp( \Gamma_0 + a_* + a).
\end{align*}
We have that
\begin{align*}
\nn_y {\bf Z}_3\cdot \nabla_y^\perp( \Gamma_0 + a_* + a) & =
 \nn_y {\bf Z}_3\cdot \nabla_y^\perp a  +   O(\ve^4 |y| ) \nabla_y^\perp( \Gamma_0 + a_* + a)\\   & = O(\ve^4) + O(\ve^{2+\nu} |y|^{-3}) .
\end{align*}
We conclude that
$$
\gamma _3c_3( t)   =   \int_{\R^2} E  (\cdot, t) \, {\bf Z}_3 \, dy \,  +  O( \ve^{2+\nu}) \|\phi(\cdot , t) U_0^{-\frac 12} \|_{L^2(\R^2)} + O(\ve^2) \sum_{\ell=0}^2  |c_\ell (t)|
$$
In a similar way, we obtain, for $l=0$, we compute
\begin{align*}
&\int_{\R^2}  \big [
\nabla_y^\perp( \Gamma_0 + a_* + a)\cdot  \nabla_y ( \phi- f'(\Gamma_0 + a_*) \psi )\big ](\cdot, t) \,  {\bf Z}_0\, dy \\
& = \ -  \int_{\R^2}  ( \phi   + f'(\Gamma_0 + a_*) \psi ) \nabla_y^\perp(\Gamma_0 +  a_* + a)\cdot  \nabla_y 1  \\
& = 0 .
\end{align*}
Hence
$$
\gamma _0c_0( t)   =   \int_{\R^2} E  (\cdot, t) \, {\bf Z}_3 \, dy \, .
$$
On the other hand, for $l=1$ we have
\begin{align*}
&\int_{B_{5R}}  \big [
\nabla_y^\perp( \Gamma_0 + a_* + a)\cdot  \nabla_y ( \phi - f'(\Gamma_0 + a_*) \psi )\big ](\cdot, t) \,  {\bf Z}_1\, dy \\
& = \   \int_{{B_{5R}}}  ( \Delta\psi    + f'(\Gamma_0 + a_*) \psi ) \nabla_y^\perp(\Gamma_0 +  a_* + a)\cdot e_1\, dy   \\
& + \,  \int_{\pp B_{5R}}  ( \phi - f'(\Gamma_0 ) \psi ) \nabla_y^\perp \Gamma_0 \cdot \nu(y) y_1\, d\sigma  \\
& =\    \int_{{B_{5R}}}  ( \Delta\psi    + f'(\Gamma_0 + a_*) \psi )  \pp_{y_2}(\Gamma_0 +  a_* )\, dy  \\
&  +  \,  O(\ve^{2+\nu}) \|\phi(\cdot, t) U_0^{-\frac 12} \|_{L^2(\R^2)}
\end{align*}
Now, we see that
$$
\int_{{B_{5R}}} \Delta\psi \pp_{y_2}(\Gamma_0 +  a_* )\, dy =  \int_{{B_{5R}}}\psi  \Delta \pp_{y_2}(\Gamma_0 +  a_* )\, dy  +
\int_{\pp B_{5R}}  ( \pp_\nu \psi  \pp_{y_2}\Gamma_0 -  \psi \pp_\nu \pp_{y_2}\Gamma_0)\, d\sigma
$$
To carry out this analysis, we observe that the integral
$$
I:=  \int_{\pp B_{5R}}  ( \pp_\nu \psi  \pp_{y_2}\Gamma_0 -  \psi \pp_\nu \pp_{y_2}\Gamma_0)\, d\sigma
$$
has sufficient smallness as $\ve\to 0$. Indeed, using the expansion \equ{fou1}  and estimate \equ{72} we obtain that
$$
I \, =    \,  \sum_{k=-1,+1}  \int_{\pp B_{5R}}  ( \pp_r \psi_{k}(r)e^{ik\theta} \pp_{y_2}\Gamma_0 -  \psi_{k}(r)e^{ik\theta}\pp_r \pp_{y_2}\Gamma_0)\, d\sigma =
 \mathcal M  O(\ve^{\frac \beta 2} ) .
$$
Using assumption \equ{uuu}
we find that
\begin{align*}
& \int_{{B_{5R}}}  ( \Delta\psi    + f'(\Gamma_0 + a_*) \psi )  \pp_{y_2}(\Gamma_0 +  a_* )\, dy  \\
= & \int_{{B_{5R}}}  \pp_{y_2} ( \Delta (\Gamma_0 + a_*)   + f(\Gamma_0 + a_*)  )\, \psi \, dy  + I  \\
=&
O(\ve^{2+ \nu } ) \|U_0^{-\frac 12} \phi\|_{L^2(\R^2)} +  \mathcal M  O(\ve^{\frac \beta 2} )  .
\end{align*}
We can argue in exactly the same way to estimate $c_2(t)$.
As a conclusion, we find that
$$
\gamma_\ell c_\ell (t)\ =\  \int_{\R^2 } E(\cdot ,t ) \, {\bf Z}_\ell dy +     \mathcal M  O(\ve^{\beta'} )  , \quad \ell=1,2
$$
for some uniform number $\beta'>0$.

\medskip
Combining the above estimates, we obtain the following result.

\begin{prop}\label{prop71}
 There exists   a number $\beta'>0$ such that for all sufficiently small $\ve>0$ and any functions $E(y,t)$ and $\phi(y,t)$ that satisfy  $\equ{orto4}$, $\equ{innerL2}$ and  $\equ{EEnu}$, we have that the numbers $c_\ell(t)$ define linear functionals of $E$ which satisfy the estimate
$$
\gamma _\ell c_\ell(t) =  \int_{\R^2}  E(\cdot ,t){\bf Z_\ell}\, dy \, +\, O( \ve^{\beta'})  \|E\|_{3+\beta}   .
$$
Besides $\phi$ and $\psi$ satisfy the estimates $\equ{pass1}$, $\equ{pass2}$, $\equ{pass3}$.
\end{prop}

\subsection{Some a priori estimates} \label{apriories}
We want to apply this proposition to  obtain a priori estimates to
  System \equ{bpsidef}, \equ{equinner}, \equ{equc}, \equ{equout1}, \equ{equout2}. Specifically, we want to deal with equations \equ{equinner}.
 Let us write it in the form
\be\label{rrrr}
\begin{aligned}
 \ve^2 \pp_t \bphi_j    &+  \kappa_j \nn_y^\perp \big [ \Gamma_0    +   \la {a}_{j}( \bpsi_j , \alpha_j, \psi^{out},\ttt\xi )  \big ]   \cdot
 \nn_y \bphi_j \, +\, \kappa_j  \nabla^\perp_y \bpsi_j   \cdot \nabla_y (U_0 +\la \eta_4\phi_j^*)
 \\
 &+\,  \Theta_{j,\la}  ( \bpsi_j , \alpha_j, \psi^{out};\ttt \xi)  \, =\,
 \sum_{l=0}^3 c_{lj}(t)\mathcal Z_{1l}(y)
\inn \R^2\times [0,T].
\end{aligned}\ee
where
$$
\begin{aligned}
 \Theta_{j,\la}  ( \bpsi_j , \alpha_j, \psi^{out};\ttt \xi)\ := &\  \la\,\big[ \, \ttt \EE_{j}( \bpsi_j , \alpha_j,  \psi^{out},\ttt \xi ) + \ve \eta_4\, \nabla^\perp_x \psi^{out} \cdot \nn U_0  \, \big ] \\
 &  +\,  \ve  [ -\dot {\ttt \xi}_j +  (D_{\xi_j}\nn^\perp_{\xi} K)(\xi_0+ \xi_1) [\ttt\xi] ] \cdot \nn U_0  \\ & +\,  \ve^2\sum_{l=0,3}  \dot \alpha_{lj}  \mathcal Z_{1l}  .
\end{aligned}
$$
The main observation is that the linear operator in $\bphi_j$ given by the first row of formula $\equ{rrrr}$ can  essentially be written as one of the form involved in equation \equ{innerL2}, provided that the functions $\psi_j ,\psi^{out}, \alpha_j$ are of sufficiently small order.
After scaling out $\kappa_j$ to assume it equal to $1$ we write
$$
a_* + a  := \la {a}_{j}( \bpsi_j , \alpha_j, \psi^{out},\ttt\xi )
$$
where
\be
a(y,t) \   :=  \     \la\,\eta_4\, \big ( \bpsi_j + \sum_{l=0,3} \alpha_{lj}\mathcal Z_{2l} +  \psi^{out}  + \ve\dot{\ttt\xi} \cdot y\, \big ) .
\label{ccc}\ee
We will consider parameter functions lying on a region where $\nn a = O(\ve^{2+\nu})$ for some $\nu>0$.  Recalling the definition of $a_j$
in \equ{defaj} we find  that the functions  $a(y,t)$, $a^*(y,t)$ satisfy the structure assumptions \equ{nu0}-\equ{nu}.
Moreover, carefully checking the terms involved in  \equ{defaj}, we obtain that
$$
a_*  =  \la  (\eta_4   \psi_j^*  +     \mathcal H)
$$
where \begin{align*}
&\mathcal H = \mathcal H_0  + \mathcal H_1, \quad
\psi_j^* =  \psi^*_{j0}  +  \psi^*_{j1} ,\\
&\mathcal H_0  =  O(\ve^2|y|^2), \quad    \mathcal H_1= O(\ve^4|y|^2) + O (\ve^3|\log\ve|\, |y|^{-1}),  \\
&\Delta \psi^*_{j0} +  f'(\Gamma_0)\psi^*_{j0}   +  f'(\Gamma_0)\mathcal H_0   = 0, \\
&\Delta  \mathcal H_0 = 0, \quad \Delta  \mathcal H_1 =  O(\ve^4) + O (\ve^3|\log\ve|\, |y|^{-3}).
 \end{align*}
 while, also,
\begin{align*}
\phi_j^* =   -\Delta \psi^*_{j0}  +  O(\ve^4) + O (\ve^3|\log\ve|\, |y|^{-3}) .
 \end{align*}
From the above relations, the following facts are readily obtained
$$
  U_0 +   \la \eta_4  \phi_j^* = f(\Gamma_0 + a_*)  =    +  \la\eta_4 [ O(\ve^4) + O (\ve^3|\log\ve|\, |y|^{-3})  \,].
$$
and for
$$
{\bf B} (y,t) = \Delta (\Gamma_0 + a_*) +   f(\Gamma_0 + a_*)
$$
we have
$$
|{\bf B} (y,t)|\, +\,(1+ |y|) |\nn_y {\bf B} (y,t)|\, \le \,   C\,\la [\, \ve^4 +  \ve^3|\log\ve|\, |y|^{-3} \,],
$$
We can therefore write in this setting  Equation \equ{rrrr} in the form
$$\left \{
\begin{aligned}
%\label{innerL}
\varepsilon^2 \phi_t +  \nabla_y^\perp( \Gamma_0 + a_* +a)\cdot  \nabla_y ( \phi & - f'( \Gamma_0 +a_*) \psi )\,\\ &
+ \,  \Theta_{j,\la}  +  \la[ A\cdot \nn \psi +  B  \psi ]
\\ &  = \, \sum_{l=0}^3 c_{l}(t) \mathcal  Z_{1l}(y)  \inn \R^2 \times (0,T)\\
% \\ &\quad\ + \, E(y,t)\, = \, 0  \inn \R^2 \times (0,T),\\
%\label{innerInitial}
&\qquad\quad\phi(\cdot,0)\,  =\, 0 \inn \R^2
\end{aligned}\right.
$$
where $$ |A| + (1+|y|) |B| \le     C\big [\, \ve^4 +  \ve^3|\log\ve|\, |y|^{-3} \big ]  . $$
and a priori bounds of the form
\be \label{roger1}
\gamma _\ell c_{jl}(t) =  \int_{\R^2}   \Theta_{j,\la} (\cdot ,t){\bf Z_\ell}\, dy \, +\, O( \ve^{\beta'})  \| \Theta_{j,\la} \|_{3+\beta}   .
\ee
for some $\beta'>0$ readily follow, of course provided that $a(y,t)$ satisfies the required smallness assumptions.

\medskip
Moreover, we make the following observation which is useful to obtain an improvement of the a priori estimate for $\ttt\phi_j$.
We can rewrite equation \equ{rrrr}
in the form

\be\label{rrrr1}\left \{
\begin{aligned}
& \ve^2 \pp_t \bphi_j    +  \kappa_j \nn_y^\perp \big [ \Gamma_0    +   \la {a}_{j}( \bpsi_j , \alpha_j, \psi^{out},\ttt\xi )  \big ]   \cdot
 \nn_y \bphi_j \\
 &\, +\, \kappa_j  \nabla^\perp_y \bpsi_j   \cdot \nabla_y (U_0 +\la \eta_4\phi_j^*)+\,  \la \ttt \Theta_{j}  ( \bpsi_j , \alpha_j, \psi^{out};\ttt \xi)\\ &  \, =\,
 \sum_{l=0}^3 \ttt c_{lj}(t)\mathcal Z_{1l}(y)
\inn \R^2\times [0,T], \\
&\qquad \ttt \phi_j(y, 0) = 0 \inn \R^2.
\end{aligned} \right .\ee
for certain numbers $\ttt c_{lj}(t)$ with $\ttt \Theta_{j,\la}$  consist of some pieces of $\Theta_{j,\la}$ taken away. More precisely,

\be\label{tetatilde}
\begin{aligned}
&\ttt \Theta_{j}  ( \bpsi_j , \alpha_j, \psi^{out};\ttt \xi)(y,t)\ := \   \ttt \EE_{j}( \bpsi_j , \alpha_j,  \psi^{out},\ttt \xi )(y,t) \\   &+\,  \ve \eta_4\, [\nabla^\perp_x \psi^{out}(\xi_j(t) +\ve y,t) - \nabla^\perp_x \psi^{out}(\xi_j(t),t)]  \cdot \nn U_0 (y) \, .
\end{aligned}
\ee
where $\ttt \EE_{j}$ is defined in \equ{defaj}.
Applying Lemma \ref{lin} and  Corollary \ref{co71} we find that
if $a(y,t)$ has sufficient smallness then $\ttt\phi_j(y,t)$ solving \equ{rrrr1}
satisfies the bounds

\be \label{pass11}
\sup_{t\in [0,T]} \|U_0^{-\frac 12 } \phi(\cdot , t)\|_{L^2(\R^2)} \ \le \ C\,\ve^{-2} \, |\log \ve |^{1\over 2} \, \|\ttt \Theta_{j}\|_{3+\beta}.
\ee
and
\be \label{pass31}
|\phi(y,t)| \ \le \  C\, \Big[  \frac {\ve^{-2}}  {1+|y|^{3+\beta}}  +  \frac {\ve^{-4-\sigma}}  {1+|y|^{7}}   \Big ]\,  \|\ttt \Theta_{j}\|_{3+\beta}.
\ee

\medskip

\section{Reformulation and a priori bounds}

We consider in this section System \equ{bpsidef}, \equ{equinner}, \equ{equc}, \equ{equout1}, \equ{equout2}, which we will
set up as a fixed problem of the form  \equ{equp} in an appropriate Banach space  $X$.

\medskip
\subsection{The space $X$.}
We begin by defining an appropriate norm for the functions
$\bphi_j(y,t)$
in agreement with the estimates found in the previous section.

\medskip
Let us fix a  small number  $0< \beta<1$.
For and arbitrary functions $\phi(y,t)$  we define the inner norm
\begin{align*}
 \| \phi\|_{i}\, :=  &\,  \sup_{t\in [0,T]} \| \phi(\cdot,t) U_0^{-\frac 12} \|_{L^2(\R^2)}\\
 &+   \sup_{(y,t)\in \R^2\times [0,T]} |\, (1+|y|)^{3+\beta}\min \{ 1, \ve^{2+ \frac \beta 4 } (1+|y|)^{4-\beta} \} \phi(y,t)   \, |
 %\\ \quad \|\psi \|_{i2} \, := &\, \sup_{(y,t)\in \R^2\times [0,T]} |\, (1+|y|)  \psi(y,t)  \, |\, +\,   \sup_{(y,t)\in \R^2\times [0,T]} |\, (1+|y|)^2  \nn_y\psi(y,t)   \, |\, .\\  &+\,     (1+|y|)^{2+\alpha}  [\nn_y\psi(\cdot ,t)]_\mu (y)
\end{align*}
For the outer functions $\ttt\psi^{out}$, $\ttt\phi^{out}$ we consider the following norms
for functions $\phi(x,t)$, $\psi(x,t)$ defined in $\Omega\times [0,T]$.
\begin{align*}
 \| \phi\|_{o1}\, := &\,   \| \phi \|_ {L^\infty (\Omega\times [0,T])} ,  \\
\|\psi \|_{o2} \, :=&\,  \| \psi \|_ {L^\infty (\Omega\times [0,T])} + \|\nn_x \psi \|_ {L^\infty (\Omega\times [0,T])}
\, .\end{align*}
We consider the space $X$ of all continuous functions
$\vec { p}  = (\bphi, \alpha, \phi^{out},\psi^{out} , \ttt\xi) $
such that
$\nn_y\bpsi(y,t),\, \nn_x \psi^{out}(x,t),\, \dot \alpha(t), \, \dot {\ttt\xi}(t),  $ exist and are continuous and such that
$$
\|\vec p \,\| _X \,:=\,     \| \phi^{out}\|_{o1}+  \|\psi^{out}\|_{o2} + \sum_{j=1}^N \big( \| \bphi_j\|_{i} + \|\alpha_j\|_{C^1[0,T]} + \|\ttt\xi_j \|_{C^1[0,T]}\big)    \, <\, +\infty .
$$
We define the set $\OO$ as a ``deformed ball'' centered at $\vec p = \vec 0$. We fix an arbitrarily small number $\sigma>0$ and let $\OO $ be the set of all functions
$\vec { p}  = (\bphi, \alpha, \ttt\xi, \phi^{out},\psi^{out})\in X $ such that
\be\label{OO} \left\{ \begin{aligned}
& \sum_{j=1}^N  \| \bphi_j\|_{i}
\ < \ \ve^{3- 3\beta} ,\\
&\sum_{j=1}^N \|\alpha_j\|_{C^1[0,T]}  \, < \, \ve^{3- 3\beta} ,\quad  \sum_{j=1}^N \|\ttt\xi_j\|_{C^1[0,T]}  \, < \, \ve^{4-3\beta}  ,\\
&\|\phi^{out}\|_{o1}  \ < \ \ve^{4- 3\beta} ,\quad   \|\psi^{out}\|_{o2}  \ < \ \ve^{4- 3\beta }.
\end{aligned}
\right.
\ee

\subsection{Fixed point formulation}
Let us express System  \equ{equinner}, \equ{equc}, \equ{equout1}, \equ{equout2} in the fixed point form \equ{equp}
for a suitable operator $ {\mathcal F} (\cdot , \la )$, in a region of
the form $\equ{OO}$.

\medskip
We start with \equ{equinner}.
For a given function $a(y,t)$ with $\Delta_y a \in L^\infty (\R^2\times [0,T])$  let us consider the transport operator
$$
\TT_j (a)[\phi] :=  \ve^2\phi_t +  \kappa_j \nn_y^\perp ( \Gamma_0 + a) \cdot \nn_y \phi ,
$$
and for a bounded function $E(y,t)$ the linear equation
\begin{align*}
\TT_j (a)[\phi] + E &= 0 \inn \R^2\times [0,T]\\
\phi(\cdot,0) &= 0 \inn \R^2  .
\end{align*}
We call
$
\phi =  \mathcal \TT^{-1}_j (a) [ E] $
the unique solution of this problem, through the representation formula \equ{phi}, which defines a linear operator of $E$.
Let us write the operator ${\EE }_{j,\la}$ in \equ{Ejla} in the form
 \begin{align*}
& {\EE }_{j,\la}  (  \bpsi_j  , \alpha_j,  \psi^{out};\ttt \xi)\\  = & \TT_j ( \,\la {a}_{j}( \bpsi_j , \alpha_j, \psi^{out},\ttt\xi \,) )[\ttt \phi_j]  +   \kappa_j  \nabla^\perp_y \bpsi_j   \cdot \nabla_y (U_0 +\la \eta_4\phi_j^*)+
\la\ttt\Theta_{j}  ( \bpsi_j , \alpha_j, \psi^{out};\ttt \xi).
\end{align*}
where, as usual, $\bpsi_j = \mathcal N(\bphi_j)$, and $\ttt\Theta_{j}$ is given by \equ{tetatilde}.
We reformulate equations \equ{equinner} as
\be\label{fixed1} \boxed{
\begin{aligned}
\bphi =
\mathcal F_{\la}^{in} (\bphi , \alpha, \psi^{out},\ttt\xi ),\\
%\bpsi =&
%\mathcal F^{2,in} ( \bphi )
\end{aligned} }
\ee
where
\begin{align}
&(\mathcal F_{\la}^{in})_j (\bphi , \alpha,   \psi^{out},\ttt\xi )\  : =  \label{Fin} \\ \  \mathcal \TT^{-1}_j\big (
\,\la {a}_{j}( \bpsi_j , \alpha_j, \psi^{out},\ttt\xi \,) \,\big )\, \big [  & \kappa_j  \nabla^\perp_y \bpsi_j   \cdot \nabla_y (U_0 +\la \eta_4\phi_j^*) + \la\ttt\Theta_{j}  ( \bpsi_j , \alpha_j,  \psi^{out};\ttt \xi) - \sum_{l=0}^3 c_{lj} \mathcal Z_{1l}    \big ].\nonumber
%\\\mathcal N ( \bphi_j ) \ : =&\   \frac 1{2\pi} \log \frac 1{|\cdot|}* \bphi_j\, .
\end{align}

%$$
%\mathcal F^{2,in} (\ttt \phi)= (
% \mathcal N  ( \bphi_1 ),\ldots , \mathcal N  ( \bphi_N )),
%$$

We reformulate the outer equations \equ{equout1}-\equ{equout2} in  a similar way.
For a given function $e(x,t)$ with $\Delta_x e \in L^\infty (\Omega\times [0,T])$  and $e= 0$ on $\pp\Omega\times [0,T]$
let us consider the transport operator
$$
\TT (e)[\phi] :=  \phi_t +   \nn_y^\perp (\Psi_* + e)  \cdot \nn_x \phi ,
$$
and for a bounded function $E(x,t)$, the linear equation
\begin{align*}
\TT (e)[\phi] + E &= 0 \inn \Omega\times [0,T]\\
\phi(\cdot,0) &= 0 \inn \Omega  .
\end{align*}
We call
$
\phi =  \TT^{-1} (e) [ E] $
the unique solution of this problem, through the representation formula \equ{phi3}.
We write \equ{equout1}-\equ{equout2} in the form
\be\label{fixed2}  \begin{boxed}{
\begin{aligned}
\phi^{out} = & \mathcal F_{1\la}^{out}  (\bphi ,  \alpha, \psi^{out},\ttt\xi )\\
\psi^{out} = & \mathcal F_{2\la}^{out}  ( \bphi , \alpha, \phi^{out},\ttt\xi )
\end{aligned}}\end{boxed}
\ee
where
$$
\mathcal F_{1\la}^{out}  (\bphi,  \alpha, \psi^{out},\ttt\xi ) :=
 \TT^{-1}(\la\sum_{j=1}^N \eta_{2j}  \psi_j + \la  \psi^{out} ) \left[ \la \ttt \EE_{1}^{out}( \ttt \phi ,  \alpha,     \psi^{out};\ttt \xi) \right]
$$
with $ \ttt \EE_{1}^{out}$ given by \equ{EE1}
and
$$
\mathcal F_{2\la}^{out}  ( \bphi , \alpha, \phi^{out},\ttt\xi ):= (-\Delta )^{-1} \left[ \la\,  \phi^{out}
      +  \la\, \sum_{j=1}^N \kappa_j ( \psi_j \Delta_x\eta_{2j}  + 2 \nn_x \eta_{2j} \cdot \nn_x   \psi_j) + \la\,E_{2*}^{out}(\ttt\xi) \right]
$$
where $\psi = (-\Delta)^{-1} h$ is the unique solution of
$$-\Delta \psi = h \inn \Omega , \quad \psi = 0 \inn \partial \Omega.
$$
and $E_{2*}^{out}(\ttt\xi)$ is given by \equ{E2*}.

Equations \equ{equc} can be better described using expressions \equ{roger1}.
Indeed we have that for $a$ as in \equ{ccc} the bound $\nn a =  O(\ve^{2+\nu})$ holds.
Then we find that the equations $c_{\ell j}=0$ read
\be\label{pico71} \left \{
\begin{aligned}
\dot {\ttt \xi}_j  & =  (D_{\xi_j}\nn^\perp_{\xi} K)(\xi_0+ \xi_1) [\ttt\xi]  + \la \ve^{-1}(\mathcal G_{1} )_j( \bphi , \alpha, \psi^{out};\ttt \xi),\\
\dot {\alpha}_j    &= \la \ve^{-2}\mathcal (G_{0})_j
 ( \bphi , \alpha, \psi^{out};\ttt \xi),  \quad j=1,\ldots, N, \\
\xi(0) &=0,  \quad  \alpha(0)=0 .
\end{aligned}\right.
\ee
As usual we write $$\alpha_j = (\alpha_{0j}, \alpha_{3j}),\quad  \xi_j= (\xi_{j1},\xi_{j2}),
\quad
\alpha = (\alpha_1,\ldots,\alpha_N), \quad \xi = (\xi_1,\ldots,\xi_N)$$
and denote
$$(\mathcal G_0)_j = (\mathcal G_{0j}, \mathcal G_{3j}),\quad  (\mathcal G_1)_j = (\mathcal G_{1j}, \mathcal G_{2j}),
$$
where
$$
\begin{aligned}
 \mathcal G _{j\ell}  ( \bphi , \alpha, \psi^{out},\ttt \xi)(t)\ := &\  \gamma_\ell \int_{\R^2} \big[ \, \ttt \EE_{j}( \bpsi_j , \alpha_j,  \psi^{out},\ttt \xi ) + \ve \eta_4\, \nabla^\perp_x \psi^{out} \cdot \nn U_0  \, \big ] {\bf Z}_\ell\, dy. \\
\end{aligned}
$$
Equations \equ{pico71} can be written in fixed point form as
\be\label{sisalfa} \boxed{
\begin{aligned}
\ttt \xi (t)  &=  \mathcal F_{1\la} ( \bphi , \alpha,  \psi^{out},\ttt \xi ) =  \int_0^t \big(  B(s)[\ttt\xi]    + \la \mathcal G_{1} ( \bphi , \alpha, \psi^{out},\ttt \xi)\, \big)\, ds \\
\alpha(t)& = \mathcal F_{0\la} ( \bphi , \alpha, \psi^{out},\ttt \xi)=  \int_0^t  \la \mathcal G_{0}
 ( \bphi , \alpha, \psi^{out},\ttt \xi,\la) \, ds
\end{aligned}
}\ee
where
$$
\big( B (t) [\ttt \xi] \big)_j =   (D_{\xi_j}\nn^\perp_{\xi} K)(\xi_0+ \xi_1) [\ttt\xi],
$$
$$
\mathcal G_0 = ((\mathcal G_0)_1,\ldots, (\mathcal G_0)_N), \quad \mathcal G_1 = ((\mathcal G_1)_1,\ldots, (\mathcal G_1)_N).$$

\medskip
Summarizing,
our problem is equivalent to
system  \equ{fixed1}-\equ{fixed2}-\equ{sisalfa},
which we will next replace by an equivalent fixed point equation in $\OO$
for which the right hand side is a
compact operator for the topology in $X$.

Let us consider the equations
$$ \boxed{
\begin{aligned}
\bphi = &
\mathcal F_{\la}^{in} (\bphi , \alpha, \psi^{out},\ttt\xi ),\\
\phi^{out} = & \tilde {\mathcal F}_{1\la}^{out}  (\bphi , \alpha, \psi^{out},\ttt\xi )  \\
\psi^{out} = & \tilde {\mathcal F}_{2\la}^{out}    (\bphi , \alpha, \psi^{out},\ttt\xi ) \\
\ttt \xi   = & \tilde {\mathcal F} _{1\la} (\bphi , \alpha, \psi^{out},\ttt\xi )  \\
\alpha = & \tilde {\mathcal F} _{0\la} (\bphi , \alpha, \psi^{out},\ttt\xi )
%\bpsi =&
%\mathcal F^{2,in} ( \bphi )
\end{aligned} }
$$
where
\begin{align*}
\tilde {\mathcal F}_{1\la}^{out}  (\bphi , \alpha, \psi^{out},\ttt\xi ) &:= \mathcal F_{1\la}^{out}   (\mathcal F_{\la}^{in}  (\bphi , \alpha, \psi^{out},\ttt\xi )  ,  \alpha, \psi^{out},\ttt\xi )\\
\tilde {\mathcal F}_{2\la}^{out}  (\bphi , \alpha, \psi^{out},\ttt\xi )&:=
\mathcal F_{2\la}^{out}    (\mathcal F_{\la}^{in} (\bphi , \alpha, \psi^{out},\ttt\xi ) , \alpha , \tilde {\mathcal F}_{1\la}^{out}  (\bphi , \alpha, \psi^{out},\ttt\xi ) , \ttt\xi ) \\
\tilde {\mathcal F} _{1\la} (\bphi , \alpha, \psi^{out},\ttt\xi )&:=
\mathcal F _{1\la} (\mathcal F_{\la}^{in} (\bphi , \alpha, \psi^{out},\ttt\xi ) , \alpha ,  \tilde {\mathcal F}_{2\la}^{out}    (\bphi , \alpha, \psi^{out},\ttt\xi )   , \ttt\xi )\\
\tilde {\mathcal F} _{0\la} (\bphi , \alpha, \psi^{out},\ttt\xi )&:=
\mathcal F _{0\la} (\mathcal F_{\la}^{in} (\bphi , \alpha, \psi^{out},\ttt\xi ) , \alpha ,  \tilde {\mathcal F}_{2\la}^{out}    (\bphi , \alpha, \psi^{out},\ttt\xi )   , \ttt\xi ).
\end{align*}
With a slight abuse of notation, we see that System
 \equ{fixed1}-\equ{fixed2}-\equ{sisalfa} in $\bar \OO$ is equivalent
to the fixed point problem
\be\label{ecc}
\vec p \ =\ \tilde {\mathcal F_\la} (\vec p) , \quad \vec p\in\bar\OO
\ee
where
\be \label{opera}
\left\{
\begin{aligned}
\tilde {\mathcal F_\la} (\vec p) \, := &\,   ( {\mathcal F}^{in}_\la(\vec p), \tilde {\mathcal F}_{0\la}(\vec p), \tilde {\mathcal F}_{1\la}(\vec p) , \tilde {\mathcal F}^{out}_{1\la}(\vec p), \tilde {\mathcal F}^{out}_{2\la}(\vec p)
), \\
\vec p\, = &\, (\bphi,\alpha, \ttt\xi, \phi^{out}, \psi^{out}) .
\end{aligned}  \right.
\ee

\begin{lemma}
The operator $\tilde {\mathcal F} : \OO\times [0,1] \to X $ given by
$\tilde {\mathcal F} ( \cdot , \la )=   \tilde {\mathcal F}_\la$  in $\equ{opera}$
is compact.
\end{lemma}

\begin{proof}
We check that each of the five operators defining $\ttt {\mathcal F }_\la(\vec p)$ is compact in $\OO $ (uniformly in $\la$).
We start with $\mathcal F^{in}_\la(\vec p)$ in \equ{Fin}.
The key fact is that
the operator
$ g = \mathcal T^{-1}_j (a)[h] $ has the property in Lemma \ref{modc2}, which states that a uniform bound in $\Delta_y a $ and a control of the modulus of continuity in $y$ of $h(y,t)$ uniformly in $t$ yields a uniform control of the modulus of  continuity of $g$  in both variables $(y,t)$.
We see in \equ{Fin} that  for a certain $C_\ve>0$ we have
$$ \|\Delta_y a_j\|_ {L^\infty(\R^2\times [0,T])} \le  C_\ve  \foral \vec p\in \bar\OO. $$
and it vanishes outside a compact set. Moreover, we have a uniform H\"older
control in space variables on the corresponding arguments $h$ for $\vec p\in \bar\OO$ as it follows from the H\"older estimates for the gradients of $\ttt\psi_j$ and $\psi^{out}$  inherited from the uniform bounds holding for
$\ttt\psi$ and $\phi^{out}$ in the definition of $\OO$ (see the argument in the proof of \equ{pass2}). Also, the numbers
$c_{lj}(t)$ have a uniform bound, thanks to \equ{roger1}.
 Uniform Lipschitz bounds hold for the remaining errors, as it follows in particular from Remark \ref{nota} for the control of the terms involving
$\nn_y \phi_j^*$.  Lemma \ref{modc2} then implies that
$\ttt{\mathcal F}^{in}_\la(\bar\OO)$ is a set of continuous functions
 $g:\R^2\times [0,T]\to \R^N$ whose restrictions to any compact set defines a uniformly bounded, equicontinuous set. Hence, any sequence $\phi_n \in \ttt{\mathcal F}^{in}_\la(\OO)$
has a subsequence $\phi_{n'}$ which is uniformly convergent on each compact set.
 Finally, we observe that $\|\phi_n\|_{4}\le C_\ve$ since the argument of
 the transport operator has this property. This implies that $\phi_{n'}$ is actually convergent in the space of continuous functions with finite
 $\| \cdot \|_{3+\beta} $-norm, since $0<\beta <1$. Hence
 $\ttt{\mathcal F}^{in}_\la(\OO)$ is precompact in this space.
 The compactness of the operator $\tilde {\mathcal F}_{1\la}^{out}$ into $C(\bar\Omega\times [0,T])$ follows directly from Arzela-Ascoli's theorem, again from the corresponding control for the transport equation and the
 uniform controls on space and time variables valid for the operator $\tilde {\mathcal F}_{1\la}^{in}$. From here the compactness for $\tilde {\mathcal F}_{2\la}^{out}$ follows in similar manner. Finally, the compactness of the operators
 $
 \tilde {\mathcal F}_{0\la}(\vec p), \tilde {\mathcal F}_{1\la}(\vec p)
 $
into $C^1([0,T])$ follows again from the equicontinuity in $t$ inherited for the
different terms involved in their definition. The proof is concluded.
\end{proof}

\subsection{Conclusion of the proof of Theorem \ref{teo1}}
The original problem has been so far reduced  to finding a solution of the fixed point problem
\equ{ecc} for $\la= 1$. To do this, we will prove that for all $\la\in [0,1]$ equation \equ{ecc} has no solution $\vec p\in \pp\OO$, at least whenever $\ve$ is chosen sufficiently small.  Let us assume that $\vec p\in \bar\OO$ satisfies
\equ{ecc} for some $\la$. We claim that actually $\vec p\in \OO$.
We  use the considerations in \S \ref{apriories}.
Using bounds \equ{OO} and Lemma \ref{interpolacion} we find that the function $a(y,t)$ in \equ{ccc} satisfies $\nn_y a = O( \ve^{3-4\beta }) = O(\ve^{1+\nu}) $ provided that $\beta$ was chosen sufficiently small. Then estimates  \equ{pass11}, \equ{pass31} apply for the coordinate $\ttt \phi$ of $\vec p$.
Examining the function
\equ{tetatilde} we quickly see that if $\beta$ was chosen sufficiently small then
$$
\|\ttt \Theta_j( \ttt\psi , \alpha, \psi^{out}, \xi) \|_{3+\beta}  \le \ve^{5- \frac 32\beta}
$$
Estimates \equ{pass11}, \equ{pass31} then yield, by definition of the inner norm,
$$
\|\ttt \phi\|_i \le \ve^{3-2\beta} \ll \ve^{3-3\beta},
$$
the latter number being that involved in the definition of $\OO$ in \equ{OO}.
Let us consider the outer equations.
Examining expression \equ{EE1} that determines the size of $\phi^{out}$,
we see that its magnitude does not exceed the order $O(\ve^{4-\beta})$. Here we have used the remote size of $\ttt \phi$ implicit in the norm
$\|\ttt \phi_j\|_i$. Indeed using the size induced in $\ttt \psi$, we find
that
$$
\| \phi^{out}\|_{o1}  +   \| \psi^{out}\|_{o2} \le \ve^{4-2\beta} \ll \ve^{4-3\beta}.
$$
Finally from the size of $\Theta_j$ we readily see that
$$\|\ttt\xi_j\|_{C^1[0,T]} +  \ve\|\alpha_j\|_{C^1[0,T]} \le \ve^{4-2\beta} \ll \ve^{4-3\beta}.
$$
As a conclusion, we get that $\vec p\in \OO$ and the claim has been proven.

\medskip
Standard degree theory applies then to yield that the degree
$\deg ( I -\ttt{ \mathcal F}(\cdot, \la), \OO, 0)$ is well-defined and it is constant in $\la\in [0,1]$. Since $\ttt {\mathcal F}(\cdot, 0)$ is a linear compact operator, this constant is actually non-zero.
Existence of a solution in $\OO$ for $\la=1$ then follows. The proof is concluded. \qed

%Let us consider a sequence $\vec p_n\in \

\bigskip\noindent
{\bf Acknowledgements:}
  We are grateful to Robert L. Jerrard for many useful discussions on the two-dimensional Euler flow, in particular about its linearization around a Liouville profile.
%We thank Professor Fanghua Lin for useful discussions.
J.~D\'avila and M.~del Pino have been supported by grants Fondecyt  1170224,  1150066 and Fondo Basal CMM, Chile.  M.~Musso has been supported by  Fondecyt grant 1160135.The  research  of J.~Wei is partially supported by NSERC, Canada.

\end{document}